\documentclass{amsart}
\usepackage{amsmath, amscd, amssymb, amsthm}
\usepackage{bbm}
\usepackage{latexsym}
\usepackage{amsfonts}
\usepackage{graphicx}
\usepackage[all,cmtip]{xy}
\usepackage[colorlinks,linkcolor=blue,breaklinks=true,urlcolor=blue,citecolor=blue,anchorcolor=blue,pagebackref]{hyperref}%
\usepackage{geometry}
\newtheorem{theorem}{Theorem}
\newtheorem{lemma}{Lemma}

\newtheorem{proposition}{Proposition}

\newtheorem{conjecture}{Conjecture}

\newcommand{\R}{{\mathbb R}}

\newcommand{\C}{{\mathbb C}}

\renewcommand*\backref[1]{}
\renewcommand*\backrefalt[4]{ \ifcase #1 \or (cited on page #2) \else (cited on pages #2) \fi}

\newcommand{\be}{\begin{equation}}
\newcommand{\ee}{\end{equation}}
\newcommand{\bea}{\begin{eqnarray}}
\newcommand{\eea}{\end{eqnarray}}

\newcommand{\E}{\mathrm{e}}
\newcommand{\I}{\mathrm{i}}
\newcommand{\tr}{\mathrm{tr}}
\newcommand{\im}{\mathrm{Im}}
\newcommand{\re}{\mathrm{Re}}

\newcommand{\vphi}{\varphi}

\newcommand{\ka}{\kappa}

\def\XXint#1#2#3{{\setbox0=\hbox{$#1{#2#3}{\int}$ }
\vcenter{\hbox{$#2#3$ }}\kern-.6\wd0}}

\begin{document}

\title[Hermitian Lie algebras with abelian
ideals of codimension two]
{Curvature property on Hermitian Lie algebras with abelian
ideals of codimension two}

\author{Shuwen Chen}
\address{Shuwen Chen. School of Mathematical Sciences, Chongqing Normal University, Chongqing 401331, China}
\email{{3153017458@qq.com}}\thanks{Zheng is the corresponding author. He is partially supported by National Natural Science Foundations of China
with the grant No.12471039 and  12141101, and is supported by the 111 Project D21024.}

\author{Fangyang Zheng}
\address{Fangyang Zheng. School of Mathematical Sciences, Chongqing Normal University, Chongqing 401331, China}
\email{20190045@cqnu.edu.cn; franciszheng@yahoo.com} \thanks{}



\begin{abstract}
Let $(\mathfrak g,J,g)$ be a unimodular Hermitian Lie algebra  containing an abelian ideal
$\mathfrak a$ of real codimension two.  We study the curvature behaviour of $g$ and show that, if $g$ has constant Chern holomorphic sectional curvature, then it must be Chern flat. We also show that, for every canonical metric connection $D_s^r$ of $g$ other than the Chern connection, if $D_s^r$ has constant holomorphic sectional curvature, then $g$ is K\"ahler flat, and in this case $\mathfrak g/\mathfrak a$ is abelian.
\end{abstract}
\maketitle

\tableofcontents

\section{Introduction}

An old conjecture in complex geometry states that any compact Hermitian manifold with constant Chern holomorphic sectional curvature must be either K\"ahler or Chern flat.  At present the conjecture is known to be true in complex dimension 2 but is open in higher dimensions in general, except in several special circumstances. The conjecture is also open for general {\em Lie-Hermitian manifolds}. Recall that a Lie-Hermitian manifold means a compact Hermitian manifold in the form  $G/L$, where $G$ is a unimodular Lie group equipped with a left-invariant complex structure and a compatible left-invariant metric, and $L \subset G$ is a discrete subgroup. For Lie-Hermitian manifolds, the conjecture is still largely open, except in some special cases. For instance, the conjecture has been confirmed when $G$ is nilpotent \cite{LiZ}, or when $G$ is solvable and its commutator is $J$-invariant \cite{HuangZ}, or when $G$ is almost abelian (meaning that it contains a normal abelian subgroup of codimension one), by the work \cite{LiZhengHSC}.

Denote by ${\mathfrak g}$ the Lie algebra of $G$. It is well-known that left-invariant metrics on $G$ are in one one correspondence with inner products on ${\mathfrak g}$, and left-invariant complex structures on $G$ are in one one correspondence with {\em complex structures} on ${\mathfrak g}$, namely, an almost complex structure $J$  on the vector space ${\mathfrak g}$ satisfying the integrability condition:
\[ [x,y] - [Jx,Jy] + J[Jx,y] + J[x,Jy]=0, \ \ \ \ \ \ \ \ \forall \ x,y\in {\mathfrak g}.
\]
Given a Lie algebra ${\mathfrak g}$, if $J$ is a complex structure on ${\mathfrak g}$ and $g=\langle , \rangle$ is an inner product on ${\mathfrak g}$ compatible with $J$ (namely, $\langle Jx, Jy\rangle = \langle x, y\rangle$ for any $x,y\in {\mathfrak g}$), then we will call $(J,g)$ a {\em Hermitian structure} on ${\mathfrak g}$, and call $({\mathfrak g},J,g)$ a {\em Hermitian Lie algebra}. Note that when $G$ has compact quotients, ${\mathfrak g}$ is necessarily unimodular, meaning that $\mbox{tr}(ad_x)=0$ for any $x\in {\mathfrak g}$. Since we are primarily concerned with compact Hermitian manifolds, we will restrict our attention to Hermitian Lie algebras that are unimodular.

Recall that a Lie algebra ${\mathfrak g}$ is said to be {\em almost abelian,} if it contains an abelian ideal of codimension 1. In recent years, the Hermitian geometry of almost abelian Lie algebras has been extensively studied, and many results were obtained. We refer the readers to \cite{FinoP, FinoP2} and references therein for more discussions. As a natural generalization to the almost abelian case, in this paper we will consider Hermitian Lie algebras which contain abelian ideals of (real) codimension 2.

\vspace{0.2cm}

Coming back to the constant holomorphic sectional curvature conjecture mentioned at the beginning, there are also analogous conjectures/questions when the Chern connection is replaced by other canonical metric connection. Let $(M^n,g)$ be a compact Hermitian manifold. Denote respectively by $\nabla$, $\nabla^c$, $\nabla^b$ the Levi-Civita, Chern, and Bismut connection of $g$. For any $r\in {\mathbb R}$, denote by
$$ D^r =\frac12(1+r)\nabla^c+\frac12(1-r)\nabla^b $$
the $r$-th Gauduchon connection of $g$, which is the line joining the Chern and Bismut connection. In particular, $D^1=\nabla^c$ and $D^{-1}=\nabla^b$. This one parameter family of connections will be referred to as the {\em Gauduchon line} from now on. More generally, one can consider the two parameter family of canonical
metric connections of $g$:
\begin{equation}\label{eq:canonical-family}
D_s^r=(1-s)D^r+s\nabla,
\qquad (r,s)\in\Omega,
\end{equation}
where $\Omega=\{(r,s)\in\R^2:s\neq1\}\cup\{(0,1)\}$. Note that when $g$ is K\"ahler, all these connections $D^r_s$ coincide, while when $g$ is not K\"ahler, they are all distinct, namely, for any two distinct points $(r,s)$ and $(r',s')$ in $\Omega$, one has $D^r_s\neq D^{r'}_{s'}$.  Following \cite{ChenZhengCanonical}, put
$t=\frac12(1-r+rs)$ and $\chi=t^2+\frac{s^2}{4}$.  The Chen--Nie
curve is $\Gamma=\{(r,s)\in\Omega:\chi=1\}$. In \cite{ChenNie}, Chen and Nie proposed the following:
\begin{conjecture}[Chen--Nie \cite{ChenNie}]
Let $(M^n,g)$ be a compact Hermitian manifold and $(r,s)\in \Omega$. Suppose that the holomorphic sectional curvature of $D^r_s$ is equal to a constant $\kappa$. If $\kappa\neq 0$ then $g$ must be K\"ahler. If $\kappa=0$ and $(r,s)\notin \Gamma $ then $g$ must be $D^r_s$-flat.
\end{conjecture}

Note that since the $\nabla$ and $\nabla^c$ are not on the Chen--Nie curve $\Gamma$,  the above statement covers both the Chern and Levi-Civita cases. On the other hand, for $\nabla^b$ and other connections on $\Gamma$, the zero constant holomorphic sectional curvature does not imply K\"ahlerness or flatness. For Hermitian Lie algebras containing abelian ideals of codimension 2, we prove the following:
\begin{theorem}\label{thm:main}
Let $(\mathfrak g,J,g)$ be a unimodular Hermitian Lie algebra containing an abelian ideal
$\mathfrak a$ of real codimension two.  Suppose that the  holomorphic sectional curvature of $D_s^r$ is constant for some $(r,s)\in\Omega$. If $(r,s)\neq (1,0)$, then $g$ is K\"ahler flat and $\mathfrak g/\mathfrak a$ is abelian. If $(r,s)=(1,0)$, in this case $D^1_0=\nabla^c$, then $g$ is Chern flat.
\end{theorem}

Note that K\"ahler flat Hermitian Lie algebras are classified by Barberis, Dotti, and Fino in \cite{BDF} (see also a slight variation in the appendix of \cite{VezzoniYangZheng}), utilizing Milnor's characterization on flat Lie algebras \cite{Milnor}. They are a special kind of 2-step solvable Lie algebras. In a recent work \cite{ZhangZ}, Zhang and Zheng proved that any Levi-Civita flat Hermitian Lie algebra is actually K\"ahler flat. Chern flat Hermitian Lie algebras are not fully classified, but for $\mathfrak g$ containing an abelian ideal $\mathfrak a$ of real codimension two, the Chern flat case is well understood by the work of \cite{GuoZ} and \cite{CaoZhengFV}.

Hermitian Lie algebras $(\mathfrak g,J,g)$ containing an abelian ideal
$\mathfrak a$ of real codimension two can be divided into two subsets: those with $J\mathfrak a = \mathfrak a$ and those with $J\mathfrak a \neq  \mathfrak a$. The first case resembles (and contains) the almost abelian case and is much easier to study. On such Hermitian Lie algebras, Guo and Zheng  \cite{GuoZ} obtained various geometric properties,   Li and Zheng  \cite{LiZhengHSC} proved the Chern (or Levi-Civita) flatness when corresponding holomorphic sectional curvature is assumed to be constant.

The second case, namely those with $J\mathfrak a \neq  \mathfrak a$, is much more challenging from the technical point of view. In \cite{CaoZhengFV, CaoZhengST}, Cao and Zheng constructed a unitary basis that is convenient to express relevant Hermitian geometric quantities, and divided them into three mutually exclusive categories: the generic type, the half-generic type, and the degenerate type. Using these calculations, they successfully confirmed the Fino-Vezzoni conjecture (which states that any compact complex manifold admitting both a balanced metric and a pluriclosed metric must also admit a K\"ahler metric) for such Hermitian Lie algebras. In this paper, we will follow the notations and constructions in \cite{CaoZhengFV}.

The proof is organized according to whether the quotient
$\mathfrak g/\mathfrak a$ is abelian or not.  Section~\ref{sec:nonabelian-quotient}
uses the non-abelian admissible frame of Cao--Zheng and treats the
three structural types separately.  Section~\ref{sec:abelian-quotient}
uses the different admissible frame constructed in
\cite[Appendix~B]{CaoZhengFV} and treats the cases $b=0$ and
$b\neq0$.  This organization is essential: although the two frames
lead to the same block form for the $D$-matrices, their construction,
structure equations and unimodularity identities are different.

\section{Preliminaries}

Let us start with a Hermitian Lie algebra $({\mathfrak g},J,g)$ of real dimension $2n$.
We extend the real inner product $g=\langle \cdot,\cdot \rangle $ complex bilinearly to the complexification $\mathfrak g_{\C}$ and continue to denote the extension by
$\langle\cdot,\cdot\rangle$.  Thus, for a unitary basis
$e=\{e_1,\ldots,e_n\}$ of $\mathfrak g^{1,0}$, one has
$\langle e_i,\overline e_j\rangle=\delta_{ij}$, whereas
$\langle e_i,e_j\rangle=\langle\overline e_i,\overline e_j\rangle=0$.
Let $\{\vphi_1,\ldots,\vphi_n\}$ be the dual coframe.  Following
\cite{VezzoniYangZheng,CaoZhengFV}, write
$C^j_{ik}=\langle[e_i,e_k],\overline e_j\rangle$ and
$D^j_{ik}=\langle[\overline e_j,e_k],e_i\rangle$.  Equivalently,
\[
 [e_i,e_k]=\sum_j C^j_{ik}e_j,\qquad
 [e_i,\overline e_j]
 =\sum_k\bigl(\overline{D^i_{kj}}\,e_k-D^j_{ki}\overline e_k\bigr).
\]
For each $q$, set $D_q=(D^j_{iq})_{1\leq i,j\leq n}$.

Now suppose that ${\mathfrak g}$ is unimodular. This means that $\mbox{tr}(ad_x)=0$ for any $x\in {\mathfrak g}$, or equivalently
\begin{equation} \label{unimodular}
{\mathfrak g} \ \, \mbox{is unimodular}  \ \ \Longleftrightarrow  \ \ \sum_r \big( C^r_{ri} + D^r_{ri}\big) =0 , \, \ \forall \ i.
\end{equation}

Denote by $\nabla^c$ the Chern connection and by $T=T^c$ its torsion tensor. We have
\begin{equation*} \label{Gamma}
\nabla^c e_i = \sum_j \theta_{ij} e_j, \ \ \ \ \theta_{ij} = \sum_k \big( \Gamma^j_{ik} \varphi_k - \overline{\Gamma^i_{jk}}\, \overline{\varphi}_k \big), \ \ \ \ \ \Gamma^j_{ik} = D^j_{ik},
\end{equation*}
where $\{ \varphi_1, \ldots , \varphi_n\}$ is  the coframe dual to $e$ and $\theta$ is the connection matrix of $\nabla^c$ under $e$. The torsion tensor $T$ of $\nabla^c$ has components
\begin{equation} \label{torsion}
T( e_i, \overline{e}_j)=0, \ \ \ \ T(e_i,e_j)  = \sum_k T^k_{ij}e_k, \ \ \ \  \ \ T^j_{ik}= - C^j_{ik} -  D^j_{ik} + D^j_{ki}.
\end{equation}
Denote by $\omega = \sqrt{-1}\sum_i \varphi_i \wedge \overline{\varphi}_i$ the K\"ahler form of the metric $g$.   The structure equation takes the form:
\begin{equation} \label{structure}
d\varphi_i = -\frac{1}{2} \sum_{j,k} C^i_{jk} \,\varphi_j\wedge \varphi_k - \sum_{j,k} \overline{D^j_{ik}} \,\varphi_j \wedge \overline{\varphi}_k.
\end{equation}
Differentiate the above, we get the  first Bianchi identity, which is equivalent to the Jacobi identity in this case:
\begin{equation} \label{Jacobi}
\left\{
\begin{split}
  \sum_r \big( C^r_{ij}C^{\ell}_{rk} + C^r_{jk}C^{\ell}_{ri} + C^r_{ki}C^{\ell}_{rj} \big) \ = \ 0,  \hspace{3.1cm}\\
  \sum_r \big( C^r_{ik}D^{\ell}_{jr} + D^r_{ji}D^{\ell}_{rk} - D^r_{jk}D^{\ell}_{ri} \big) \ = \ 0,  \hspace{2.9cm} \\
  \ \ \sum_r \big( C^r_{ik}\overline{D^r_{j \ell}}  - C^j_{rk}\overline{D^i_{r \ell}} + C^j_{ri}\overline{D^k_{r \ell}} -  D^{\ell}_{ri}\overline{D^k_{j r}} +  D^{\ell}_{rk}\overline{D^i_{jr}}  \big) \ = \ 0 .
\end{split}
\right.
\end{equation}
Let $R^c$ be the Chern curvature tensor, and $H^c$ the Chern holomorphic sectional curvature. The following curvature formula is standard; see
\cite[Lemma 3]{LiZhengHSC}.

\begin{lemma}\label{lem:Chern-formula}
For all indices $i,j,k,\ell$,
\begin{equation}\label{eq:Chern-formula}
\begin{split}
R^c_{i\bar j k\bar\ell}
=\sum_s\bigl(&
D^s_{ki}\overline{D^s_{\ell j}}
-D^\ell_{si}\overline{D^k_{sj}}
-D^j_{si}\overline{D^k_{\ell s}}
-\overline{D^i_{sj}}D^\ell_{ks}
\bigr).
\end{split}
\end{equation}
In particular,
\begin{equation}\label{eq:diagonal-curvature}
 R^c_{i\bar i i\bar i}
 =\sum_s\left(
 |D^s_{ii}|^2-|D^i_{si}|^2
 -2\re\{D^i_{si}\overline{D^i_{is}}\}
 \right).
\end{equation}
\end{lemma}

For any $4$-tensor $P$ of type $(2,2)$, its symmetrization  tensor $\widehat{P}$ is defined by
\begin{equation}\label{eq:symmetrization}
\widehat P_{i\bar j k\bar\ell}
=\frac14\left(
P_{i\bar j k\bar\ell}
+P_{k\bar j i\bar\ell}
+P_{i\bar\ell k\bar j}
+P_{k\bar\ell i\bar j}
\right).
\end{equation}
Polarization gives
\begin{equation}\label{eq:HSC-polarization}
H^c=\ka
\quad\Longleftrightarrow\quad
\widehat R^c_{i\bar j k\bar\ell}
=\frac{\ka}{2}
(\delta_{ij}\delta_{k\ell}+\delta_{i\ell}\delta_{kj}).
\end{equation}

For a metric connection $D$, let $R^D$ be the curvature of $D$ and $\widehat R^D$ its
symmetrization.  Define the quadratic Chern-torsion tensor $\widehat v$ by
\begin{equation}\label{eq:vhat-definition}
\begin{split}
\widehat v_{i\bar j k\bar\ell}
= \frac14 \sum_a\bigl(&T^j_{ia}\overline{T^k_{\ell a}}
+T^j_{ka}\overline{T^i_{\ell a}}+T^\ell_{ia}\overline{T^k_{ja}}
+T^\ell_{ka}\overline{T^i_{ja}}\bigr).
\end{split}
\end{equation}
In particular,
$\widehat v_{i\bar i i\bar i}=\sum_a|T^i_{ia}|^2$.
By \cite[Equation~(5)]{ChenZhengCanonical},
\begin{equation}\label{eq:canonical-sym-identity}
\widehat R^{D_s^r}_{i\bar j k\bar\ell}
=\widehat R^c_{i\bar j k\bar\ell}
-\chi\widehat v_{i\bar j k\bar\ell}, \ \ \ \ \mbox{where} \ \ \ \chi
=
t^2+\frac{s^2}{4}
=
\frac{(1-r+rs)^2+s^2}{4}.
\end{equation}
Consequently,
\begin{equation}\label{eq:master-equation}
H^{D_s^r}=\ka
\quad\Longleftrightarrow\quad
\widehat R^c_{i\bar j k\bar\ell}
-\chi\widehat v_{i\bar j k\bar\ell}
=\frac{\ka}{2}
(\delta_{ij}\delta_{k\ell}+\delta_{i\ell}\delta_{kj}).
\end{equation}

\subsection{Block curvature formulas}
\label{subsec:block-curvature}

The computations in this subsection are purely algebraic.  They do
not require the existence of an abelian ideal in the Lie algebra, the unimodularity of
the Lie algebra, or any assumption on a quotient Lie algebra.  They
will be applied in Sections~\ref{sec:nonabelian-quotient} and
\ref{sec:abelian-quotient}, after the required block form has been
established separately by the corresponding admissible-frame results
of \cite{CaoZhengFV}.

Let $(\mathfrak g,J,g)$ be a Hermitian Lie algebra of real
dimension $2n$. Since the constant holomorphic sectional curvature conjectures are known to be true for general compact Hermitian surfaces, we will from now on assume that $n\geq3$. Let $e=\{e_1,\ldots,e_n\}$ be a unitary basis
of $\mathfrak g^{1,0}$.  Put
\[
 \mathcal U=\operatorname{span}_{\C}\{e_1,e_2\},\qquad
 \mathcal W=\operatorname{span}_{\C}\{e_3,\ldots,e_n\}.
\]
Thus $\mathfrak g^{1,0}=\mathcal U\oplus\mathcal W$.  Assume that,
relative to this decomposition, the matrices $D_q=(D^j_{iq})$
satisfy
\begin{equation}\label{eq:D-block}
D_\alpha=
\begin{pmatrix}
E_\alpha&0\\
V_\alpha&Y_\alpha
\end{pmatrix},\qquad \alpha=1,2,
\end{equation}
and
\begin{equation}\label{eq:D-p-zero}
D_p=0,\qquad 3\leq p\leq n.
\end{equation}
Here $E_\alpha\in M_2(\C)$,
$V_\alpha\in M_{n-2,2}(\C)$ and
$Y_\alpha\in M_{n-2}(\C)$.  We write
\[
V_\alpha=(v_\alpha^1,v_\alpha^2)
\]
for the two columns of $V_\alpha$.

\begin{lemma}\label{lem:curvature-support}
Under \eqref{eq:D-block}--\eqref{eq:D-p-zero}, if $p\geq3$ or
$q\geq3$, then
\[
R^c_{p\bar qk\bar\ell}=0
\]
for all $k,\ell$.  Consequently, if $n\geq3$ and $H^c=\ka$ is
constant, then $\ka=0$.
\end{lemma}

\begin{proof}
We first assume $p\geq3$.  Substituting $i=p$ and $j=q$ into
\eqref{eq:Chern-formula} gives
\begin{equation*}
\begin{split}
R^c_{p\bar qk\bar\ell}
=\sum_s\bigl(&D^s_{kp}\overline{D^s_{\ell q}}
-D^\ell_{sp}\overline{D^k_{sq}} -D^q_{sp}\overline{D^k_{\ell s}}
-\overline{D^p_{sq}}D^\ell_{ks}\bigr).
\end{split}
\end{equation*}
The first three terms contain, respectively, the entries
$D^s_{kp}$, $D^\ell_{sp}$ and $D^q_{sp}$ of the matrix $D_p$.
They vanish by \eqref{eq:D-p-zero}.  For the fourth term, split the
sum over $s$ into $s\geq3$ and $s\in\{1,2\}$.  If $s\geq3$, then
$D^\ell_{ks}=0$ because $D_s=0$.  If $s\in\{1,2\}$, then
$D^p_{sq}$ is the entry in row $s\in\mathcal U$ and column
$p\in\mathcal W$ of $D_q$; it lies in the upper-right block of
$D_q$, which is zero by \eqref{eq:D-block}.  Thus the fourth term
also vanishes, proving
$R^c_{p\bar qk\bar\ell}=0$ when $p\geq3$.

Now assume that $q\geq3$. Taking the
complex conjugate, we conclude that
\[
R^c_{p\bar q k\bar\ell}
=
\overline{R^c_{q\bar p\ell\bar k}}
=0.
\]
Taking $p=q=k=\ell\geq3$ gives
\[
R^c_{p\bar pp\bar p}=0.
\]
Since $n\geq3$, such a $p$ exists. The identity \eqref{eq:HSC-polarization} gives $R^c_{p\bar pp\bar p}=\ka$.  Hence $\ka=0$.
\end{proof}

\begin{lemma}\label{lem:curvature-matrix}
For $\alpha,\beta\in\{1,2\}$, set
\[
\mathcal R_{\alpha\bar\beta}
:=
\bigl(R^c_{\alpha\bar\beta k\bar\ell}\bigr)_{1\leq k,\ell\leq n}.
\]
Then
\begin{equation}\label{eq:curvature-matrix}
\mathcal R_{\alpha\bar\beta}
=
[D_\alpha,D_\beta^*]
-\sum_{\gamma=1}^2
(E_\alpha)_{\gamma\beta}D_\gamma^*
-\sum_{\gamma=1}^2
\overline{(E_\beta)_{\gamma\alpha}}D_\gamma.
\end{equation}
Relative to
$\mathfrak g^{1,0}=\mathcal U\oplus\mathcal W$, write
\[
\mathcal R_{\alpha\bar\beta}
=
\begin{pmatrix}
P_{\alpha\beta}&*\\
*&Q_{\alpha\beta}
\end{pmatrix}.
\]
Then
\begin{align}
P_{\alpha\beta}&:=
\bigl(
R^c_{\alpha\bar\beta p\bar q}
\bigr)_{1\leq p,q\leq2}={}[E_\alpha,E_\beta^*]-V_\beta^*V_\alpha
-\sum_{\gamma=1}^2(E_\alpha)_{\gamma\beta}E_\gamma^*
-\sum_{\gamma=1}^2
\overline{(E_\beta)_{\gamma\alpha}}E_\gamma,
\label{eq:P-block}\\
Q_{\alpha\beta}&:=
\bigl(
R^c_{\alpha\bar\beta p\bar q}
\bigr)_{3\leq p,q\leq n}
={}V_\alpha V_\beta^*+[Y_\alpha,Y_\beta^*]
-\sum_{\gamma=1}^2(E_\alpha)_{\gamma\beta}Y_\gamma^*
-\sum_{\gamma=1}^2
\overline{(E_\beta)_{\gamma\alpha}}Y_\gamma.
\label{eq:Q}
\end{align}
Moreover, if $H^c=0$, then
\begin{equation}\label{eq:Q-zero}
Q_{\alpha\beta}=0,
\qquad
\alpha,\beta\in\{1,2\}.
\end{equation}
\end{lemma}

\begin{proof}
Fix $\alpha,\beta\in\{1,2\}$.  Substituting
$i=\alpha$ and $j=\beta$ into
\eqref{eq:Chern-formula}, we have
\begin{equation}\label{eq:R-alpha-beta-component}
\begin{aligned}
R^c_{\alpha\bar\beta k\bar\ell}
=\sum_s\bigl(
D^s_{k\alpha}\overline{D^s_{\ell\beta}}
-D^\ell_{s\alpha}\overline{D^k_{s\beta}}
-D^\beta_{s\alpha}\overline{D^k_{\ell s}}
-\overline{D^\alpha_{s\beta}}D^\ell_{ks}
\bigr),\quad \forall \, 1\leq k,\ell\, \leq n.
\end{aligned}
\end{equation}
Recall that $(D_\alpha)_{ks}=D^s_{k\alpha}.$ Therefore,
\[
\sum_sD^s_{k\alpha}\overline{D^s_{\ell\beta}}-\sum_s\overline{D^k_{s\beta}}D^\ell_{s\alpha}=(D_\alpha D_\beta^*)_{k\ell}-(D_\beta^*D_\alpha)_{k\ell}
=([D_\alpha,D_\beta^*])_{k\ell},\quad \forall \, 1\leq k,\ell\, \leq n.
\]
Formula \eqref{eq:D-p-zero} implies that
$D_s=0$ whenever $s\geq3$. Let $s=\gamma$, we have
\[
D^\beta_{\gamma\alpha}=(E_\alpha)_{\gamma\beta},
\qquad
\overline{D^k_{\ell\gamma}}=(D_\gamma^*)_{k\ell},
\]
and hence
\[
-\sum_sD^\beta_{s\alpha}\overline{D^k_{\ell s}}-\sum_s\overline{D^\alpha_{s\beta}}D^\ell_{ks}
=-\sum_{\gamma=1}^2
(E_\alpha)_{\gamma\beta}(D_\gamma^*)_{k\ell}-\sum_{\gamma=1}^2
\overline{(E_\beta)_{\gamma\alpha}}(D_\gamma)_{k\ell},\quad \forall \, 1\leq k,\ell\, \leq n.
\]
Combining the two terms proves \eqref{eq:curvature-matrix}.

We now compute the two diagonal blocks.  By
\eqref{eq:D-block}, we have
\[
D_\alpha D_\beta^*
=
\begin{pmatrix}
E_\alpha E_\beta^*&E_\alpha V_\beta^*\\
V_\alpha E_\beta^*&
V_\alpha V_\beta^*+Y_\alpha Y_\beta^*
\end{pmatrix},
\]
and
\[
D_\beta^*D_\alpha
=
\begin{pmatrix}
E_\beta^*E_\alpha+V_\beta^*V_\alpha&
V_\beta^*Y_\alpha\\
Y_\beta^*V_\alpha&
Y_\beta^*Y_\alpha
\end{pmatrix}.
\]

It follows that the $\mathcal U$-to-$\mathcal U$ block of
the commutator is
\[
\begin{aligned}
\bigl([D_\alpha,D_\beta^*]\bigr)_{\mathcal U\mathcal U}
&=
E_\alpha E_\beta^*
-E_\beta^*E_\alpha
-V_\beta^*V_\alpha=
[E_\alpha,E_\beta^*]
-V_\beta^*V_\alpha,
\end{aligned}
\]
whereas its $\mathcal W$-to-$\mathcal W$ block is
\[
\begin{aligned}
\bigl([D_\alpha,D_\beta^*]\bigr)_{\mathcal W\mathcal W}
&=
V_\alpha V_\beta^*
+Y_\alpha Y_\beta^*
-Y_\beta^*Y_\alpha=
V_\alpha V_\beta^*
+[Y_\alpha,Y_\beta^*].
\end{aligned}
\]

Since the upper-left and lower-right blocks of
$D_\gamma$ are respectively $E_\gamma$ and
$Y_\gamma$, while those of $D_\gamma^*$ are
$E_\gamma^*$ and $Y_\gamma^*$, taking the
$\mathcal U$-to-$\mathcal U$ block in
\eqref{eq:curvature-matrix} gives
\[
\begin{aligned}
P_{\alpha\beta}
={}&[E_\alpha,E_\beta^*]-V_\beta^*V_\alpha-\sum_{\gamma=1}^2
(E_\alpha)_{\gamma\beta}E_\gamma^*
-\sum_{\gamma=1}^2
\overline{(E_\beta)_{\gamma\alpha}}E_\gamma,
\end{aligned}
\]
which is \eqref{eq:P-block}.  Taking the
$\mathcal W$-to-$\mathcal W$ block gives
\[
\begin{aligned}
Q_{\alpha\beta}
={}&V_\alpha V_\beta^*
+[Y_\alpha,Y_\beta^*]-\sum_{\gamma=1}^2
(E_\alpha)_{\gamma\beta}Y_\gamma^*
-\sum_{\gamma=1}^2
\overline{(E_\beta)_{\gamma\alpha}}Y_\gamma,
\end{aligned}
\]
which is \eqref{eq:Q}. Finally, let $p,q\geq3$.  By
\eqref{eq:symmetrization} and Lemma~\ref{lem:curvature-support}, we have
\[
\begin{aligned}
4\widehat R^c_{\alpha\bar\beta p\bar q}
={}&R^c_{\alpha\bar\beta p\bar q}
+R^c_{p\bar\beta\alpha\bar q}+R^c_{\alpha\bar q p\bar\beta}
+R^c_{p\bar q\alpha\bar\beta}=R^c_{\alpha\bar\beta p\bar q}, \quad \forall \, 3\leq p, q \leq n.
\end{aligned}
\]
If $H^c=0$, then the identity
\eqref{eq:HSC-polarization} implies $\widehat R^c_{\alpha\bar\beta p\bar q}=0.$ Hence
\[
R^c_{\alpha\bar\beta p\bar q}=0,
\qquad \forall \, 3\leq p,q\leq n.
\]
By the definition of $Q_{\alpha\beta}$, this is precisely
$Q_{\alpha\beta}=0.$ This completes the proof of the lemma.
\end{proof}

Taking the trace of \eqref{eq:Q-zero} for $\alpha=\beta$ and using
$\tr[A,B]=0$ gives
\begin{equation}\label{eq:trace-Q}
\|V_\alpha\|^2
=2\re\left\{
\sum_{\gamma=1}^2
(E_\alpha)_{\gamma\alpha}\overline{\tr(Y_\gamma)}
\right\}.
\end{equation}
Indeed, $\tr(V_\alpha V_\alpha^*)=\|V_\alpha\|^2$,
$\tr[Y_\alpha,Y_\alpha^*]=0$, and the last two sums in
\eqref{eq:Q} are complex conjugates of each other when
$\alpha=\beta$.

We shall also use the diagonal components.  Substituting the block
entries into \eqref{eq:diagonal-curvature} gives
\begin{equation}\label{eq:r1-general}
R^c_{1\bar11\bar1}=r_1-\|v_1^1\|^2, \quad r_1=
|(E_1)_{12}|^2-|(E_1)_{21}|^2
-2|(E_1)_{11}|^2
-2\re\{(E_1)_{21}\overline{(E_2)_{11}}\},
\end{equation}
and
\begin{equation}\label{eq:r2-general}
R^c_{2\bar22\bar2}=r_2-\|v_2^2\|^2, \quad r_2=
|(E_2)_{21}|^2-|(E_2)_{12}|^2
-2|(E_2)_{22}|^2
-2\re\{(E_2)_{12}\overline{(E_1)_{22}}\}.
\end{equation}

\begin{lemma}\label{lem:commutator-nilpotent}
Let $A,B\in M_m(\C)$ satisfy $[A,B]=\lambda A$ for some
$\lambda\in\C\setminus\{0\}$.  Then $A$ is nilpotent.  In
particular, if $m=2$, then $\tr A=0$, $\det A=0$, and $A^2=0$.
\end{lemma}
Since this is a standard fact in linear algebra and is well-known, we will omit its proof here.
%
\begin{lemma}\label{lem:circle-rigidity}
Let $A,B\in M_m(\C)$ be Hermitian matrices.  Suppose that
\begin{equation}\label{eq:circle-rigidity-assumption}
\langle Aw,\overline w\rangle^2
+
\langle Bw,\overline w\rangle^2
=
c_0
\end{equation}
for some $c_0\in\R$ and every unit vector $w\in\C^m$.  Then
there exist $\alpha_0,\beta\in\R$ such that
\[
A=\alpha_0I_m,
\qquad
B=\beta I_m.
\]
\end{lemma}

\begin{proof}
The case $m=1$ is immediate, so assume that $m\geq2$.  After a
unitary change of basis, we may write
\[
A=\operatorname{diag}(a_1,\ldots,a_m),
\qquad
a_j\in\R.
\]
Write $B=(b_{jk})$.  We first show that $B$ is diagonal in this
basis.

Fix $p\neq q$ and set
\[
w_\theta=\frac{e_p+\E^{\I\theta}e_q}{\sqrt2},
\qquad
\theta\in\R.
\]
Then
\[
\langle Aw_\theta,\overline{w_\theta}\rangle
=
\frac{a_p+a_q}{2}.
\]
Since $B$ is Hermitian,
\[
\langle Bw_\theta,\overline{w_\theta}\rangle
=
\frac{b_{pp}+b_{qq}}2
+
\operatorname{Re}\left(\E^{\I\theta}b_{pq}\right).
\]
It follows from \eqref{eq:circle-rigidity-assumption} that
\[
\langle Bw_\theta,\overline{w_\theta}\rangle^2
=
c_0-\frac{(a_p+a_q)^2}{4},
\]
which is independent of $\theta$.  The expression on the left
before squaring is a continuous real-valued function of $\theta$
whose image is contained in a set of at most two points.  It is
therefore constant.  Hence
$\operatorname{Re}(\E^{\I\theta}b_{pq})$ is constant in $\theta$.
On the other hand,
\[
\operatorname{Re}\left(\E^{\I(\theta+\pi)}b_{pq}\right)
=
-\operatorname{Re}\left(\E^{\I\theta}b_{pq}\right),
\]
so this constant must be zero.  Taking $\theta=0$ and
$\theta=\pi/2$ gives
\[
\operatorname{Re}b_{pq}
=
\operatorname{Im}b_{pq}
=
0.
\]
Thus $b_{pq}=0$.  Since $p\neq q$ were arbitrary,
\[
B=\operatorname{diag}(b_1,\ldots,b_m),
\qquad
b_j\in\R.
\]

Set $v_j=(a_j,b_j)\in\R^2$.  Taking $w=e_j$ in
\eqref{eq:circle-rigidity-assumption} gives
\[
|v_j|^2=a_j^2+b_j^2=c_0
\]
for every $j$.  Fix $p\neq q$ and, for $t\in[0,1]$, let
\[
w_t=\sqrt{1-t}\,e_p+\sqrt t\,e_q.
\]
Since $A$ and $B$ are diagonal,
\eqref{eq:circle-rigidity-assumption} becomes
\[
\left|(1-t)v_p+tv_q\right|^2=c_0.
\]
But
\[
\begin{aligned}
\left|(1-t)v_p+tv_q\right|^2
&=
(1-t)|v_p|^2+t|v_q|^2
-t(1-t)|v_p-v_q|^2\\
&=
c_0-t(1-t)|v_p-v_q|^2.
\end{aligned}
\]
Choosing $0<t<1$ yields $v_p=v_q$.  Since $p$ and $q$ are
arbitrary, all the vectors $v_j$ coincide.  Thus there exist
$\alpha_0,\beta\in\R$ such that $a_j=\alpha_0$ and $b_j=\beta$
for every $j$. Consequently,
\[ A=\alpha_0I_m, \qquad B=\beta I_m. \]
This completes the proof of the lemma.
\end{proof}

\vspace{0.3cm}

\section{The $J{\mathfrak a} \neq {\mathfrak a}$ and ${\mathfrak g} / {\mathfrak a}$ non-abelian case}
\label{sec:nonabelian-quotient}

Throughout this section, we will assume that {\em  ${\mathfrak g}$ is a unimodular Lie algebra of real dimension $2n$ containing an abelian ideal ${\mathfrak a}$ of codimension $2$,  $J$  a complex structure on ${\mathfrak g}$, and $g=\langle , \rangle $ an inner product on ${\mathfrak g}$ compatible with $J$. }

When $J{\mathfrak a}={\mathfrak a}$, the structure constants $C$ and $D$ take a relatively simple form \cite{GuoZ}, and the Chern and Levi-Civita case of Theorem \ref{thm:main} are proved in \cite{LiZhengHSC}. For other connections $D^r_s$ it can be proved similarly, and we will omit it here since it will involve a whole different set of notations. So from now on we will always assume that $J{\mathfrak a} \neq {\mathfrak a}$. In this section, we will also assume that ${\mathfrak g}/{\mathfrak a}$ is not abelian, and leave the other possibility to the next section. In other words, in this section we will assume that
\begin{equation}\label{eq:nonabelian-assumptions}
  J{\mathfrak a} \neq {\mathfrak a}, \ \  \mbox{and} \ \ {\mathfrak g}/{\mathfrak a} \ \mbox{is non-abelian}.
\end{equation}
Set
\[
\mathfrak a_J:=\mathfrak a\cap J\mathfrak a,
\qquad
\mathfrak a':=\mathfrak a+[\mathfrak g,\mathfrak g],
\qquad
\mathfrak b=\mathfrak a\cap J\mathfrak a'.
\]
By \cite[Lemma 3]{CaoZhengFV}, $\mathfrak a_J$ is an ideal of real
codimension $4$, ${\mathfrak a}'$ is of codimension $1$ in ${\mathfrak g}$, and
\[
\mathfrak a_J\subsetneq\mathfrak b\subsetneq\mathfrak a
\subsetneq\mathfrak a'\subsetneq\mathfrak g.
\]
Furthermore, if $x\in {\mathfrak b}\setminus {\mathfrak a}_J$, then $Jx \in {\mathfrak a}' \setminus {\mathfrak a}$; while if $y\in {\mathfrak a}\setminus {\mathfrak b}$, then $Jy \notin {\mathfrak a}' $.

Choose orthonormal vectors $x,y\in\mathfrak a\cap\mathfrak a_J^\perp$
with $x\in\mathfrak b$, and define
\[
e_1=\frac1{\sqrt2}(x-\I Jx),
\qquad
Y=\frac1{\sqrt2}(y-\I Jy).
\]
Writing
\[
\delta=\langle Jx,y\rangle\in(-1,1),
\qquad
\delta'=\sqrt{1-\delta^2}\in (0,1],
\]
we choose $e_2$ so that
\begin{equation}\label{eq:nonab-admissible-frame}
Y=\I\delta e_1+\delta'e_2,
\qquad
(\mathfrak a_J)^{1,0}
=\operatorname{span}_{\C}\{e_3,\ldots,e_n\}.
\end{equation}
This is the admissible frame of
\cite[Definition above Lemma 4]{CaoZhengFV}.  Notice that this
particular definition uses the non-abelian quotient assumption through
the distinguished line
$\mathfrak b/\mathfrak a_J$ and the requirement $x\in\mathfrak b$.

By \cite[Lemma 4]{CaoZhengFV}, the only possibly non-zero $D$-matrices
in this frame are
\[
D_\alpha=
\begin{pmatrix}E_\alpha&0\\V_\alpha&Y_\alpha\end{pmatrix},
\qquad \alpha=1,2,
\qquad
D_p=0\quad(p\geq3).
\]
Thus all conclusions of Subsection \ref{subsec:block-curvature} apply.
In particular, constant Chern holomorphic sectional curvature would imply
\begin{equation}\label{eq:nonab-kappa-zero}
\ka=0,
\qquad
Q_{\alpha\beta}=0\quad(1\leq\alpha,\beta\leq2).
\end{equation}

More precisely, \cite[Lemma~4]{CaoZhengFV} gives, for
$\alpha,\beta\in\{1,2\}$, $p,q\geq3$, and $1\leq \ast \leq n$,
\begin{equation}\label{eq:nonab-C-D-relations}
\begin{aligned}
& C^{\ast}_{pq} = D^{\ast}_{\ast p} = C^{\alpha }_{\beta p} = D^p_{\alpha\beta} =0, \ \ \ \zeta:=\frac{\I\delta}{\delta'},\\
& C^{\ast}_{1p}= \overline{D^p_{\ast 1}} , \ \ \ C^{\ast}_{2p}= \overline{D^p_{\ast 2}} --2\zeta\overline{D^p_{\ast1}}, \ \ \ C^{\ast}_{12}= \overline{D^2_{\ast 1}} -  \overline{D^1_{\ast 2}} + 2\zeta\overline{D^1_{\ast1}},
\end{aligned}
\end{equation}
Write $V_\alpha=(v_\alpha^1,v_\alpha^2)$ and define
\begin{equation}\label{eq:nonab-Z-definitions}
Z_1=Y_1-Y_1^*,
\qquad
Z_2=Y_2-Y_2^*+2\zeta Y_1^*.
\end{equation}
We regard $Y_\alpha$ and $Z_\alpha$ as the corresponding
complex-linear endomorphisms of $\mathcal W$ in the chosen unitary
basis. Throughout this section, Greek indices
$\alpha,\beta,\gamma$ range over $\{1,2\}$, whereas
$p,q,r$ range over $\{3,\ldots,n\}$. By \eqref{torsion} and
\eqref{eq:nonab-C-D-relations}, we have
\begin{align}
T^q_{1p}
&=-\overline{D^p_{q1}}+D^q_{p1}
=(Z_1)_{pq},
\label{eq:nonab-torsion-1}\\
T^q_{2p}
&=-\overline{D^p_{q2}}
 +2\zeta\overline{D^p_{q1}}
 +D^q_{p2}
=(Z_2)_{pq},
\label{eq:nonab-torsion-2}\\
T^\beta_{\alpha p}
&=D^\beta_{p\alpha}
=(v_\alpha^\beta)_p,
\label{eq:nonab-torsion-3}\\
T^q_{pr}
&=-C^q_{pr}-D^q_{pr}+D^q_{rp}=0,\\
T^*_{pq}&=0,
\label{eq:nonab-torsion-W}
\end{align}
Moreover, the $\mathcal W$-valued component of $T(e_1,e_2)$ is
\begin{equation}\label{eq:nonab-u}
u=(T^p_{12})_{p=3}^n
=-\overline{v_1^2}+\overline{v_2^1}
-2\zeta\overline{v_1^1}, \quad \zeta=\frac{\I\delta}{\delta'}.
\end{equation}
Indeed, $D^p_{12}=D^p_{21}=0$ and the last identity in
\eqref{eq:nonab-C-D-relations} give
$T^p_{12}=-\overline{D^2_{p1}}+\overline{D^1_{p2}}
-2\zeta\overline{D^1_{p1}}$.

We now recall how the real structure constants arise.  Modulo
$\mathfrak a_J$, write
\begin{equation}\label{eq:nonab-real-brackets}
\left\{ \begin{split}
[Jx, x] =  ax + by ;\ \,\\
[Jx, y] =  cx + dy ;\ \,\\
\ [Jy, x] =  a'x + b'y;\\
\ [Jy, y] =  c'x + d'y.
\end{split}
\right.
\end{equation}
where all coefficients are real.  Since $x\in\mathfrak b$, one has
$Jx\in\mathfrak a'\setminus\mathfrak a$.  Therefore
\begin{equation}\label{eq:sigma-definition}
[Jx,Jy]=\sigma Jx\pmod{\mathfrak a_J}
\end{equation}
for a real number $\sigma$.  The quotient is non-abelian, so
$\sigma\neq0$.  The integrability identity
\[
[Jx,Jy]=J([Jx,y]-[Jy,x])
\]
and \eqref{eq:nonab-real-brackets} give
\[
(c-a')Jx+(d-b')Jy
=\sigma Jx\pmod{\mathfrak a_J}.
\]
Since $Jx,Jy$ are linearly independent modulo $\mathfrak a$, it follows
that
\begin{equation}\label{eq:aprime-bprime}
a'=c-\sigma,
\qquad
b'=d.
\end{equation}

Let
\[
A=\begin{pmatrix}a&b\\c&d\end{pmatrix},
\qquad
B=\begin{pmatrix}a'&b'\\c'&d'\end{pmatrix}.
\]
Applying the Jacobi identity
\[
[[Jx,Jy],z]=[Jx,[Jy,z]]-[Jy,[Jx,z]]
\]
for $z=x,y$, and passing to $\mathfrak a/\mathfrak a_J$, yields
\begin{equation}\label{eq:AB-relation}
[A,B]=-\sigma A.
\end{equation}
Since $[A,B]=-\sigma A$ and $\sigma\neq0$,
Lemma~\ref{lem:commutator-nilpotent} implies that $A$ is nilpotent.
Since $A$ is a $2\times2$ matrix, one has $\tr A=\det A=0$; hence
\begin{equation}\label{eq:A-nilpotent-relations}
d=-a,
\qquad
a^2+bc=0.
\end{equation}
Together with \eqref{eq:aprime-bprime}, this also gives $b'=-a$.
Expanding \eqref{eq:AB-relation} entry by entry gives
\begin{equation}\label{eq:nonab-Jacobi-scalars}
bc'=-a(c+\sigma),
\qquad
bd'=a^2-2b\sigma,
\qquad
2ac'+cd'=c^2.
\end{equation}

The three mutually exclusive types of complex structures are those of
\cite[Definition preceding Lemma 5]{CaoZhengFV}:
\[
\begin{array}{lll}
\text{$J$ is generic} &\Longleftrightarrow& b\neq0,\\
\text{$J$ is half-generic} &\Longleftrightarrow& b=0,\ c\neq0,\\
\text{$J$ is degenerate} &\Longleftrightarrow& b=c=0.
\end{array}
\]
Solving \eqref{eq:A-nilpotent-relations} and
\eqref{eq:nonab-Jacobi-scalars} in these three cases gives
\begin{equation}\label{eq:nonab-three-types}
\begin{array}{ll}
\text{$J$ is generic:}&
 b\neq0,\quad c=-a^2/b,\quad
 c'=-\dfrac ab(c+\sigma),\quad d'=-c-2\sigma;\\[1mm]
\text{$J$ is half-generic:}&
 a=b=0,\quad c\neq0,\quad d'=c;\\[1mm]
\text{$J$ is degenerate:}&
 a=b=c=0.
\end{array}
\end{equation}

The matrices $E_1,E_2$ are
\begin{equation}\label{eq:E1-nonab}
E_1=\frac1{\sqrt2}
\begin{pmatrix}
b\delta+\I a&
\displaystyle\frac{-2a\delta+\I c+\I b\delta^2}{\delta'}\\[2mm]
\I b\delta'&-b\delta-\I a
\end{pmatrix},
\end{equation}
and
\begin{equation}\label{eq:E2-nonab}
E_2=\frac1{\sqrt2}
\begin{pmatrix}
\displaystyle\frac{\I(c-\sigma-b\delta^2)}{\delta'}&
\displaystyle\frac{\I(c'+a\delta^2)
+\delta(d'+b\delta^2+\sigma)}{\delta'^2}\\[2mm]
b\delta-\I a&
\displaystyle\frac{\I(d'+b\delta^2)}{\delta'}
\end{pmatrix}.
\end{equation}
By \cite[(22)--(23)]{CaoZhengFV}, the Jacobi identities \eqref{Jacobi} become
\begin{equation}\label{eq:nonab-D-Jacobi}
[D_1,D_2]=\lambda D_1,
\qquad
\lambda:=\frac{\I\sigma}{\sqrt2\,\delta'}\neq0.
\end{equation}
or equivalently,
\begin{align}
&[E_1,E_2]=\lambda E_1,
\label{eq:nonab-Jacobi-E}\\
&[Y_1,Y_2]=\lambda Y_1,
\label{eq:nonab-Jacobi-Y}\\
&V_1E_2+Y_1V_2-V_2E_1-Y_2V_1
=\lambda V_1.
\label{eq:nonab-block-Jacobi}
\end{align}
By Lemma~\ref{lem:commutator-nilpotent}, both $E_1$ and $Y_1$
are nilpotent.

\begin{lemma}[\cite{CaoZhengFV}]\label{lem:nonab-structure}
Let $({\mathfrak g}, J, g)$ be a Hermitian Lie algebra containing an abelian ideal ${\mathfrak a}$ of codimension $2$ satisfying \eqref{eq:nonabelian-assumptions}.
Let
${e_1,\ldots,e_n}$ be the admissible unitary frame
described above. Then we have
\begin{equation}\label{eq:nonab-C12}
C^1_{12}
=-\frac{\I\sigma}{\sqrt2,\delta'},
\qquad
C^2_{12}=0.
\end{equation}
Furthermore, $\mathfrak g$ is unimodular if and only if
\begin{equation}\label{eq:unimod-nonab}
\tr(Y_2)-\overline{\tr(Y_2)}
=\frac{\I}{\sqrt2 \delta'}(2\sigma-d'-c).
\end{equation}
\end{lemma}

Put $\tau_\alpha:=\tr(Y_\alpha)$ for $\alpha=1,2$.  Taking the
trace of the second identity in \eqref{eq:nonab-Jacobi-Y} and using
$\tr[A,B]=0$ for arbitrary square matrices $A,B$, we obtain
$0=\lambda\tr(Y_1)$.  Since $\lambda\neq0$, it follows that
\begin{equation}\label{eq:nonab-trY1}
\tau_1=\tr(Y_1)=0.
\end{equation}
This identity depends only on the non-abelian admissible-frame
relations and will be used in the Chern and canonical-connection
arguments below.

\subsection{The Chern case: generic type}

In the generic case,
\begin{equation}\label{eq:generic-data}
b\neq0,\qquad
c=-\frac{a^2}{b},\qquad
c'=-\frac ab(c+\sigma),\qquad
d'=-c-2\sigma .
\end{equation}
The trace identities needed in the proof are consequences of the
structural equations and therefore are recorded before the
proposition.  By \eqref{eq:unimod-nonab} and $d'=-c-2\sigma$,
\begin{align*}
\tau_2-\overline{\tau_2}
&=\frac{\I}{\sqrt2\,\delta'}(2\sigma-d'-c)\\
&=\frac{\I}{\sqrt2\,\delta'}
  \{2\sigma-(-c-2\sigma)-c\}\\
&=\frac{2\sqrt2\,\I\sigma}{\delta'}.
\end{align*}
Since $\tau_2-\overline{\tau_2}=2\I\,\im\tau_2$, we obtain
\begin{equation}\label{eq:generic-ImtrY2}
\im\tau_2=\frac{\sqrt2\,\sigma}{\delta'}.
\end{equation}

\begin{proposition}\label{prop:generic}
Let $(\mathfrak g,J,g)$ be a unimodular Hermitian Lie algebra of
complex dimension $n\geq3$ containing an abelian ideal
$\mathfrak a$ of real codimension two.  Assume that
$J\mathfrak a\neq\mathfrak a$, that
$\mathfrak g/\mathfrak a$ is non-abelian, and that the associated
complex structure is of generic type, equivalently $b\neq0$ in
\eqref{eq:nonab-real-brackets}.  Then the Chern holomorphic
sectional curvature of $g$ cannot be constant.
\end{proposition}

\begin{proof}
Assume that the Chern holomorphic sectional
curvature is constant.  By Lemma \ref{lem:curvature-support}, its constant value is zero, $H^c\equiv0.$ In particular,
\begin{equation}\label{eq:generic-diagonal-zero}
R^c_{1\bar11\bar1}=0,
\qquad
R^c_{2\bar22\bar2}=0,
\end{equation}
and \eqref{eq:nonab-kappa-zero} gives
$$Q_{\alpha\beta}=0,\quad \forall \, 1\leq\alpha,\beta\leq2.$$ Let
\begin{equation}\label{eq:generic-dimensionless}
 U:=\frac{a^2}{b^2},
 \qquad
 H:=\frac{\sigma}{b},
 \qquad
 \eta:=\delta^2.
\end{equation}
Thus
\begin{equation}\label{eq:generic-LS}
 \left\{ \begin{split} L:=1-\eta=(\delta')^2>0,\ \
 S:=U+\eta, \hspace{4.0cm}\\
  c=-\frac{a^2}{b}=-bU, \ \
 d'=-c-2\sigma=b(U-2H), \ \  c'=a(U-H). \end{split} \right.
\end{equation}
Note that the symbol $H$ in \eqref{eq:generic-dimensionless} is a dimensionless
structure parameter and is not the holomorphic sectional curvature
$H^c$.

\medskip
\noindent\emph{Step 1: computation of $\|V_1\|^2$.}
Taking the trace of $Q_{11}=0$ and using
\eqref{eq:trace-Q}, \eqref{eq:E1-nonab}, and \eqref{eq:nonab-trY1}, we obtain
\[
\begin{aligned}
\|V_1\|^2=
2\re\left\{
(E_1)_{21}\overline{\tau_2}
\right\}
=
2\re\left\{
\frac{\I b\delta'}{\sqrt2}\,
\overline{\tau_2}
\right\}
=
\sqrt2\,b\delta'\,\im\tau_2.
\end{aligned}
\]
Using \eqref{eq:generic-ImtrY2}, we conclude that
\begin{equation}\label{eq:generic-V1-norm}
\|V_1\|^2
=
2b\sigma
=
2b^2H.
\end{equation}
The left hand side is non-negative, so $H\geq0$.  Moreover,
$b\neq0$ in the generic case and $\sigma\neq0$ in the non-abelian
quotient case; hence $H=\sigma/b\neq0$.  Thus
\begin{equation}\label{eq:generic-H-positive}
        H>0.
\end{equation}

\medskip
\medskip
\noindent\emph{Step 2: computation of $\|V_2\|^2$.}
Taking the trace of $Q_{22}=0$ and using
\eqref{eq:trace-Q}, \eqref{eq:E2-nonab}, \eqref{eq:nonab-trY1}, and \eqref{eq:generic-ImtrY2}, we obtain
\begin{equation}\label{eq:generic-V2-first}
\begin{aligned}
\|V_2\|^2=2\re\left\{(E_2)_{22}\overline{\tau_2}\right\}=\frac{2\sigma}{(\delta')^2}
(d'+b\delta^2).
\end{aligned}
\end{equation}
By \eqref{eq:generic-dimensionless} and \eqref{eq:generic-LS}, the formula \eqref{eq:generic-V2-first} becomes
\begin{equation}\label{eq:generic-V2-norm}
 \|V_2\|^2
 =\frac{2b^2H}{1-\eta}(U+\eta-2H).
\end{equation}
Because $b^2>0$, $H>0$ and $1-\eta>0$, the non-negativity of
$\|V_2\|^2$ implies
\begin{equation}\label{eq:generic-H-bound}
        H\leq\frac{U+\eta}{2}.
\end{equation}

\medskip
\noindent\emph{Step 3: computation of $r_1$.}
Recall from \eqref{eq:r1-general} that
\begin{align}
 r_1={}&|(E_1)_{12}|^2-|(E_1)_{21}|^2
 -2|(E_1)_{11}|^2-2\re\{(E_1)_{21}\overline{(E_2)_{11}}\}.
 \label{eq:generic-r1-start}
\end{align}
We compute the four terms separately. By \eqref{eq:E1-nonab}, \eqref{eq:E2-nonab},  \eqref{eq:generic-dimensionless}, and \eqref{eq:generic-LS}, we have
\begin{align*}
|(E_1)_{12}|^2=\frac{b^2(U+\eta)^2}{2(1-\eta)},\quad |(E_1)_{21}|^2 =\frac{b^2(1-\eta)}2,\ \ \ \ \ \\ |(E_1)_{11}|^2 =\frac{b^2(U+\eta)}2, \quad(E_2)_{11}=-\frac{\I b(U+H+\eta)}{\sqrt2\,\delta'}.
\end{align*}
and therefore
\begin{equation}\label{eq:generic-cross-r1}
-2\re\{(E_1)_{21}\overline{(E_2)_{11}}\}
=b^2(U+H+\eta).
\end{equation}
Substituting these four expressions into
\eqref{eq:generic-r1-start} gives
\begin{align*}
\frac{r_1}{b^2}=H+\frac{(U+\eta)^2-(1-\eta)^2}{2(1-\eta)}.
\end{align*}
Define
\begin{equation}\label{eq:generic-F-definition}
        F:=U^2+2\eta U+2\eta-1.
\end{equation}
Since
\[
(U+\eta)^2-(1-\eta)^2
=U^2+2\eta U+2\eta-1=F,
\]
we obtain
\begin{equation}\label{eq:generic-r1}
\frac{r_1}{b^2}
 =H+\frac{F}{2(1-\eta)}.
\end{equation}

\medskip
\noindent\emph{Step 4: an upper bound for $F$.}
By \eqref{eq:generic-diagonal-zero} and
\eqref{eq:r1-general},
\[
        r_1=\|v_1^1\|^2.
\]
Since $V_1=(v_1^1,v_1^2)$,
\[
\|V_1\|^2=\|v_1^1\|^2+\|v_1^2\|^2,
\]
and therefore $\|v_1^1\|^2\leq\|V_1\|^2$.  Using
\eqref{eq:generic-V1-norm},
\[
\frac{\|v_1^1\|^2}{b^2}
 \leq\frac{\|V_1\|^2}{b^2}=2H.
\]
Substituting \eqref{eq:generic-r1} and using $1-\eta>0$ gives
\begin{equation}\label{eq:generic-F-upper}
        F\leq2H(1-\eta).
\end{equation}

\medskip
\noindent\emph{Step 5: computation of
$r_2-\|V_2\|^2$.}
Using \eqref{eq:E1-nonab} \eqref{eq:E2-nonab}, \eqref{eq:generic-dimensionless}, and \eqref{eq:generic-LS}, we obtain
\begin{equation*}\label{eq:generic-E212}
|(E_2)_{21}|^2=\frac{b^2S}{2},\quad
 |(E_2)_{12}|^2=\frac{b^2S(S-H)^2}{2L^2},\quad
\end{equation*}
and
\begin{equation*}\label{eq:generic-E222}
  2|(E_2)_{22}|^2=\frac{b^2(S-2H)^2}{L},\quad
 -2\re\{(E_2)_{12}\overline{(E_1)_{22}}\}=\frac{b^2S(S-H)}{L}.
\end{equation*}
Substituting these four expressions into \eqref{eq:r2-general} obtains
\begin{equation}\label{eq:generic-r2-expanded}
\frac{r_2}{b^2}
=\frac S2-\frac{S(S-H)^2}{2L^2}
 -\frac{(S-2H)^2}{L}
 +\frac{S(S-H)}{L}.
\end{equation}
On the other hand, by \eqref{eq:generic-V2-norm}, we have
\begin{equation}\label{eq:generic-V2-SL}
        \frac{\|V_2\|^2}{b^2}
        =\frac{2H(S-2H)}{L}.
\end{equation}
Subtracting \eqref{eq:generic-V2-SL} from
\eqref{eq:generic-r2-expanded} and multiplying by $2L^2$ gives
\begin{align*}
2L^2\frac{r_2-\|V_2\|^2}{b^2}
= -S\{S^2-L^2+H^2-2H(S+L).
\end{align*}
Now
\[
S^2-L^2=(U+\eta)^2-(1-\eta)^2=F,
\qquad
S+L=U+1.
\]
Therefore
\begin{equation}\label{eq:generic-r2-difference}
 \frac{r_2-\|V_2\|^2}{b^2}
 =-\frac{(U+\eta)\{F+H^2-2H(U+1)\}}
 {2(1-\eta)^2}.
\end{equation}

\medskip
\noindent\emph{Step 6: contradiction.}
By \eqref{eq:generic-diagonal-zero} and
\eqref{eq:r2-general},
\[
        r_2=\|v_2^2\|^2\leq\|V_2\|^2.
\]
Hence the left hand side of
\eqref{eq:generic-r2-difference} is non-positive.  Furthermore,
\eqref{eq:generic-H-bound} and \eqref{eq:generic-H-positive} imply
\[
        U+\eta\geq2H>0.
\]
Since the coefficient
$-(U+\eta)/\{2(1-\eta)^2\}$ in
\eqref{eq:generic-r2-difference} is strictly negative, we conclude
that
\begin{equation}\label{eq:generic-B-positive}
        F+H^2-2H(U+1)\geq0.
\end{equation}
On the other hand, \eqref{eq:generic-F-upper} gives
\begin{align*}
F+H^2-2H(U+1)
&\leq2H(1-\eta)+H^2-2H(U+1)\\
&=H^2-2H(U+\eta).
\end{align*}
Using again $U+\eta\geq2H$ and $H>0$, we obtain
\[
H^2-2H(U+\eta)
\leq H^2-4H^2=-3H^2<0.
\]
This contradicts \eqref{eq:generic-B-positive}.  Therefore the generic
case cannot admit constant Chern holomorphic sectional curvature.
\end{proof}

\subsection{The Chern case: the half-generic type}

Here
\[
        a=b=0,\qquad c\neq0,\qquad d'=c.
\]

\begin{proposition}\label{prop:half}
Let $(\mathfrak g,J,g)$ be a unimodular Hermitian Lie algebra of
complex dimension $n\geq3$ containing an abelian ideal
$\mathfrak a$ of real codimension two.  Assume that
$J\mathfrak a\neq\mathfrak a$, that
$\mathfrak g/\mathfrak a$ is non-abelian, and that the associated
complex structure is of half-generic type, equivalently
$a=b=0$ and $c\neq0$.  Then the Chern holomorphic sectional
curvature of $g$ cannot be constant.
\end{proposition}

\begin{proof}
By Lemma \ref{lem:curvature-support}, the constant value of $H^c$ is
zero.  In the half-generic case,
\[
a=b=0,\qquad c\neq0,
\]
and hence \eqref{eq:E1-nonab} gives
\[
E_1=
\begin{pmatrix}
0&\dfrac{\I c}{\sqrt2\,\delta'}\\[2mm]
0&0
\end{pmatrix}.
\]
In particular,
\[
(E_1)_{11}=(E_1)_{21}=0.
\]
Putting $\alpha=\beta=1$ in \eqref{eq:Q}, all terms involving the
first column of $E_1$ vanish, and \eqref{eq:Q-zero} becomes
\[
0=Q_{11}=V_1V_1^*+[Y_1,Y_1^*].
\]
Taking the trace gives
\[
0=\tr(V_1V_1^*)+\tr[Y_1,Y_1^*]=\|V_1\|^2,
\]
Thus $V_1=0$. In particular, $v_1^1=0$.  Formula
\eqref{eq:r1-general} and the entries of $E_1$ give
\begin{align*}
R^c_{1\bar111\bar1}
&=|(E_1)_{12}|^2-|(E_1)_{21}|^2
  -2|(E_1)_{11}|^2
  -2\re\{(E_1)_{21}\overline{(E_2)_{11}}\}\\
&=\left|\frac{\I c}{\sqrt2\,\delta'}\right|^2
=\frac{c^2}{2\delta'^2}>0,
\end{align*}
since $c\neq 0$ and $\delta'>0$. This contradicts $H^c=0$.  Hence the half-generic case cannot
occur.
\end{proof}

\subsection{The Chern case: the degenerate type}

Here $a=b=c=0$.

\begin{proposition}\label{prop:degenerate}
Let $(\mathfrak g,J,g)$ be a unimodular Hermitian Lie algebra of
complex dimension $n\geq3$ containing an abelian ideal
$\mathfrak a$ of real codimension two.  Assume that
$J\mathfrak a\neq\mathfrak a$, that
$\mathfrak g/\mathfrak a$ is non-abelian, and that the
complex structure $J$ is of degenerate type, equivalently
$a=b=c=0$.  If the Chern holomorphic sectional curvature $H^c$ is
constant $\ka$, then $\ka=0$ and $R^c=0$.
\end{proposition}

\begin{proof}
By Lemma \ref{lem:curvature-support}, $\kappa =0$. In the degenerate case $a=b=c=0$, \eqref{eq:E2-nonab} gives $(E_2)_{21}=0$. Moreover, formula \eqref{eq:E1-nonab} gives $E_1=0$.  Hence \eqref{eq:Q-zero} with $\alpha=\beta=1$ becomes
\[
0=Q_{11}=V_1V_1^*+[Y_1,Y_1^*].
\]
Taking the trace yields $\|V_1\|^2=0$, and therefore $V_1=0$.
Substitution back gives $[Y_1,Y_1^*]=0$, so $Y_1$ is normal.  By Lemma \ref{lem:commutator-nilpotent} , $Y_1$ is also nilpotent.  Thus $Y_1=0$ and $D_1=0$.

Formula \eqref{eq:curvature-matrix} now gives
\begin{equation}\label{R2}
\mathcal R_{1\bar1}=\mathcal R_{1\bar2}
=\mathcal R_{2\bar1}=0.
\end{equation}
Indeed, $D_1=E_1=0$, and in the mixed cases the only remaining scalar
coefficient is $(E_2)_{21}=0$ or its conjugate.

Let $X=z_1e_1+z_2e_2+w=\sum_{k=1}^nX^ke_k$, where
$w\in\mathcal W:=(\mathfrak a_J)^{1,0}
=\operatorname{span}_{\C}\{e_3,\ldots,e_n\}$. By Lemma \ref{lem:curvature-support} and \eqref{R2}, we have
\begin{equation}\label{eq:degenerate-quartic}
R^c_{X\bar X X\bar X}
=|z_2|^2\sum_{k,\ell=1}^n
R^c_{2\bar2k\bar\ell}X^k\overline{X^\ell}.
\end{equation}
Since $H^c=0$, the left hand side vanishes for every $X$.  Define
\[
q(X)=\sum_{k,\ell=1}^n
R^c_{2\bar2k\bar\ell}X^k\overline{X^\ell}.
\]
For $z_2\neq0$, equation \eqref{eq:degenerate-quartic} gives $q(X)=0$.
The set $\{z_2\neq0\}$ is dense and $q$ is continuous, hence
$q\equiv0$.  Explicitly, for a vector $X_0$ with $z_2=0$, the sequence
$X_m=X_0+m^{-1}e_2$ lies in $\{z_2\neq0\}$ and converges to $X_0$;
therefore $q(X_0)=\lim_mq(X_m)=0$.

Let
$
B(U,V)=\sum_{k,\ell=1}^n
R^c_{2\bar2k\bar\ell}U^k\overline{V^\ell}.
$
Then $q(X)=B(X,X)$.  The Hermitian polarization identity is
\begin{align*}
B(U,V)=\frac14\{&q(U+V)-q(U-V)+\I q(U+\I V)-\I q(U-\I V)\}.
\end{align*}
Since $q\equiv0$, it follows that $B\equiv0$.  Taking $U=e_k$ and
$V=e_\ell$ yields $R^c_{2\bar2k\bar\ell}=0$ for all $k,\ell$, so
$\mathcal R_{2\bar2}=0$.  Together with \eqref{R2} and Lemma~\ref{lem:curvature-support}, this gives $R^c=0$.
\end{proof}

\subsection{Canonical connections with $\chi>0$}

Recall that for $(r,s)\in \Omega$, $\chi$ is defined by (\ref{eq:canonical-sym-identity}) as
$$ \chi
=
t^2+\frac{s^2}{4}
=
\frac{(1-r+rs)^2+s^2}{4}.$$
So $\chi \geq 0$, and $\chi =0$ only when $(r,s)=(1,0)$, or equivalently when the connection is $D^1_0=\nabla^c$.
Throughout this subsection we assume $\chi>0$ and use the torsion formulas
\eqref{eq:nonab-torsion-1}--\eqref{eq:nonab-u}.  The following
statement is the starting point for both the zero and negative
constant cases.

\begin{proposition}\label{prop:nonab-sign-reduction}
Let $(\mathfrak g,J,g)$ be a unimodular Hermitian Lie algebra of
complex dimension $n\geq3$ containing an abelian ideal
$\mathfrak a$ of real codimension two.  Assume that
$J\mathfrak a\neq\mathfrak a$,
$\mathfrak g/\mathfrak a$ is non-abelian, and $\chi>0$. If
$H^{D_s^r}=\ka$ is constant, then for every unit vector
$w\in\mathcal W=(\mathfrak a_J)^{1,0}$,
\begin{equation}\label{eq:nonab-W-reduction}
\ka=-\chi\left(
|\langle Z_1w,\overline w\rangle|^2
+|\langle Z_2w,\overline w\rangle|^2\right).
\end{equation}
Consequently, $\ka\leq0$.  If $\ka=0$, then $Z_1=Z_2=0$.
\end{proposition}

\begin{proof}
Fix a unit vector $w\in\mathcal W$.  Choose a unitary basis
$\{e_3,\ldots,e_n\}$ of $\mathcal W$ such that $e_3=w$, and keep
$e_1,e_2$ unchanged.  A unitary change of basis inside
$\mathcal W$ preserves the block form of the matrices $D_\alpha$,
the identities $D_p=0$ for $p\geq3$, and the torsion formulas
\eqref{eq:nonab-torsion-1}--\eqref{eq:nonab-torsion-3}.

Taking all four indices equal to $3$ in the definition
\eqref{eq:vhat-definition} gives
\begin{equation}\label{eq:nonab-vhat-diagonal}
\widehat v_{3\bar3 3\bar3}
=\sum_{a=1}^n|T^3_{3a}|^2
=\sum_{a=1}^n
\left|\left\langle T(w,e_a),\overline w\right\rangle\right|^2.
\end{equation}
For $a\geq3$, (\ref{eq:nonab-torsion-W}) implies
$T^q_{pa}=0$.  Hence only the terms $a=1,2$ remain in
\eqref{eq:nonab-vhat-diagonal}.  By the skew-symmetry of the Chern
torsion and \eqref{eq:nonab-torsion-1}--
\eqref{eq:nonab-torsion-2},
\[
\begin{aligned}
\left\langle T(w,e_1),\overline w\right\rangle
&=-\left\langle T(e_1,w),\overline w\right\rangle
=-\left\langle Z_1w,\overline w\right\rangle,\\
\left\langle T(w,e_2),\overline w\right\rangle
&=-\left\langle T(e_2,w),\overline w\right\rangle
=-\left\langle Z_2w,\overline w\right\rangle.
\end{aligned}
\]
Consequently,
\begin{equation}\label{eq:nonab-vhat-Z}
\widehat v_{w\bar w w\bar w}
=|\langle Z_1w,\overline w\rangle|^2
+|\langle Z_2w,\overline w\rangle|^2.
\end{equation}

By Lemma~\ref{lem:curvature-support},
$R^c_{w\bar w w\bar w}=0$.  Taking all four arguments equal to $w$
in \eqref{eq:master-equation}, and using
\eqref{eq:nonab-vhat-Z}, proves \eqref{eq:nonab-W-reduction}.
Since $\chi>0$, the right-hand side is non-positive.

If $\ka=0$, then
\[
\langle Z_\alpha w,\overline w\rangle=0,
\qquad \alpha=1,2,
\]
for every unit vector $w\in\mathcal W$, and hence for every
$w\in\mathcal W$ by homogeneity.  For each $\alpha\in\{1,2\}$,
define
$$B_\alpha(u,v):=\langle Z_\alpha u,\overline v\rangle,\qquad u,v\in\mathcal W.$$
Then $B_\alpha$ is sesquilinear and
$B_\alpha(w,w)=0$ for all $w\in\mathcal W$.  The complex
polarization identity implies that $B_\alpha\equiv0$, and therefore
$Z_\alpha=0$.  Thus $Z_1=Z_2=0.$
\end{proof}

\subsection{The $H^{D_s^r}=0$ case}

\begin{proposition}\label{prop:nonab-zero-canonical}
Let $(\mathfrak g,J,g)$ be a unimodular Hermitian Lie algebra of
complex dimension $n\geq3$ containing an abelian ideal
$\mathfrak a$ of real codimension two.  Assume that
$J\mathfrak a\neq\mathfrak a$ and that
$\mathfrak g/\mathfrak a$ is non-abelian.  Let $D_s^r$ be a
canonical metric connection with $\chi>0$.  Then
$H^{D_s^r}$ cannot vanish identically.
\end{proposition}

\begin{proof}
Assume $H^{D_s^r}=0$.  Proposition~\ref{prop:nonab-sign-reduction}
gives $Z_1=Z_2=0$.  Thus $Y_1=Y_1^*$; since
$[Y_1,Y_2]=\lambda Y_1$ with $\lambda\neq0$, by Lemma \ref{lem:commutator-nilpotent}, the matrix $Y_1$ is
nilpotent, and hence $Y_1=0$.  The identity $Z_2=0$ now gives
$Y_2=Y_2^*$ by \eqref{eq:nonab-Z-definitions}.  Taking the trace in
\eqref{eq:unimod-nonab} therefore yields
$2\sigma-d'-c=0$.

\medskip
\noindent\emph{The generic case.}
In this case, \eqref{eq:nonab-three-types} gives $d'=-c-2\sigma$, hence $0=2\sigma-d'-c=4\sigma$, contrary to $\sigma\neq0$.

\medskip
\noindent\emph{The half-generic case.}
In this case, \eqref{eq:nonab-three-types} gives $a=b=0$, $c\neq0$ and $d'=c$.
The entries in \eqref{eq:E2-nonab} satisfy
$(E_2)_{21}=(E_1)_{22}=0$ and
\begin{equation*}
(E_2)_{12}=\frac{\I c'+\delta(c+\sigma)}
{\sqrt2(\delta')^2},
\qquad
(E_2)_{22}=\frac{\I c}{\sqrt2\,\delta'}.
\end{equation*}
Consequently, by \eqref{eq:r2-general}, $r_2=-|(E_2)_{12}|^2-c^2/(\delta')^2<0$ because $c\neq0$ and
$\delta'>0$, and hence
\begin{equation}\label{eq:half-generic-R2222-negative}
R^c_{2\bar2 2\bar2}=r_2-\|v_2^2\|^2<0
\end{equation}
by \eqref{eq:r2-general}. On the other hand, the $D_s^r$-holomorphic sectional curvature is
assumed to vanish.  Taking $i=j=k=\ell=2$ in
\eqref{eq:master-equation}, and using $ \widehat R^c_{2\bar2 2\bar2}=R^c_{2\bar2 2\bar2},\,\widehat v_{2\bar2 2\bar2}=\sum_{a=1}^n|T^2_{2a}|^2,$ we obtain
\[
R^c_{2\bar2 2\bar2}
=
\chi\sum_{a=1}^n|T^2_{2a}|^2
\geq0,
\]
because $\chi>0$.  This contradicts
\eqref{eq:half-generic-R2222-negative}.  Hence the half-generic
case is impossible.

\medskip
\noindent\emph{The degenerate case.}
In this case, \eqref{eq:nonab-three-types} gives $a=b=c=0.$ Substituting these identities into \eqref{eq:E1-nonab}, we obtain $E_1=0.$ In particular, $(E_1)_{11}=(E_1)_{12}=(E_1)_{21}=0.$ It follows from the definition of $r_1$ in \eqref{eq:r1-general} that $r_1=0.$ Therefore, by \eqref{eq:r1-general}, \begin{equation}\label{eq:degenerate-R1111-nonpositive}
R^c_{1\bar1 1\bar1}
=
-\|v_1^1\|^2
\leq0.
\end{equation}
We next compute the torsion component $T^1_{12}$.  By
\eqref{torsion}, $T^1_{12}
=
-C^1_{12}-D^1_{12}+D^1_{21}.$ The identity \eqref{eq:nonab-C12} gives $C^1_{12}=-\frac{\I\sigma}{\sqrt2\,\delta'}.$ Moreover, from \eqref{eq:E2-nonab} and \eqref{eq:E1-nonab},
\[
D^1_{12}
=
(E_2)_{11}
=
\frac{\I(c-\sigma-b\delta^2)}
{\sqrt2\,\delta'},
\qquad
D^1_{21}
=
(E_1)_{21}
=
\frac{\I b\delta'}{\sqrt2}.
\]
Consequently,
\begin{align}
T^1_{12}
&=
\frac{\I\sigma}{\sqrt2\,\delta'}
-\frac{\I(c-\sigma-b\delta^2)}
{\sqrt2\,\delta'}
+\frac{\I b\delta'}{\sqrt2}\notag=
\frac{\I(b+2\sigma-c)}
{\sqrt2\,\delta'}.
\label{eq:nonab-mu1}
\end{align}
Here we used $(\delta')^2=1-\delta^2$.  Since $b=c=0$ in the
degenerate case, $$T^1_{12}=\frac{\sqrt2\,\I\sigma}{\delta'}\neq0,$$ because $\sigma\neq0$ and $\delta'>0$. On the other hand, the holomorphic sectional curvature of
$D_s^r$ is assumed to vanish identically.  Taking
$i=j=k=\ell=1$ in \eqref{eq:master-equation}, and using $\widehat v_{1\bar1 1\bar1}=\sum_{a=1}^n|T^1_{1a}|^2,$
we obtain $R^c_{1\bar1 1\bar1}=\chi\sum_{a=1}^n|T^1_{1a}|^2.$ Since $\chi>0$ and $T^1_{12}=T^1_{12}\neq0$, it follows that
\begin{equation}\label{eq:degenerate-R1111-positive}
R^c_{1\bar1 1\bar1}
\geq
\chi|T^1_{12}|^2
>0.
\end{equation}
This contradicts
\eqref{eq:degenerate-R1111-nonpositive}.  Hence the degenerate case
is impossible.

Together with the generic and half-generic cases considered above, this excludes all three non-abelian types.
\end{proof}

\subsection{The negative constant case}

\begin{lemma}\label{lem:nonab-negative-normal-form}
Let $(\mathfrak g,J,g)$ be a unimodular Hermitian Lie algebra of
complex dimension $n\geq3$ containing an abelian ideal
$\mathfrak a$ of real codimension two.  Assume that
$J\mathfrak a\neq\mathfrak a$ and that
$\mathfrak g/\mathfrak a$ is non-abelian. Assume $\chi>0$ and
$H^{D_s^r}=\ka<0$, and put $\mathcal W:=(\mathfrak a_J)^{1,0}
=\operatorname{span}_{\C}\{e_3,\ldots,e_n\}$ and
$m:=\dim_{\C}\mathcal W=n-2$.  Then there is
$\beta\in\R\setminus\{0\}$ such that
\begin{equation}\label{eq:nonab-negative-normal-form}
Z_1=0,
\qquad Z_2=2\I\beta I_m,
\qquad \ka=-4\chi\beta^2,
\end{equation}
and
\begin{equation}\label{eq:nonab-beta}
Y_1=0,
\qquad
\beta=\frac{2\sigma-d'-c}{2\sqrt2\,m\delta'}.
\end{equation}
\end{lemma}

\begin{proof}
Write $Y_1=H_1+\I S_1,\, Y_2=H_2+\I S_2,$ where $H_\alpha=\frac12(Y_\alpha+Y_\alpha^*),\,S_\alpha=\frac1{2\I}(Y_\alpha-Y_\alpha^*)$
are Hermitian matrices.  Since $\zeta=\I \frac{\delta}{\delta'}$, the definitions of
$Z_1$ and $Z_2$ give
\[
Z_1
=
Y_1-Y_1^*
=
2\I S_1,\quad Z_2=Y_2-Y_2^*+2\zeta Y_1^*=2\frac{\delta}{\delta'}S_1+2\I(S_2+\frac{\delta}{\delta'}H_1).
\]

For every unit vector $w\in\mathcal W$, the matrices $S_1$ and
$S_2+\frac{\delta}{\delta'}H_1$ are Hermitian, and hence $\langle S_1w,\overline w\rangle\in\R,\,\langle(S_2+\frac{\delta}{\delta'}H_1)w,\overline w\rangle\in\R.$ Therefore
\[
\left|
\langle Z_1w,\overline w\rangle
\right|^2
=
4\langle S_1w,\overline w\rangle^2,
\quad \left|
\langle Z_2w,\overline w\rangle
\right|^2
={}
4\frac{\delta^2}{\delta'^2}\langle S_1w,\overline w\rangle^2
+
4\langle(S_2+\frac{\delta}{\delta'}H_1)w,\overline w\rangle^2.
\]
Since
\[
1+\frac{\delta^2}{(\delta')^2}
=
\frac1{(\delta')^2},
\]
equation \eqref{eq:nonab-W-reduction} becomes
\begin{equation}\label{eq:nonab-circle}
\left\langle
\frac{S_1}{\delta'}w,\overline w
\right\rangle^2
+
\left\langle
(S_2+\frac{\delta}{\delta'}H_1)w,\overline w
\right\rangle^2
=
-\frac{\ka}{4\chi}
\end{equation}
for every unit vector $w\in\mathcal W$.

Apply Lemma~\ref{lem:circle-rigidity} to the Hermitian matrices $A:=\frac{S_1}{\delta'},\,B:=S_2+\frac{\delta}{\delta'}H_1.$
There exist $\alpha_0,\beta\in\R$ such that
\begin{equation}\label{eq:nonab-S-scalar}
\frac{S_1}{\delta'}
=
\alpha_0I_m,
\qquad
S_2+\frac{\delta}{\delta'}H_1
=
\beta I_m.
\end{equation}
Recall that
\[
[Y_1,Y_2]=\lambda Y_1,
\qquad
\lambda=\frac{\I\sigma}{\sqrt2\,\delta'}\neq0.
\]
Taking traces gives $0=\tr[Y_1,Y_2]=\lambda\tr(Y_1),$ and hence
$\tr(Y_1)=0.$ By \eqref{eq:nonab-Z-definitions}, $\tr Z_1=\tr(Y_1)-\overline{\tr(Y_1)}=0.$ On the other hand, the first identity in
\eqref{eq:nonab-S-scalar} gives
$
Z_1
=
2\I S_1
=
2\I\delta'\alpha_0I_m.
$
Thus
\[
0
=
\tr Z_1
=
2\I m\delta'\alpha_0.
\]
Since $m=n-2\geq1$ and $\delta'>0$, it follows that $\alpha_0=0.$ Therefore
\[
S_1=0
\qquad\text{and}\qquad
Z_1=0.
\]
Substituting these identities into \eqref{eq:nonab-circle}, using \eqref{eq:nonab-S-scalar} and the fact that $w$ is a unit vector, we obtain $\beta^2=-\frac{\ka}{4\chi}.$ Hence $\ka=-4\chi\beta^2.$ Since $\ka<0$ and $\chi>0$, one has $\beta\neq0$.  Since $S_1=0$, we obtain
$$Z_2 = 2\frac{\delta}{\delta'}S_1+2\I(S_2+\frac{\delta}{\delta'}H_1) = 2\I\beta I_m$$
by using the second identity in \eqref{eq:nonab-S-scalar}. This proves \eqref{eq:nonab-negative-normal-form}.

On the one hand, the matrix $Y_1=H_1$ is Hermitian by $S_1=0$. On the other hand,
\[
[Y_1,Y_2]=\lambda Y_1,
\qquad
\lambda\neq0.
\]
By Lemma~\ref{lem:commutator-nilpotent}, $Y_1$ is nilpotent.  A
Hermitian nilpotent matrix is zero, and therefore
\[
Y_1=0.
\]
Hence $Z_2=Y_2-Y_2^*$ by \eqref{eq:nonab-Z-definitions}. Taking traces and using \eqref{eq:unimod-nonab}, we find
\[
\begin{aligned}
2\I m\beta=\tr Z_2=\tr(Y_2)-\overline{\tr(Y_2)}=
\frac{\I}{\sqrt2\,\delta'}
(2\sigma-d'-c).
\end{aligned}
\]
Thus
\[
\beta
=
\frac{2\sigma-d'-c}
{2\sqrt2\,m\delta'}.
\]
This proves \eqref{eq:nonab-beta}.
\end{proof}

\begin{lemma}\label{lem:nonab-negative-reduction}
Let $(\mathfrak g,J,g)$ be a unimodular Hermitian Lie algebra of
complex dimension $n\geq3$ containing an abelian ideal
$\mathfrak a$ of real codimension two.  Assume that
$J\mathfrak a\neq\mathfrak a$, that
$\mathfrak g/\mathfrak a$ is non-abelian, and that a canonical
metric connection $D_s^r$ with $\chi>0$ has constant holomorphic
sectional curvature $H^{D_s^r}=\ka<0$.  Put $m=n-2$ and use the
normal form $Z_1=0$ and $Z_2=2\I\beta I_m$ from
Lemma~\ref{lem:nonab-negative-normal-form}.  Let $u$ be the vector
defined in \eqref{eq:nonab-u}.  Then
\[
u=v_2^1=v_2^2=0,
\qquad
v_1^2=2\zeta v_1^1.
\]
Consequently,
\begin{equation}\label{eq:nonab-V-negative}
V_1=(v_1^1,v_1^2),
\qquad
v_1^2=2\zeta v_1^1,
\qquad
V_2=0.
\end{equation}
\end{lemma}

\begin{proof}
By Lemma~\ref{lem:nonab-negative-normal-form}, there exists
$\beta\in\R\setminus\{0\}$ such that
$Z_1=0$ and $Z_2=2\I\beta I_m$.  We first consider
$\widehat v_{\alpha\bar p q\bar r}$, where
$\alpha\in\{1,2\}$ and $p,q,r\geq3$.

For every $\alpha\in\{1,2\}$ and $p,q,r\geq3$,
Lemma~\ref{lem:curvature-support} and
\eqref{eq:symmetrization} give
$\widehat R^c_{\alpha\bar p q\bar r}=0$.  Since $\delta_{\alpha p}\delta_{qr}
+\delta_{\alpha r}\delta_{qp}=0$,  the formula \eqref{eq:master-equation} gives
$-\chi\widehat v_{\alpha\bar p q\bar r}=0$.  Since $\chi>0$, we
conclude that
$\widehat v_{\alpha\bar p q\bar r}=0$.

Taking $i=\alpha$, $j=p$, $k=q$, and $\ell=r$ in
\eqref{eq:vhat-definition}, we obtain
\[
\begin{aligned}
4\widehat v_{\alpha\bar p q\bar r}
=\sum_a\bigl(T^p_{\alpha a}\overline{T^q_{ra}}
+T^p_{qa}\overline{T^\alpha_{ra}}+T^r_{\alpha a}\overline{T^q_{pa}}
+T^r_{qa}\overline{T^\alpha_{pa}}
\bigr).
\end{aligned}
\]
By \eqref{eq:nonab-torsion-1}--\eqref{eq:nonab-torsion-W},
the identities $Z_1=0$ and $Z_2=2\I\beta I_m$, and the
skew-symmetry of the torsion in its lower indices, the summands
with $a=1$ and $a\geq3$ vanish.  Indeed, when $a=1$, each product
contains a factor of the form $T^q_{p1}=-T^q_{1p}=0$, whereas for $a\geq3$, each product
contains a factor of the form $T^q_{pr}=0$.  Thus only the
summand with $a=2$ can contribute. More precisely, \eqref{eq:nonab-torsion-2},
\eqref{eq:nonab-torsion-3}, and the definition
$u_p=T^p_{12}$ give
\[
T^p_{\alpha2}=\delta_{\alpha1}u_p,\qquad
T^p_{q2}=-2\I\beta\delta_{pq},\qquad
T^\alpha_{r2}=-(v_2^\alpha)_r.
\]
Using these identities and their analogues with $p$ and $r$
interchanged, we obtain
\begin{equation}\label{eq:nonab-three-W}
\begin{aligned}
4\widehat v_{\alpha\bar p q\bar r}
=2\I\beta\bigl[\delta_{\alpha1}
 (u_p\delta_{qr}+u_r\delta_{qp})+\delta_{pq}\overline{(v_2^\alpha)_r}
 +\delta_{rq}\overline{(v_2^\alpha)_p}
\bigr].
\end{aligned}
\end{equation}
Fix $p\geq3$ and set $q=r=p$.  Taking $\alpha=2$ in
\eqref{eq:nonab-three-W} gives
$0=4\I\beta\overline{(v_2^2)_p}$, and hence
$(v_2^2)_p=0$ by $\beta\in\R\setminus\{0\}$.  Taking $\alpha=1$ gives
$0=4\I\beta(u_p+\overline{(v_2^1)_p})$, so
$(v_2^1)_p=-\overline{u_p}$.  Since $p$ is arbitrary, we have
\begin{equation}\label{eq:nonab-v2-relations}
v_2^2=0,
\qquad
v_2^1=-\overline u.
\end{equation}

Fix $p\geq3$.  By
Lemma~\ref{lem:curvature-support} and
\eqref{eq:symmetrization}, one has
$\widehat R^c_{p\bar1p\bar1}=0$. Since $\delta_{p1}=0$,
formula \eqref{eq:master-equation} gives
$-\chi\widehat v_{p\bar1p\bar1}=0$.  Since $\chi>0$, it follows
that
\[
\widehat v_{p\bar1p\bar1}=0.
\]

Taking $i=k=p$ and $j=\ell=1$ in
\eqref{eq:vhat-definition}, we have
\[
\widehat v_{p\bar1p\bar1}
=
\sum_aT^1_{pa}\overline{T^p_{1a}}.
\]
The term with $a=1$ vanishes because $T^p_{11}=0$, while the terms
with $a\geq3$ vanish by \eqref{eq:nonab-torsion-1} and $Z_1=0$.
Thus only $a=2$ contributes.  By skew-symmetry of the torsion,
\eqref{eq:nonab-torsion-3}, and the definition of $u$, we have
\[
T^1_{p2}=-T^1_{2p}=-(v_2^1)_p,
\qquad
T^p_{12}=u_p.
\]
Consequently,
\[
\widehat v_{p\bar1p\bar1}
=
-(v_2^1)_p\overline{u_p}.
\]
Using \eqref{eq:nonab-v2-relations}, namely
$(v_2^1)_p=-\overline{u_p}$, we obtain
\[
0
=
\widehat v_{p\bar1p\bar1}
=
\overline{u_p}^{\,2}.
\]
Hence $u_p=0$.  Since $p\geq3$ is arbitrary, $u=0$, and
\eqref{eq:nonab-v2-relations} then gives $v_2^1=0$.

Finally, \eqref{eq:nonab-u} gives
$\overline{v_1^2}=-2\zeta\overline{v_1^1}$, where  $\zeta=\frac{\I\delta}{\delta'}$.  Since
$\overline\zeta=-\zeta$, taking complex conjugates gives
$v_1^2=2\zeta v_1^1$.  Consequently,
\[
u=0,
\qquad
v_2^1=v_2^2=0,
\qquad
v_1^2=2\zeta v_1^1.
\]
This completes the proof.
\end{proof}

For $\alpha\in\{1,2\}$, set
\begin{equation}\label{eq:nonab-auxiliary-notation}
\mu_\alpha:=T^\alpha_{12},
\qquad
B_0:=\frac{b\delta'}{\sqrt2},
\qquad
D_0:=\frac{d'+b\delta^2}{\sqrt2\,\delta'},
\qquad
\xi:=\frac{b\delta+\I a}{\sqrt2}.
\end{equation}

\begin{lemma}\label{lem:nonab-negative-Q}
Let $(\mathfrak g,J,g)$ be a unimodular Hermitian Lie algebra of
complex dimension $n\geq3$ containing an abelian ideal
$\mathfrak a$ of real codimension two.  Assume that
$J\mathfrak a\neq\mathfrak a$, that
$\mathfrak g/\mathfrak a$ is non-abelian, and that a canonical
metric connection $D_s^r$ with $\chi>0$ has constant holomorphic
sectional curvature $H^{D_s^r}=\ka<0$.  Let $m=n-2$, and adopt the
notation and relations
\[
Z_1=0,\qquad
Z_2=2\I\beta I_m,\qquad
V_1=(v_1^1,v_1^2),\qquad
v_1^2=2\zeta v_1^1,\qquad
V_2=0
\]
from Lemmas~\ref{lem:nonab-negative-normal-form} and
\ref{lem:nonab-negative-reduction}.  Then the lower-right Chern
curvature blocks are given by
\begin{equation}\label{eq:nonab-Chern-Q-negative}
\begin{aligned}
Q_{11}
&=v_1^1(v_1^1)^*+v_1^2(v_1^2)^*-2B_0\beta I_m,\\
Q_{22}
&=-2D_0\beta I_m,\\
Q_{12}
&=-2\I\beta\xi I_m.
\end{aligned}
\end{equation}
Moreover, the $(\alpha,\bar{\beta},p,\bar{q})$-components of
\eqref{eq:master-equation}, where
$\alpha,\beta\in\{1,2\}$ and $p,q\geq3$, are equivalent to the
following matrix identities:
\begin{align}
Q_{11}
&=\chi\left\{
v_1^1(v_1^1)^*
-\bigl(8\beta^2
+4\beta\operatorname{Im}\mu_1\bigr)I_m
\right\},
\label{eq:nonab-Q11-negative}\\
Q_{22}
&=\chi\left\{
v_1^2(v_1^2)^*-4\beta^2I_m
\right\},
\label{eq:nonab-Q22-negative}\\
Q_{12}
&=\chi\left\{
v_1^2(v_1^1)^*
+2\I\beta\mu_2I_m
\right\}.
\label{eq:nonab-Q12-negative}
\end{align}
\end{lemma}

\begin{proof}
The expressions for $E_1$ and $E_2$ in
\eqref{eq:E1-nonab} and \eqref{eq:E2-nonab} give
\[
(E_1)_{21}=\I B_0,\qquad
(E_2)_{22}=\I D_0,\qquad
(E_1)_{22}=-\xi,\qquad
\overline{(E_2)_{21}}=\xi.
\]
By Lemma~\ref{lem:nonab-negative-normal-form}, Lemma~\ref{lem:nonab-negative-reduction} and \eqref{eq:nonab-Z-definitions}, we have
\[
Y_1=0,\qquad
Y_2-Y_2^*=2\I\beta I_m,\qquad
V_2=0,\qquad
u=0.
\]
In particular, $Y_2^*=Y_2-2\I\beta I_m$, and hence
$[Y_2,Y_2^*]=0$. Therefore \eqref{eq:Q} gives
\begin{align*}
Q_{11}
&=V_1V_1^*-(E_1)_{21}Y_2^*
 -\overline{(E_1)_{21}}Y_2
 =V_1V_1^*-2B_0\beta I_m,\\
Q_{22}
&=-(E_2)_{22}Y_2^*
 -\overline{(E_2)_{22}}Y_2
 =-2D_0\beta I_m,\\
Q_{12}
&=-(E_1)_{22}Y_2^*
 -\overline{(E_2)_{21}}Y_2
 =-2\I\beta\xi I_m.
\end{align*}
This proves \eqref{eq:nonab-Chern-Q-negative}.

We next compute
$\widehat v_{\alpha\bar\beta p\bar q}$ for
$\alpha,\beta\in\{1,2\}$ and $p,q\geq3$.  For every
$\gamma\in\{1,2\}$, equations
\eqref{eq:nonab-torsion-1}--\eqref{eq:nonab-torsion-3},
the above lemma, and the skew-symmetry of the torsion give
\[
T^q_{1p}=0,\qquad
T^q_{2p}=2\I\beta\delta_{pq},\qquad
T^\gamma_{1p}=(v_1^\gamma)_p,\qquad
T^\gamma_{2p}=0,
\]
together with $T^\gamma_{12}=\mu_\gamma$ and
$T^p_{12}=u_p=0$. Taking $i=\alpha$, $j=\beta$, $k=p$, and $\ell=q$ in
\eqref{eq:vhat-definition}, we obtain
\begin{equation}\label{eq:nonab-vhat-Q}
\begin{aligned}
4\widehat v_{\alpha\bar\beta p\bar q}
{}&=\sum_a\bigl(
T^\beta_{\alpha a}\overline{T^p_{qa}}
+T^\beta_{pa}\overline{T^\alpha_{qa}}+T^q_{\alpha a}\overline{T^p_{\beta a}}
+T^q_{pa}\overline{T^\alpha_{\beta a}}
\bigr)\\
&=
(v_1^\beta)_p\overline{(v_1^\alpha)_q}
+4\beta^2\delta_{\alpha2}\delta_{\beta2}\delta_{pq}+
2\I\beta
\bigl(
\delta_{\alpha1}\mu_\beta
-\delta_{\beta1}\overline{\mu_\alpha}
\bigr)\delta_{pq}.
\end{aligned}
\end{equation}
By the definition of $Q_{\alpha\beta}$, $(Q_{\alpha\beta})_{pq}
=R^c_{\alpha\bar\beta p\bar q}.$ Moreover, Lemma~\ref{lem:curvature-support} and
\eqref{eq:symmetrization} imply that $R^c_{\alpha\bar\beta p\bar q}=4\widehat R^c_{\alpha\bar\beta p\bar q}$. Thus the $(\alpha,\beta,p,q)$ component of
\eqref{eq:master-equation} gives
\[
(Q_{\alpha\beta})_{pq}
=
4\chi\widehat v_{\alpha\bar\beta p\bar q}
+
2\ka\delta_{\alpha\beta}\delta_{pq}=
4\chi\widehat v_{\alpha\bar\beta p\bar q}
-
8\chi\beta^2\delta_{\alpha\beta}\delta_{pq}.
\]
Here we use $\ka=-4\chi\beta^2$ by Lemma \ref{lem:nonab-negative-normal-form}.

Taking $(\alpha,\beta)=(1,1)$ in
\eqref{eq:nonab-vhat-Q}, we have
$\delta_{\alpha2}\delta_{\beta2}=0$ and
$\delta_{\alpha1}=\delta_{\beta1}=1$.  Hence, for every
$p,q\geq3$,
\begin{equation}\label{eq:nonab-vhat-11}
4\widehat v_{1\bar1p\bar q}=
(v_1^1)_p\overline{(v_1^1)_q}
-4\beta\operatorname{Im}\mu_1\,\delta_{pq},
\end{equation}
where we used $\mu_1-\overline{\mu_1}=2\I\operatorname{Im}\mu_1$.

On the other hand, the $(1,1,p,q)$ component of
\eqref{eq:master-equation} gives
\[
(Q_{11})_{pq}
=
4\chi\widehat v_{1\bar1p\bar q}
+
2\ka\delta_{pq}.
\]
Using \eqref{eq:nonab-vhat-11} and
$\ka=-4\chi\beta^2$, we obtain
\[
\begin{aligned}
(Q_{11})_{pq}
&=
\chi
\left\{
(v_1^1)_p\overline{(v_1^1)_q}
-
4\beta\operatorname{Im}\mu_1\,\delta_{pq}
\right\}
-
8\chi\beta^2\delta_{pq}=
\chi
\left\{
(v_1^1)_p\overline{(v_1^1)_q}
-
\bigl(8\beta^2+4\beta\operatorname{Im}\mu_1\bigr)\delta_{pq}
\right\}.
\end{aligned}
\]
Since the $(p,q)$-entry of
$v_1^1(v_1^1)^*$ is
$(v_1^1)_p\overline{(v_1^1)_q}$, the previous componentwise
identity is equivalent to
\[
Q_{11}
=
\chi\left\{
v_1^1(v_1^1)^*
-
\bigl(8\beta^2+4\beta\operatorname{Im}\mu_1\bigr)I_m
\right\}.
\]
Similarly, the choices $(\alpha,\beta)=(2,2)$, $(1,2)$ in
\eqref{eq:nonab-vhat-Q}, and using
$\mu_\alpha-\overline{\mu_\alpha}=2\I\operatorname{Im}\mu_\alpha$, we obtain
\[
Q_{22}
=
\chi\left\{
v_1^2(v_1^2)^*
-
4\beta^2I_m
\right\},
\qquad
Q_{12}
=
\chi\left\{
v_1^2(v_1^1)^*
+
2\I\beta\mu_2I_m
\right\},
\]
respectively.  These are precisely
\eqref{eq:nonab-Q11-negative}--\eqref{eq:nonab-Q12-negative}.
\end{proof}


\begin{lemma}\label{lem:nonab-m-one}
Under the hypotheses of Lemma~\ref{lem:nonab-negative-Q}, assume
that $n=3$.  Then $v_1^1=v_1^2=0.$
\end{lemma}

\begin{proof}
Since $m=1$, the quantities $v_1^1$, $v_1^2$, and $Y_2$ are
complex numbers.  Since $H^{D_s^r}=\ka$ is constant,
\[
\widehat R^{D_s^r}_{1\bar2 1\bar3}
=
\frac{\ka}{2}
\bigl(
\delta_{12}\delta_{13}
+
\delta_{13}\delta_{12}
\bigr)
=
0.
\]
Moreover, Lemma~\ref{lem:nonab-negative-reduction} gives
$u=V_2=0$, while $Z_1=0$. Substitution into \eqref{eq:vhat-definition}, together with
\eqref{eq:nonab-torsion-1}--\eqref{eq:nonab-u}, therefore gives
$\widehat v_{1\bar2 1\bar3}=0$. Hence $\widehat R^c_{1\bar2 1\bar3}=0$ by \eqref{eq:master-equation}. On the other hand, the upper-right block of
$\mathcal R_{1\bar2}$ is
$E_1V_2^*-V_2^*Y_1-\sum_{\gamma=1}^2
(E_1)_{\gamma2}V_\gamma^*$, which gives
$-(E_1)_{12}V_1^*$ because $Y_1=V_2=0$.  Since $m=1$ and
$V_1=(v_1^1,v_1^2)$, this gives
$R^c_{1\bar2 1\bar3}=-(E_1)_{12}\overline{v_1^1}$.
Furthermore, \eqref{eq:symmetrization} and
Lemma~\ref{lem:curvature-support} yield
$2\widehat R^c_{1\bar2 1\bar3}
=R^c_{1\bar2 1\bar3}$.  Hence
\[
0=2\widehat R^c_{1\bar2 1\bar3}
=
-(E_1)_{12}\overline{v_1^1}.
\]
Thus $(E_1)_{12}\overline{v_1^1}=0$.

Suppose that $v_1^1\neq0$.  Then
$(E_1)_{12}=0$.

\smallskip
\noindent
\emph{Generic type.}
Here $b\neq0$, $c=-a^2/b$, and $d'=-c-2\sigma$.  By
\eqref{eq:E1-nonab},
\[
(E_1)_{12}
=
\frac{-2a\delta+\I(c+b\delta^2)}
{\sqrt2\,\delta'}.
\]
Thus $(E_1)_{12}=0$ gives $a\delta=0$ and
$c+b\delta^2=0$.  Since $c=-a^2/b$, the latter identity becomes
$b^2\delta^2=a^2$, and hence $a=\delta=0$.  Consequently,
$c=0$, $d'=-2\sigma$, $\zeta=0$, and
$v_1^2=2\zeta v_1^1=0$.  Thus $V_1=(v_1^1,0)$.

Since $Y_1=V_2=0$, equation \eqref{eq:nonab-block-Jacobi}
gives $V_1E_2-Y_2V_1=\lambda V_1$, where $\lambda=\frac{\I\sigma}{\sqrt2\,\delta'}\neq0.$  Its first component is
$\bigl((E_2)_{11}-Y_2-\lambda\bigr)v_1^1=0$.  Since
$v_1^1\neq0$, we obtain $Y_2=(E_2)_{11}-\lambda$.  By
\eqref{eq:E2-nonab},
$(E_2)_{11}=-\I\sigma/(\sqrt2\,\delta')$, while
$\lambda=\I\sigma/(\sqrt2\,\delta')$.  Therefore
$Y_2=-2\I\sigma/(\sqrt2\,\delta')$.

Since $m=1$ and $Y_1=0$, the identity
$Z_2=2\I\beta I_m$ becomes
$Y_2-\overline{Y_2}=2\I\beta$, and hence
$\beta=-2\sigma/(\sqrt2\,\delta')$.  On the other hand,
\eqref{eq:nonab-beta}, together with $m=1$, $c=0$, and
$d'=-2\sigma$, gives
\[
\beta
=
\frac{2\sigma-d'-c}{2\sqrt2\,\delta'}
=
\frac{2\sigma}{\sqrt2\,\delta'}.
\]
Comparing the two expressions for $\beta$ gives
$4\sigma/(\sqrt2\,\delta')=0$, contradicting
$\sigma\neq0$ and $\delta'>0$.

\smallskip
\noindent
\emph{Half-generic type.}
In this case $a=b=0$ and $c\neq0$.  Hence
$(E_1)_{12}=\I c/(\sqrt2\,\delta')\neq0$ by \eqref{eq:E1-nonab}, contradicting
$(E_1)_{12}=0$.

\smallskip
\noindent
\emph{Degenerate type.}
Here $a=b=c=0$.  By \eqref{eq:E2-nonab},
\[
(E_2)_{11}=-\frac{\I\sigma}{\sqrt2\,\delta'},
\qquad
(E_2)_{21}=0,
\qquad
(E_2)_{22}=\frac{\I d'}{\sqrt2\,\delta'}.
\]
Moreover, under the assumption $v_1^1\neq0$,
Lemma~\ref{lem:nonab-negative-reduction} gives
$V_1=(v_1^1,2\zeta v_1^1)$, while
$Y_1=V_2=0$.  Hence \eqref{eq:nonab-block-Jacobi} gives
$V_1E_2-Y_2V_1=\lambda V_1$.

Its first component is
$\bigl((E_2)_{11}-Y_2-\lambda\bigr)v_1^1=0$.  Since
$v_1^1\neq0$ and
$\lambda=\I\sigma/(\sqrt2\,\delta')$, we obtain
$Y_2=-2\I\sigma/(\sqrt2\,\delta')$.  Since $m=1$ and $Y_1=0$,
the identity $Z_2=2\I\beta I_m$ becomes
$Y_2-\overline{Y_2}=2\I\beta$, and therefore
$\beta=-2\sigma/(\sqrt2\,\delta')$.  On the other hand,
\eqref{eq:nonab-beta}, with $m=1$ and $c=0$, gives
\[
-\frac{2\sigma}{\sqrt2\,\delta'}
=
\beta
=
\frac{2\sigma-d'}{2\sqrt2\,\delta'}.
\]
Hence $d'=6\sigma$.

The second component of
$V_1E_2-Y_2V_1=\lambda V_1$ is
\[
(E_2)_{12}
+2\zeta(E_2)_{22}
-2\zeta Y_2
=
2\zeta\lambda.
\]
Using $\zeta=\I\delta/\delta'$ and the preceding expressions for
$(E_2)_{22}$, $Y_2$, and $\lambda$, we obtain
\[
(E_2)_{12}
=
2\zeta\bigl(\lambda+Y_2-(E_2)_{22}\bigr)
=
\frac{2\delta(d'+\sigma)}
{\sqrt2\,(\delta')^2}.
\]
On the other hand, \eqref{eq:E2-nonab} gives
\[
(E_2)_{12}
=
\frac{\I c'+\delta(d'+\sigma)}
{\sqrt2\,(\delta')^2}.
\]
Comparing the two expressions yields
$\I c'=\delta(d'+\sigma)$.  Since $c',\delta,d'$, and $\sigma$
are real, the left-hand side is purely imaginary and the
right-hand side is real.  Hence
$c'=0$ and $\delta(d'+\sigma)=0$.  Since
$d'=6\sigma$ and $\sigma\neq0$, one has
$d'+\sigma=7\sigma\neq0$, and therefore $\delta=0$.  Thus
$\zeta=\frac{\I\delta}{\delta'}=0$ and $v_1^2=2\zeta v_1^1=0$.

Finally, comparing the two formulas for $Q_{22}$ in
\eqref{eq:nonab-Chern-Q-negative} and
\eqref{eq:nonab-Q22-negative}, and using $v_1^2=0$, gives
\[
-2\frac{d'}{\sqrt2\,\delta'}\beta
=
-4\chi\beta^2.
\]
Since $\beta\neq0$, it follows that
$d'/(\sqrt2\,\delta')=2\chi\beta$.  Substituting
$d'=6\sigma$ and
$\beta=-2\sigma/(\sqrt2\,\delta')$ gives
$6\sigma=-4\chi\sigma$, or equivalently
$6=-4\chi$, contradicting $\chi>0$ and $\sigma\neq0$.

All three cases lead to contradictions.  Therefore $v_1^1=0$.
Since $v_1^2=2\zeta v_1^1$, we also have $v_1^2=0$. This completes the proof.
\end{proof}

\begin{lemma}\label{lem:nonab-scalar-system}
Under the hypotheses of Lemma~\ref{lem:nonab-negative-Q}, one has
\begin{equation}\label{eq:nonab-scalar-system}
B_0
=
2\chi\bigl(2\beta+\operatorname{Im}\mu_1\bigr),
\qquad
D_0=2\chi\beta,
\qquad
\xi=-\chi\mu_2.
\end{equation}
\end{lemma}

\begin{proof}
Assume first that $m\geq2$.  Choose a unit vector
$w\in\C^m$ such that $(v_1^1)^*w=0$.  Since
$v_1^2=2\zeta v_1^1$, one also has $(v_1^2)^*w=0$.  Consequently,
\[
v_1^1(v_1^1)^*w
=
v_1^2(v_1^2)^*w
=
v_1^2(v_1^1)^*w
=
0.
\]

Comparing the two expressions for $Q_{11}w$ in
\eqref{eq:nonab-Chern-Q-negative} and
\eqref{eq:nonab-Q11-negative}, we obtain
\[
-2B_0\beta w
=
-\chi\bigl(
8\beta^2+4\beta\operatorname{Im}\mu_1
\bigr)w.
\]
Since $\beta\neq0$ and $w\neq0$, this gives
$B_0=2\chi(2\beta+\operatorname{Im}\mu_1)$.

Similarly, comparing the two expressions for $Q_{22}w$ gives
$-2D_0\beta w=-4\chi\beta^2w$, and hence
$D_0=2\chi\beta$.  Finally, comparison of the two expressions for
$Q_{12}w$ yields
$-2\I\beta\xi w=2\I\chi\beta\mu_2w$, so
$\xi=-\chi\mu_2$.

Suppose now that $m=1$.  Lemma~\ref{lem:nonab-m-one} gives
$v_1^1=v_1^2=0$.  Hence all rank-one terms in
\eqref{eq:nonab-Chern-Q-negative} and
\eqref{eq:nonab-Q11-negative}--\eqref{eq:nonab-Q12-negative}
vanish.  Comparing the resulting scalar identities, and using
$\beta\neq0$, gives exactly the three equations in
\eqref{eq:nonab-scalar-system}.
\end{proof}

\begin{proposition}\label{prop:nonab-negative-excluded}
Let $(\mathfrak g,J,g)$ be a unimodular Hermitian Lie algebra of
complex dimension $n\geq3$ containing an abelian ideal
$\mathfrak a$ of real codimension two.  Assume that
$J\mathfrak a\neq\mathfrak a$ and that
$\mathfrak g/\mathfrak a$ is non-abelian. Then the holomorphic sectional curvature of a canonical metric
connection $D_s^r$ with $\chi>0$ cannot be a negative constant.
\end{proposition}

\begin{proof}
Assume  that
$H^{D_s^r}=\ka<0$.  By
Lemma~\ref{lem:nonab-scalar-system}, the identities
\eqref{eq:nonab-scalar-system} hold. By \eqref{torsion},
\eqref{eq:nonab-C-D-relations},
\eqref{eq:E1-nonab}, and \eqref{eq:E2-nonab}, we have
\begin{equation}\label{eq:nonab-mu-formulas}
\mu_1
=
\frac{\I(b+2\sigma-c)}{\sqrt2\,\delta'},
\qquad
\mu_2
=
-\frac{\delta(d'+b+\sigma)+\I(a+c')}
{\sqrt2\,(\delta')^2}.
\end{equation}
Indeed, $\mu_1=-C^1_{12}-D^1_{12}+D^1_{21}$ and
$\mu_2=-C^2_{12}-D^2_{12}+D^2_{21}$.

\smallskip
\noindent
\emph{Generic type.}
Recall that in the generic type,
\[
b\neq0,\qquad
c=-\frac{a^2}{b},\qquad
c'=-\frac{a}{b}(c+\sigma),\qquad
d'=-c-2\sigma.
\]
Put $U:=a^2/b^2$, $H:=\sigma/b$, $\eta:=\delta^2$,
$L:=1-\eta=(\delta')^2>0$, and $S:=U+\eta$.  Then
\begin{equation}\label{eq:nonab-generic-relations}
c=-bU,\qquad
d'=b(U-2H),\qquad
c'=a(U-H).
\end{equation}
It follows from \eqref{eq:nonab-beta}, \eqref{eq:nonab-auxiliary-notation}, and \eqref{eq:nonab-generic-relations} that
\[
\beta=\frac{\sqrt2\,bH}{m\delta'},
\qquad
D_0=\frac{b(S-2H)}{\sqrt2\,\delta'}.
\]
Thus the identity $D_0=2\chi\beta$ in
\eqref{eq:nonab-scalar-system} becomes
\begin{equation}\label{eq:nonab-neg-generic-1}
S=2H\left(1+\frac{2\chi}{m}\right).
\end{equation}
Since $S\geq0$, $\chi>0$, and $m\geq1$, it follows that
$H\geq0$.  Moreover, $H=\sigma/b\neq0$, since $\sigma\neq0$ and
$b\neq0$.  Hence $H>0$.

By \eqref{eq:nonab-mu-formulas} and
\eqref{eq:nonab-generic-relations},
\[
\operatorname{Im}\mu_1
=
\frac{b+2\sigma-c}{\sqrt2\,\delta'}
=
\frac{b(1+U+2H)}{\sqrt2\,\delta'}.
\]
Substituting this identity, together with
$B_0=b\delta'/\sqrt2$ from
\eqref{eq:nonab-auxiliary-notation} and
$\beta=\sqrt2\,bH/(m\delta')$, into the first equation of
\eqref{eq:nonab-scalar-system}, we obtain
\[
\frac{b\delta'}{\sqrt2}
=
2\chi\left\{
\frac{2\sqrt2\,bH}{m\delta'}
+
\frac{b(1+U+2H)}{\sqrt2\,\delta'}
\right\}.
\]
Since $b\neq0$, multiplying by $\sqrt2\,\delta'/b$ and using
$L=(\delta')^2$ gives
\begin{equation}\label{eq:nonab-neg-generic-2}
L
=
2\chi\left(
1+U+2H+\frac{4H}{m}
\right).
\end{equation}
Using the second formula in \eqref{eq:nonab-mu-formulas},
\eqref{eq:nonab-generic-relations}, and the definition of $\xi$ in
\eqref{eq:nonab-auxiliary-notation}, we have
\[
a+c'=a(1+U-H),
\qquad
d'+b+\sigma=b(1+U-H),
\]
and hence
\[
\mu_2
=
-\frac{\delta(d'+b+\sigma)+\I(a+c')}
{\sqrt2\,(\delta')^2}
=
-\frac{1+U-H}{L}\,\xi.
\]
We claim that $\xi\neq0$.  Indeed, since
$\xi=(b\delta+\I a)/\sqrt2$ and $a,b,\delta\in\R$,
the equality $\xi=0$ would imply $a=0$ and $b\delta=0$.
As $b\neq0$, this gives $\delta=0$.  Thus $U=\eta=S=0$, and
\eqref{eq:nonab-neg-generic-1} would become
$-2H=4\chi H/m$, contradicting $H\neq0$ and $\chi>0$.
Therefore $\xi\neq0$.

The third equation in \eqref{eq:nonab-scalar-system} now gives
$\xi=-\chi\mu_2=\frac{\chi(1+U-H)}{L}\,\xi.$ Hence
\begin{equation}\label{eq:nonab-neg-generic-3}
L=\chi(1+U-H)
\end{equation}
by $\xi\neq0$.
Comparing \eqref{eq:nonab-neg-generic-2} and
\eqref{eq:nonab-neg-generic-3}, and dividing by $\chi>0$, we obtain
\[
2\left(1+U+2H+\frac{4H}{m}\right)=1+U-H,
\]
or equivalently
$0=1+U+(5+8/m)H$.  This contradicts $U\geq0$ and $H>0$.

\smallskip
\noindent
\emph{Half-generic type.}
Here $a=b=0$, $c\neq0$, and $d'=c$.  It follows from
\eqref{eq:nonab-beta}, \eqref{eq:nonab-mu-formulas}, and
\eqref{eq:nonab-auxiliary-notation} that
\[
\beta=\frac{\sigma-c}{\sqrt2\,m\delta'},
\qquad
\operatorname{Im}\mu_1
=\frac{2\sigma-c}{\sqrt2\,\delta'},
\qquad
B_0=0,
\qquad
D_0=\frac{c}{\sqrt2\,\delta'}.
\]
Thus the first two equations in
\eqref{eq:nonab-scalar-system} give
\begin{equation}\label{eq:nonab-neg-half}
\frac{2(\sigma-c)}m+2\sigma-c=0,
\qquad
c=\frac{2\chi}{m}(\sigma-c).
\end{equation}
The first equality gives
\[
\frac{c}{\sigma}
=
\frac{2(m+1)}{m+2}
>1.
\]
Thus $c$ and $\sigma$ have the same sign, whereas
\[
\sigma-c
=
-\frac{m}{m+2}\sigma
\]
has the opposite sign to $\sigma$, and hence also to $c$.  This
contradicts the second equality in
\eqref{eq:nonab-neg-half}, since $2\chi/m>0$ would require
$c$ and $\sigma-c$ to have the same sign.

\smallskip
\noindent
\emph{Degenerate type.}
Here $a=b=c=0$.  Equations \eqref{eq:nonab-beta},
\eqref{eq:nonab-mu-formulas}, and
\eqref{eq:nonab-auxiliary-notation} give
\[
\beta=\frac{2\sigma-d'}{2\sqrt2\,m\delta'},
\qquad
\operatorname{Im}\mu_1
=\frac{2\sigma}{\sqrt2\,\delta'},
\qquad
B_0=0,
\qquad
D_0=\frac{d'}{\sqrt2\,\delta'}.
\]
Since $B_0=0$, the first equation in
\eqref{eq:nonab-scalar-system}, together with $\chi>0$, gives
$2\beta+\operatorname{Im}\mu_1=0$.  Substituting the above
expressions for $\beta$ and $\operatorname{Im}\mu_1$, we obtain
\[
0=\frac{1}{\sqrt2\,\delta'}
\left(
\frac{2\sigma-d'}m+2\sigma
\right).
\]
Since $\delta'>0$, it follows that
$(2\sigma-d')/m+2\sigma=0$, and hence
$d'=2\sigma(m+1)$.

On the other hand, the second equation in
\eqref{eq:nonab-scalar-system} is $D_0=2\chi\beta$.  Substituting
the formulas for $D_0$ and $\beta$ gives
\[
d'=\frac{\chi}{m}(2\sigma-d').
\]
Using $d'=2\sigma(m+1)$, we find
\[
d'
=
\frac{\chi}{m}
\bigl(2\sigma-2\sigma(m+1)\bigr)
=
-2\chi\sigma.
\]
Thus
$2\sigma(m+1)=-2\chi\sigma$.  Since $\sigma\neq0$, division by
$2\sigma$ yields $\chi=-(m+1)$, contradicting $\chi>0$.

Thus each of the generic, half-generic, and degenerate types leads
to a contradiction.  Hence the assumption $H^{D_s^r}=\ka<0$ is impossible.  This completes the proof.
\end{proof}

\vspace{0.3cm}

\section{The $J{\mathfrak a} \neq {\mathfrak a}$ and ${\mathfrak g} / {\mathfrak a}$ abelian case}
\label{sec:abelian-quotient}

Throughout this section, we will assume that {\em  ${\mathfrak g}$ is a unimodular Lie algebra of real dimension $2n$, ${\mathfrak a}\subseteq {\mathfrak g}$ an abelian ideal of codimension $2$ with ${\mathfrak g} / {\mathfrak a}$  abelian, and $J$ a complex structure on ${\mathfrak g}$ with $J{\mathfrak a} \neq {\mathfrak a}$.} This case requires a different admissible-frame construction from the
one used in Section \ref{sec:nonabelian-quotient}.

Write ${\mathfrak g}'=[{\mathfrak g},{\mathfrak g}]$. Since ${\mathfrak g} / {\mathfrak a}$ is abelian, we have ${\mathfrak g}'\subseteq {\mathfrak a}$ so ${\mathfrak g}$ is $2$-step solvable. Let
\[
\mathfrak a_J=\mathfrak a\cap J\mathfrak a.
\]
Let $V=\mathfrak a_J^\perp\cap\mathfrak a, \,\mathfrak a=\mathfrak a_J\oplus V,$ and choose an orthonormal basis $\{x,y\}$ of $V$ for which the induced endomorphism $A_x$ of $V\simeq\mathfrak a/\mathfrak a_J$ satisfies
\begin{equation}\label{eq:trace-Ax-zero}
\tr(A_x)=0.
\end{equation}
Write
\[
e_1=\frac1{\sqrt2}(x-\I Jx),
\qquad
Y=\frac1{\sqrt2}(y-\I Jy)
=\I\delta e_1+\delta'e_2,
\]
where $\delta=\langle Jx,y\rangle \in(-1,1)$ and
$\delta'=\sqrt{1-\delta^2} \in(0,1]$.  Together with a unitary basis
$e_3,\ldots,e_n$ of $(\mathfrak a_J)^{1,0}$, this is the admissible
frame of \cite[Appendix B]{CaoZhengFV}.

By \cite[Appendix~B, Lemma~14]{CaoZhengFV}, the admissible frame in the abelian quotient case satisfies
\begin{equation}\label{eq:abelian-block-form}
D_\alpha=
\begin{pmatrix}E_\alpha&0\\V_\alpha&Y_\alpha\end{pmatrix},
\qquad \alpha=1,2,
\qquad
D_p=0\quad(p\geq3),
\end{equation}
Moreover,
\begin{equation}\label{eq:abelian-C12}
C^1_{12}=C^2_{12}=0.
\end{equation}
Taking $(i,k)=(1,2)$ in the second Jacobi identity in
\eqref{Jacobi}, we obtain
\[
[D_1,D_2]
=
-\sum_{r=1}^nC^r_{12}D_r.
\]
By \eqref{eq:abelian-C12}, the terms with $r=1,2$ vanish, while
\eqref{eq:abelian-block-form} gives $D_r=0$ for $r\geq3$.
Hence $[D_1,D_2]=0$.
Comparing the lower-right blocks yields
\begin{equation}\label{eq:abelian-Y-commute}
[Y_1,Y_2]=0.
\end{equation}
Comparing the lower-left blocks yields
\begin{equation}\label{eq:abelian-block-Jacobi}
V_1E_2+Y_1V_2-V_2E_1-Y_2V_1=0.
\end{equation}
Meanwhile, \cite{CaoZhengFV} gives, for
$\alpha,\beta\in\{1,2\}$, $p,q\geq3$, and $1\leq \ast \leq n$,
\begin{equation}\label{eq:abelian-C-D-relations}
\begin{aligned}
& C^{\ast}_{pq} = D^{\ast}_{\ast p} = C^{\alpha }_{\beta p} = D^p_{\alpha\beta} =0, \ \ \ \zeta:=\frac{\I\delta}{\delta'},\\
& C^{\ast}_{1p}= \overline{D^p_{\ast 1}} , \ \ \ C^{\ast}_{2p}= \overline{D^p_{\ast 2}} -2\zeta\overline{D^p_{\ast1}}, \ \ \ C^{\ast}_{12}= \overline{D^2_{\ast 1}} -  \overline{D^1_{\ast 2}} + 2\zeta\overline{D^1_{\ast1}},
\end{aligned}
\end{equation}
Write $V_\alpha=(v_\alpha^1,v_\alpha^2)$ and set
\begin{equation}\label{eq:abelian-Z-definitions}
Z_1=Y_1-Y_1^*,
\qquad
Z_2=Y_2-Y_2^*+2\zeta Y_1^*.
\end{equation}
We regard $Y_\alpha$ and $Z_\alpha$ as the corresponding
complex-linear endomorphisms of $\mathcal W$ in the chosen unitary
basis. Exactly as in the non-abelian case, but now using
\eqref{eq:abelian-C-D-relations}, the Chern torsion satisfies
\begin{align}
T^q_{1p}&=(Z_1)_{pq},
\label{eq:abelian-torsion-1}\\
T^q_{2p}&=(Z_2)_{pq},
\label{eq:abelian-torsion-2}\\
T^\beta_{\alpha p}&=(v_\alpha^\beta)_p.
\label{eq:abelian-torsion-3}
\end{align}
The $\mathcal W$-valued component of $T(e_1,e_2)$ is
\begin{equation}\label{eq:abelian-u}
u=(T^p_{12})_{p=3}^n
=-\overline{v_1^2}+\overline{v_2^1}
-2\zeta\overline{v_1^1}.
\end{equation}
These formulas are derived from
$T^j_{ik}=-C^j_{ik}-D^j_{ik}+D^j_{ki}$ and are not consequences of
the block form alone; the full identities in
\eqref{eq:abelian-C-D-relations} are essential.
  Therefore the purely algebraic results of
Subsection \ref{subsec:block-curvature} apply again.  In particular,
constant Chern holomorphic sectional curvature implies
\begin{equation}\label{eq:abelian-kappa-zero}
\ka=0,
\qquad
Q_{\alpha\beta}=0\quad(1\leq\alpha,\beta\leq2).
\end{equation}

Modulo $\mathfrak a_J$, write
\begin{equation}\label{eq:abelian-real-brackets}
[Jx,x]=ax+by,
\quad [Jx,y]=cx-ay,
\quad [Jy,x]=cx-ay,
\quad [Jy,y]=c'x+d'y.
\end{equation}
The equality of the middle two brackets follows from integrability and
$[Jx,Jy]\in\mathfrak a_J$.  The induced matrices
\[
A_x=\begin{pmatrix}a&b\\c&-a\end{pmatrix},
\qquad
A_y=\begin{pmatrix}c&-a\\c'&d'\end{pmatrix}
\]
commute.  Expanding $[A_x,A_y]=0$ gives
\begin{equation}\label{eq:abelian-relations}
 ac+bc'=0,
\qquad
 b(d'-c)=2a^2,
\qquad
 c(d'-c)=-2ac'.
\end{equation}
Unlike the non-abelian quotient case, no nilpotence conclusion for
$A_x$ is available; in particular, $a^2+bc$ need not vanish.

By \cite[Appendix B, Lemma 14]{CaoZhengFV}, the matrices $E_1,E_2$ are
\begin{equation}\label{eq:E1-ab}
E_1=\frac1{\sqrt2}
\begin{pmatrix}
b\delta+\I a&
\displaystyle\frac{-2a\delta+\I c+\I b\delta^2}{\delta'}\\[2mm]
\I b\delta'&-b\delta-\I a
\end{pmatrix},
\end{equation}
and
\begin{equation}\label{eq:E2-ab}
E_2=\frac1{\sqrt2}
\begin{pmatrix}
\displaystyle\frac{\I(c-b\delta^2)}{\delta'}&
\displaystyle\frac{\I(c'+a\delta^2)
+\delta(d'+b\delta^2)}{\delta'^2}\\[2mm]
b\delta-\I a&
\displaystyle\frac{\I(d'+b\delta^2)}{\delta'}
\end{pmatrix}.
\end{equation}
Let
$
\tau_\alpha:=\tr(Y_\alpha)$,
$\alpha=1,2$.
Unimodularity is equivalent to
\begin{equation}\label{eq:unimod-abelian}
\tr Z_1=0,
\qquad
\tr Z_2+\frac{\I(c+d')}{\sqrt2\,\delta'}=0.
\end{equation}
The first identity in \eqref{eq:unimod-abelian} gives
$0=\tau_1-\overline{\tau_1}$, and hence
\begin{equation}\label{eq:abelian-tau1-real}
\tau_1\in\R.
\end{equation}
By the definition of $Z_2$, the second identity in
\eqref{eq:unimod-abelian} becomes $\tau_2-\overline{\tau_2}
+\frac{2\I\delta}{\delta'}\overline{\tau_1}
+\frac{\I(c+d')}{\sqrt2\,\delta'}
=0.$ Using \eqref{eq:abelian-tau1-real} and
$\tau_2-\overline{\tau_2}
=2\I\operatorname{Im}\tau_2$, we obtain
\begin{equation}\label{eq:abelian-imtau2}
\operatorname{Im}\tau_2
=
-\frac{\delta}{\delta'}\tau_1
-\frac{c+d'}{2\sqrt2\,\delta'}.
\end{equation}

\subsection{The Chern case: $b=0$}

\begin{proposition}\label{prop:abelian-b0}
Let $(\mathfrak g,J,g)$ be a unimodular Hermitian Lie algebra of
complex dimension $n\geq3$ containing an abelian ideal
$\mathfrak a$ of real codimension two.  Assume that
$J\mathfrak a\neq\mathfrak a$, that
$\mathfrak g/\mathfrak a$ is abelian, and that the structural
coefficient $b$ in \eqref{eq:abelian-real-brackets} vanishes.  If
$H^c=\ka$ is a constant, then $\ka=0$ and $R^c=0$.
\end{proposition}

\begin{proof}
Assume that the Chern holomorphic sectional curvature $H^c$ is a constant $\kappa$, by Lemma~\ref{lem:curvature-support}, one has $\kappa=0$.
Hence
\[
H^c\equiv0
\qquad\text{and}\qquad
R^c_{i\bar i i\bar i}=0,\quad \forall \, 1\leq i\leq n.
\]
Moreover,
\eqref{eq:Q-zero} gives $Q_{\alpha\beta}=0,\,\alpha,\beta\in\{1,2\}.$
The second equation in \eqref{eq:abelian-relations} is
$b(d'-c)=2a^2$.  Since $b=0$, it gives $a=0$.  The third equation then becomes
$$c(d'-c)=0.$$
If $c\neq0$, then  $d'=c$. Substitute $a=b=0$ into \eqref{eq:E1-ab}, we obtain
\[
E_1=\begin{pmatrix}
0&\displaystyle\frac{\I c}{\sqrt2\delta'}\\[2mm]
0&0
\end{pmatrix}.
\]
Its first column is zero, so \eqref{eq:Q} gives
$Q_{11}=V_1V_1^*+[Y_1,Y_1^*]=0$.  Taking traces yields $V_1=0$.
Therefore $v_1^1=0$, and \eqref{eq:r1-general} gives
\[
R^c_{1\bar1 1\bar1}
=\left|\frac{\I c}{\sqrt2\delta'}\right|^2
=\frac{c^2}{2\delta'^2}>0,
\]
contradicting $H^c=0$.  Hence $c=0$.

Now $a=b=c=0$, and hence $E_1=0$.  Thus \eqref{eq:Q} and \eqref{eq:Q-zero}  give $Q_{11}=V_1V_1^*+[Y_1,Y_1^*]=0.$ Taking the trace as above yields $V_1=0.$ Substituting this back into $Q_{11}=0$, we obtain $[Y_1,Y_1^*]=0.$ By \eqref{eq:abelian-tau1-real}, one has $\tau_1\in\R.$ Moreover, since $c=0$, equation
\eqref{eq:abelian-imtau2} becomes
\begin{equation}\label{eq:b0-imtau2}
\operatorname{Im}\tau_2
=
-\frac{\delta}{\delta'}\tau_1
-\frac{d'}{2\sqrt2\,\delta'}.
\end{equation}
Applying \eqref{eq:trace-Q} with $\alpha=2$, $a=b=c=0$, \eqref{eq:E2-ab}, and \eqref{eq:abelian-tau1-real} we obtain
\[
\begin{aligned}
\|V_2\|^2
&=
2\re\left\{
\frac{\I c'+\delta d'}{\sqrt2\,(\delta')^2}
\,\overline{\tau_1}
+
\frac{\I d'}{\sqrt2\,\delta'}
\,\overline{\tau_2}
\right\}\\
&=
\frac{\sqrt2\,\delta d'}{(\delta')^2}\tau_1
+
\frac{\sqrt2\,d'}{\delta'}\operatorname{Im}\tau_2.
\end{aligned}
\]
Substituting \eqref{eq:b0-imtau2}, we find
\[
0\leq\|V_2\|^2
=
-\frac{(d')^2}{2(\delta')^2}
\leq0, \quad \delta'\in (0,1]
\]
and hence $d'=0,\,V_2=0.$ It follows that
\[
E_2=
\begin{pmatrix}
0&\displaystyle\frac{\I c'}{\sqrt2\,(\delta')^2}\\[2mm]
0&0
\end{pmatrix}.
\]
Since $H^c=0$, equation \eqref{eq:r2-general} gives
\[
0=R^c_{2\bar2 2\bar2}
=
-\frac{(c')^2}{2(\delta')^4}.
\]
Therefore $c'=0$, and thus
\[
E_1=E_2=V_1=V_2=0.
\]

Now \eqref{eq:Q-zero} gives
$[Y_\alpha,Y_\beta^*]=0$ for all $\alpha,\beta\in\{1,2\}$.  In the
curvature matrix formula \eqref{eq:curvature-matrix}, the two scalar
sums vanish because $E_1=E_2=0$, and
\[
[D_\alpha,D_\beta^*]
=\begin{pmatrix}0&0\\0&[Y_\alpha,Y_\beta^*]\end{pmatrix}=0.
\]
Thus $\mathcal R_{\alpha\bar\beta}=0$ for
$\alpha,\beta=1,2$.  Together with Lemma~\ref{lem:curvature-support},
this proves $R^c=0$.
\end{proof}

\subsection{The Chern case: $b\neq0$}

\begin{proposition}\label{prop:abelian-bnonzero}
Let $(\mathfrak g,J,g)$ be a unimodular Hermitian Lie algebra of
complex dimension $n\geq3$ containing an abelian ideal
$\mathfrak a$ of real codimension two.  Assume that
$J\mathfrak a\neq\mathfrak a$, that
$\mathfrak g/\mathfrak a$ is abelian, and that the structural
coefficient $b$ in \eqref{eq:abelian-real-brackets} is nonzero.
Then the Chern holomorphic sectional curvature of $g$ cannot be
constant.
\end{proposition}

\begin{proof}
Assume that $H^c$ is constant.  By
Lemma~\ref{lem:curvature-support}, $H^c=0$, and hence
$Q_{\alpha\beta}=0$ for $\alpha,\beta=1,2$.

Since $b\neq0$, the relations in
\eqref{eq:abelian-relations} give
$c'=-ac/b$ and $d'=c+2a^2/b$.  Put
\[
u_0:=\frac ab,
\qquad
\rho:=\frac{a^2+bc}{b^2},
\qquad
\eta:=\delta^2,
\qquad
L:=1-\eta=(\delta')^2>0.
\]
Then
\begin{equation}\label{eq:bne-parameters}
a=bu_0,
\qquad
c=b(\rho-u_0^2),
\qquad
c'=-bu_0(\rho-u_0^2),
\qquad
d'=b(\rho+u_0^2).
\end{equation}

\medskip
\noindent\emph{Step 1: the norms of $V_1$ and $V_2$.}
Taking the trace of $Q_{11}=0$ and using \eqref{eq:trace-Q},
\[
\|V_1\|^2
=2\re\{(E_1)_{11}\overline{\tau_1}
+(E_1)_{21}\overline{\tau_2}\}.
\]
By \eqref{eq:E1-ab}, \eqref{eq:abelian-tau1-real}, and \eqref{eq:abelian-imtau2}, we have
\[
\|V_1\|^2=-\frac{b(c+d')}{2}.
\]
By \eqref{eq:bne-parameters}, we have
\begin{equation}\label{eq:abelian-V1norm-detailed}
\frac{\|V_1\|^2}{b^2}=-\rho.
\end{equation}
In particular,
\begin{equation}\label{eq:rho-nonpositive}
\rho\leq0.
\end{equation}

Similarly, taking the trace of $Q_{22}=0$ gives
\[
\|V_2\|^2
=2\re\{(E_2)_{12}\overline{\tau_1}
+(E_2)_{22}\overline{\tau_2}\}.
\]
Using \eqref{eq:E2-ab}, \eqref{eq:abelian-tau1-real}, and \eqref{eq:abelian-imtau2}, we obtain
\[
\|V_2\|^2
=-\frac{(d'+b\eta)(c+d')}{2(1-\eta)}.
\]
Using \eqref{eq:bne-parameters} and $c+d'=2b\rho$, we obtain
\begin{equation}\label{eq:abelian-V2norm-detailed}
\frac{\|V_2\|^2}{b^2}
=-\frac{\rho(\rho+u^2+\eta)}{1-\eta}.
\end{equation}
Equations \eqref{eq:abelian-V1norm-detailed} and
\eqref{eq:abelian-V2norm-detailed} are the required norm identities.

\medskip
\noindent\emph{Step 2: the first diagonal curvature inequality.}
Define
\begin{equation}\label{eq:A-def-detailed}
A=\rho^2+2\eta\rho-2\rho u_0^2
+2\eta u_0^2+2\eta+u_0^4-1.
\end{equation}
Substituting \eqref{eq:bne-parameters} into the \eqref{eq:E1-ab} and \eqref{eq:E2-ab}, we get
\begin{align*}
&\frac{|(E_1)_{12}|^2}{b^2}
=\frac{4u_0^2\eta+(\rho-u_0^2+\eta)^2}{2(1-\eta)}, \qquad \quad
\frac{|(E_1)_{21}|^2}{b^2}=\frac{1-\eta}{2},\\
&-\frac{2}{b^2}\re\{(E_1)_{21}\overline{(E_2)_{11}}\}
=-(\rho-u_0^2-\eta),\quad
\frac{|(E_1)_{11}|^2}{b^2}=\frac{u_0^2+\eta}{2}.
\end{align*}
Substituting this into \eqref{eq:r1-general}, we obtain
\begin{equation}\label{eq:r1-A-detailed}
\frac{r_1}{b^2}+\rho=\frac{A}{2(1-\eta)}.
\end{equation}
Since $H^c=0$, equations \eqref{eq:r1-general} and
\eqref{eq:abelian-V1norm-detailed} give
\[
\frac{r_1}{b^2}=\frac{\|v_1^1\|^2}{b^2}
\leq\frac{\|V_1\|^2}{b^2}=-\rho.
\]
Together with \eqref{eq:r1-A-detailed} and $1-\eta>0$, this implies
\begin{equation}\label{eq:A-nonpositive-detailed}
A\leq0.
\end{equation}

\medskip
\noindent\emph{Step 3: the second diagonal curvature inequality.}
Similarly, substituting \eqref{eq:bne-parameters} into
\eqref{eq:E1-ab} and \eqref{eq:E2-ab}, and then using
\eqref{eq:r2-general}, we obtain the purely algebraic identity
\begin{equation}\label{eq:r2-A-detailed}
\frac{r_2}{b^2}
=
\frac{-(\eta+u_0^2)A
-4\rho u_0^2(1-\eta)}
{2(1-\eta)^2}
-
\frac{\rho(\rho+u_0^2+\eta)}{1-\eta}.
\end{equation}
Combining \eqref{eq:r2-A-detailed} with
\eqref{eq:abelian-V2norm-detailed}, we obtain
\begin{equation}\label{eq:abelian-r2-difference-detailed}
\frac{r_2}{b^2}
-\frac{\|V_2\|^2}{b^2}
=
\frac{-(\eta+u_0^2)A
-4\rho u_0^2(1-\eta)}
{2(1-\eta)^2}.
\end{equation}
Since $H^c=0$, equations \eqref{eq:r2-general} and \eqref{eq:abelian-V2norm-detailed} give
\[
\frac{\|v_2^2\|^2-\|V_2\|^2}{b^2}\leq0.
\]
The right hand side is non-negative because
$\eta+u_0^2\geq0$, $A\leq0$, $\rho\leq0$, and $1-\eta>0$.
Consequently both sides vanish.  Since the two summands in the
numerator on the right are separately non-negative, each is zero:
\begin{equation}\label{eq:equality-conditions-detailed}
(\eta+u_0^2)A=0,
\qquad
\rho u_0^2=0.
\end{equation}
Moreover,
$\|v_2^2\|^2=\|V_2\|^2
=\|v_2^1\|^2+\|v_2^2\|^2$, so
\begin{equation}\label{eq:v21-zero}
v_2^1=0.
\end{equation}

If $\eta+u_0^2=0$, then $\eta=u_0=0$.  Equation
\eqref{eq:abelian-V2norm-detailed} becomes
$\|V_2\|^2/b^2=-\rho^2$, hence $\rho=0$.  But then
\eqref{eq:A-def-detailed} gives $A=-1$, and
\eqref{eq:r1-A-detailed} yields $\frac{r_1}{b^2}=-1/2$, contradicting
$\frac{r_1}{b^2}=\|v_1^1\|^2/b^2\geq0$.  Therefore
$\eta+u_0^2>0$, and \eqref{eq:equality-conditions-detailed} gives
\begin{equation}\label{eq:A-zero-detailed}
A=0.
\end{equation}
Equation \eqref{eq:abelian-V1norm-detailed} and \eqref{eq:r1-A-detailed} now imply
$\frac{r_1}{b^2}=-\rho=\|V_1\|^2/b^2$.  Hence
$\|v_1^1\|^2=\|V_1\|^2
=\|v_1^1\|^2+\|v_1^2\|^2$, and therefore
\begin{equation}\label{eq:v12-zero}
v_1^2=0.
\end{equation}
Combining \eqref{eq:v21-zero} and \eqref{eq:v12-zero}, we have
\begin{equation}\label{eq:sparse-V-detailed}
v_1^2=0,
\qquad
v_2^1=0.
\end{equation}

\medskip
\noindent\emph{Step 4: the subcase $u_0\neq0$.}
If $u_0\neq0$, then $\rho=0$ by \eqref{eq:equality-conditions-detailed}.  The equality
$A=0$ factors as
\[
(1+u_0^2)(2\eta+u_0^2-1)=0,
\]
so
\begin{equation}\label{eq:u-nonzero-eta}
\eta=\frac{1-u_0^2}{2}.
\end{equation}
Since $\eta=\delta^2$,
\begin{equation}\label{eq:u-nonzero-delta-relations}
(\delta')^2=1-\delta^2=\frac{1+u_0^2}{2}
=\delta^2+u_0^2.
\end{equation}
Substituting $\rho=0$ into \eqref{eq:bne-parameters}, we get
\begin{equation}\label{eq:u-nonzero-struct-parameters}
a=bu_0,
\qquad c=-bu_0^2,
\qquad c'=bu_0^3,
\qquad d'=bu_0^2.
\end{equation}
Set
\begin{equation}\label{eq:xyz}
x=\frac b{\sqrt2}(\delta+\I u_0),
\quad
y=\frac b{\sqrt2\delta'}
\{-2u_0\delta+\I(\delta^2-u_0^2)\},
\quad
z=\frac{\I b\delta'}{\sqrt2}.
\end{equation}
Substitution in $E_1,E_2$ and use of
\eqref{eq:u-nonzero-delta-relations} give
\begin{equation}\label{eq:u-nonzero-E-matrices}
E_1=\begin{pmatrix}x&y\\z&-x\end{pmatrix},
\qquad
E_2=\begin{pmatrix}-z&x\\\overline x&z\end{pmatrix}.
\end{equation}
Because $z$ is purely imaginary,
\[
E_1^*=\begin{pmatrix}\overline x&-z\\\overline y&-\overline x\end{pmatrix},
\qquad
E_2^*=\begin{pmatrix}z&x\\\overline x&-z\end{pmatrix}.
\]

The upper-left curvature block is given by \eqref{eq:P-block}.  By definition,
\[
R^c_{2\bar2 1\bar2}=(P_{22})_{12},
\qquad
R^c_{1\bar2 2\bar2}=(P_{12})_{22}.
\]
The $V$-contributions vanish because
\[
(V_2^*V_2)_{12}=(v_2^1)^*v_2^2=0,
\qquad
(V_2^*V_1)_{22}=(v_2^2)^*v_1^2=0
\]
by \eqref{eq:sparse-V-detailed}.
For $P_{22}$, formula \eqref{eq:P-block} and \eqref{eq:u-nonzero-E-matrices} give
\[
P_{22}=[E_2,E_2^*]-xE_1^*-zE_2^*
-\overline xE_1+zE_2-V_2^*V_2.
\]
Taking the $(1,2)$-entry and using
\eqref{eq:u-nonzero-E-matrices},
\[
(P_{22})_{12}=-4zx+xz-zx-\overline xy+zx
=-3zx-\overline xy.
\]
Thus
\begin{equation}\label{eq:u-nonzero-R2212}
R^c_{2\bar2 1\bar2}=2b^2\delta'(u_0-\I\delta)
\end{equation}
by \eqref{eq:xyz}. Similarly, formula \eqref{eq:P-block} and \eqref{eq:u-nonzero-E-matrices} give
\[
P_{12}=[E_1,E_2^*]-yE_1^*+xE_2^*
-zE_1-xE_2-V_2^*V_1.
\]
Taking the $(2,2)$-entries and using
\eqref{eq:u-nonzero-E-matrices}, we get
\begin{equation}\label{eq:u-nonzero-R1222}
R^c_{1\bar2 2\bar2}=(P_{12})_{22}=0.
\end{equation}
By \eqref{eq:symmetrization}, we have
$$\widehat R^c_{2\bar2 1\bar2}=\frac12\{R^c_{2\bar2 1\bar2}+R^c_{1\bar2 2\bar2}\}.$$
Using \eqref{eq:u-nonzero-R2212} and
\eqref{eq:u-nonzero-R1222}, we obtain
\begin{equation}\label{eq:cubic-u-nonzero}
\frac{\widehat R^c_{2\bar2 1\bar2}}{b^2}
=\delta'(u_0-\I\delta).
\end{equation}
Its real part is $\delta'u_0\neq0$.  This contradicts $H^c=0$.

\medskip
\noindent\emph{Step 5: the subcase $u_0=0$.}
By \eqref{eq:bne-parameters} and $u_0=0$, we obtain $a=0$, $c=b\rho$, $c'=0$, and $d'=b\rho=c$.  The equality $A=0$
becomes
\[
(\rho+1)(\rho-1+2\eta)=0.
\]
The possibility $\rho=-1$ is excluded by
\eqref{eq:abelian-V2norm-detailed}, which would give
$\|V_2\|^2/b^2=-1$.  Hence
\begin{equation}\label{eq:u-zero-relations}
\delta^2=\eta=\frac{1-\rho}{2},
\qquad
(\delta')^2=\frac{1+\rho}{2},
\qquad
-1<\rho\leq0.
\end{equation}
In particular, $\delta\neq0$.
Set
\[
x=\frac{b\delta}{\sqrt2},
\qquad
z=\frac{\I b\delta'}{\sqrt2},
\qquad
h=\frac{\I b(\rho-\delta^2)}{\sqrt2\delta'}.
\]
Using \eqref{eq:E1-ab}, \eqref{eq:E2-ab} and \eqref{eq:u-zero-relations}, the matrices become
\begin{equation}\label{eq:u-zero-E-matrices}
E_1=\begin{pmatrix}x&z\\z&-x\end{pmatrix},
\qquad
E_2=\begin{pmatrix}h&x\\x&z\end{pmatrix}.
\end{equation}
Here $x$ is real and $h,z$ are purely imaginary.  Again the relevant
$V$-entries vanish by \eqref{eq:sparse-V-detailed}.

For $P_{22}$, taking the $(1,2)$-entry in \eqref{eq:P-block}, and using \eqref{eq:u-zero-E-matrices}, we get
\[
(P_{22})_{12}=2x(h-z).
\]
Since
\[
h-z=\frac{\I b}{\sqrt2\delta'}(\rho-1)
=-\frac{\sqrt2\I b\delta^2}{\delta'},
\]
we obtain
\begin{equation}\label{eq:u-zero-R2212}
R^c_{2\bar2 1\bar2}=(P_{22})_{12}
=-\frac{2\I b^2\delta^3}{\delta'}.
\end{equation}
For $P_{12}$, taking the $(2,2)$-entry of \eqref{eq:P-block} and using \eqref{eq:u-zero-E-matrices}, we get
\[
(P_{12})_{22}=-x(h+z).
\]
Since $h+z=\frac{\sqrt2\I b\rho}{\delta'},$ we find
\begin{equation}\label{eq:u-zero-R1222}
R^c_{1\bar2 2\bar2}=(P_{12})_{22}
=-\frac{\I b^2\delta\rho}{\delta'}.
\end{equation}
Consequently,
\begin{align*}
\frac{\widehat R^c_{2\bar2 1\bar2}}{b^2}=\frac1{2b^2}\{R^c_{2\bar2 1\bar2}+R^c_{1\bar2 2\bar2}\}=-\frac{\I\delta}{2\delta'}(2\delta^2+\rho)=-\frac{\I\delta}{2\delta'}=-\frac{\I\delta(\rho+1)}{4\delta'^3},
\end{align*}
where we used $2\delta^2+\rho=1$ and
$\rho+1=2\delta'^2$.  Since $\delta\neq0$ by \eqref{eq:u-zero-relations} and $\delta'>0$, this
component is non-zero, again contradicting $H^c=0$.

Both subcases are impossible.  Therefore the abelian quotient case
with $b\neq0$ cannot admit constant Chern holomorphic sectional
curvature.
\end{proof}

\subsection{Canonical connections with $\chi>0$}

\begin{proposition}\label{prop:abelian-sign-reduction}
Let $(\mathfrak g,J,g)$ be a unimodular Hermitian Lie algebra of
complex dimension $n\geq3$ containing an abelian ideal
$\mathfrak a$ of real codimension two.  Assume that
$J\mathfrak a\neq\mathfrak a$,
$\mathfrak g/\mathfrak a$ is abelian, and $\chi>0$. If
$H^{D_s^r}=\ka$ is a constant, then
\begin{equation}\label{eq:abelian-W-reduction}
\ka=-\chi\left(
|\langle Z_1w,\overline w\rangle|^2
+|\langle Z_2w,\overline w\rangle|^2\right)
\end{equation}
for every unit vector $w\in\mathcal W=(\mathfrak a_J)^{1,0}$. Consequently, $\ka\leq0$.  If $\ka=0$, then $Z_1=Z_2=0$.
\end{proposition}

\begin{proof}
Fix a unit vector $w\in\mathcal W$.  Using the admissible frame of
Cao--Zheng \cite[Appendix~B]{CaoZhengFV}, make a unitary change of
basis in $\mathcal W$ so that $e_3=w$, while keeping $e_1$ and
$e_2$ fixed.  Since this change of basis acts only on
$\mathcal W$, the block form \eqref{eq:abelian-block-form}, the
relations \eqref{eq:abelian-C-D-relations}, and the torsion
identities
\eqref{eq:abelian-torsion-1}--\eqref{eq:abelian-torsion-3}
remain valid.

As in \eqref{eq:nonab-vhat-diagonal}, the formula \eqref{eq:vhat-definition} gives
\begin{equation*}
\widehat v_{w\bar w w\bar w}
=
\sum_{a=1}^n
\left|
\left\langle T(w,e_a),\overline w\right\rangle
\right|^2.
\end{equation*}
For $a\geq3$, one has $T(w,e_a)=0$, since
$T(\mathcal W,\mathcal W)=0$ by
\eqref{eq:abelian-C-D-relations}.  Moreover, the skew-symmetry of
the torsion and
\eqref{eq:abelian-torsion-1}--\eqref{eq:abelian-torsion-2} give
\[
\left\langle T(w,e_1),\overline w\right\rangle
=
-\left\langle Z_1w,\overline w\right\rangle,
\qquad
\left\langle T(w,e_2),\overline w\right\rangle
=
-\left\langle Z_2w,\overline w\right\rangle.
\]
Consequently,
\begin{equation}\label{eq:abelian-vhat-Z}
\widehat v_{w\bar w w\bar w}
=
\left|\left\langle Z_1w,\overline w\right\rangle\right|^2
+
\left|\left\langle Z_2w,\overline w\right\rangle\right|^2.
\end{equation}

Since $w=e_3$ in the chosen unitary basis,
Lemma~\ref{lem:curvature-support} gives
\[
\widehat R^c_{w\bar w w\bar w}
=
R^c_{w\bar w w\bar w}
=
0.
\]
For a unit vector $w$, the $(w,\bar w,w,\bar w)$ component of
\eqref{eq:master-equation} is
$\widehat R^c_{w\bar w w\bar w}
-\chi\widehat v_{w\bar w w\bar w}=\ka$.  Combining this with
\eqref{eq:abelian-vhat-Z}, we obtain
\[
\ka
=
-\chi\left(
\left|\left\langle Z_1w,\overline w\right\rangle\right|^2
+
\left|\left\langle Z_2w,\overline w\right\rangle\right|^2
\right),
\]
which proves \eqref{eq:abelian-W-reduction}.  Since $\chi>0$, it
follows immediately that $\ka\leq0$.

Suppose now that $\ka=0$.  Then
$\langle Z_\alpha w,\overline w\rangle=0$ for every unit
$w\in\mathcal W$ and $\alpha=1,2$.  By homogeneity, the same
identity holds for every $w\in\mathcal W$.  Define the sesquilinear
forms
\[
B_\alpha(u,v)
:=
\left\langle Z_\alpha u,\overline v\right\rangle,
\qquad
\alpha=1,2.
\]
Since $B_\alpha(w,w)=0$ for every $w$, the complex polarization
identity gives $B_\alpha(u,v)=0$ for all $u,v\in\mathcal W$.
Therefore $Z_1=Z_2=0$.
\end{proof}
\subsection{The $H^{D_s^r}=0$ case}

\begin{proposition}\label{prop:abelian-zero-canonical}
Let $(\mathfrak g,J,g)$ be a unimodular Hermitian Lie algebra of complex dimension $n\geq3$ containing an abelian ideal
$\mathfrak a$ of real codimension two.  Assume that
$J\mathfrak a\neq\mathfrak a$,
$\mathfrak g/\mathfrak a$ is abelian, and $\chi>0$. If $H^{D_s^r}=0$, then $b=0$ and
$g$ is K\"ahler flat.
\end{proposition}

\begin{proof}
Let $\mu_\alpha:=T^\alpha_{12}$.  By Proposition~\ref{prop:abelian-sign-reduction}, one has
$Z_1=Z_2=0$.  In particular, $\tr Z_2=0$.  The second identity in
\eqref{eq:unimod-abelian} therefore gives
$\I(c+d')/(\sqrt2\,\delta')=0$.  Since $\delta'>0$, we obtain
\begin{equation}\label{eq:abelian-zero-unimod}
d'=-c.
\end{equation}
Since $Z_1=Z_2=0$,
\eqref{eq:vhat-definition} and
\eqref{eq:abelian-torsion-3} give
$\widehat v_{\alpha\bar\alpha\alpha\bar\alpha}
=\|v_\alpha^\alpha\|^2+|\mu_\alpha|^2$.
Moreover, \eqref{eq:r1-general} and \eqref{eq:r2-general} give
$R^c_{\alpha\bar\alpha\alpha\bar\alpha}
=r_\alpha-\|v_\alpha^\alpha\|^2$.  Therefore the diagonal
components of \eqref{eq:master-equation}, with $\ka=0$, yield
\begin{equation}\label{eq:abelian-zero-diagonal}
r_\alpha
=
(1+\chi)\|v_\alpha^\alpha\|^2
+\chi|\mu_\alpha|^2
\geq0,
\qquad
\alpha=1,2.
\end{equation}

Suppose that $b\neq0$.  The second relation in
\eqref{eq:abelian-relations}, together with
\eqref{eq:abelian-zero-unimod}, gives
$b(-2c)=2a^2$, and hence
\begin{equation}\label{eq:abelian-zero-rho}
a^2+bc=0.
\end{equation}
Put
\[
u_0:=\frac{a}{b},
\qquad
\rho:=\frac{a^2+bc}{b^2},
\qquad
\eta:=\delta^2,
\qquad
L:=1-\eta=(\delta')^2>0.
\]
Thus $\rho=0$.

Specializing the purely algebraic identities
\eqref{eq:r1-A-detailed} and \eqref{eq:r2-A-detailed} to
$\rho=0$, we obtain
\begin{equation}\label{eq:abelian-zero-r12}
\frac{r_1}{b^2}
=
\frac{A_0}{2L},
\qquad
\frac{r_2}{b^2}
=
-\frac{(\eta+u_0^2)A_0}{2L^2},
\qquad A_0:=2\eta u_0^2+2\eta+u_0^4-1.
\end{equation}

If $\eta+u_0^2=0$, then $\eta=u_0=0$, and hence $A_0=-1$ and
$L=1$.  The first identity in \eqref{eq:abelian-zero-r12} would
then give $r_1=-b^2/2<0$, contradicting
\eqref{eq:abelian-zero-diagonal}.  Thus $\eta+u_0^2>0$.

Since $L>0$, while \eqref{eq:abelian-zero-diagonal} gives $r_1\geq0$, the first identity in
\eqref{eq:abelian-zero-r12} gives $A_0\geq0$.  Since
$\eta+u_0^2>0$, $L>0$, and $r_2\geq0$, the second identity gives $A_0\leq0$.  Therefore $A_0=0$, and hence $r_1=r_2=0$.

We next compute $\mu_1$.  By \eqref{eq:abelian-C12},
$C^1_{12}=0$.  By the definition of $D_\alpha=(D^j_{i\alpha})$, equation
\eqref{eq:abelian-block-form} gives
$D^1_{12}=(E_2)_{11}$ and
$D^1_{21}=(E_1)_{21}$.  Hence
\eqref{torsion}, together with
\eqref{eq:E1-ab} and \eqref{eq:E2-ab}, yields
\[
\mu_1
=
-C^1_{12}-D^1_{12}+D^1_{21}
=
-(E_2)_{11}+(E_1)_{21}
=
\frac{\I(b-c)}{\sqrt2\,\delta'}.
\]
By \eqref{eq:abelian-zero-rho}, $c=-a^2/b=-bu_0^2$, and therefore
\[ \mu_1 = \frac{\I b(1+u_0^2)}{\sqrt2\,\delta'} \neq0, \]
since $b\neq0$ and $1+u_0^2>0$. On the other hand, \eqref{eq:abelian-zero-diagonal} and $r_1=0$ give \[ \chi|\mu_1|^2=0. \]
Since $\chi>0$, we obtain $\mu_1=0$, which contradicts $\mu_1\neq0$. Consequently, the case $b\neq0$ is impossible, and hence \begin{equation}\label{eq} b=0. \end{equation}

We now compute the remaining structure constants.  The second
relation in \eqref{eq:abelian-relations} gives $a=0$.  Using
$d'=-c$ from \eqref{eq:abelian-zero-unimod}, the third relation
in \eqref{eq:abelian-relations} becomes
$c(-2c)=0$.  Hence
\[
a=b=c=d'=0.
\]
Substituting these identities into \eqref{eq:E1-ab} and
\eqref{eq:E2-ab}, we obtain
\[
E_1=0,
\qquad
E_2=
\begin{pmatrix}
0&\dfrac{\I c'}{\sqrt2\,(\delta')^2}\\[2mm]
0&0
\end{pmatrix}.
\]
Since $E_1=0$ and $(E_2)_{11}=0$, the formula
\eqref{eq:r1-general} gives $r_1=0$.  Moreover, by
\eqref{eq:abelian-C12}, one has $C^1_{12}=0$, while
\eqref{eq:abelian-block-form} gives
$D^1_{12}=(E_2)_{11}$ and $D^1_{21}=(E_1)_{21}$.  Hence the
Chern torsion formula \eqref{torsion} yields
\[
\mu_1
=
T^1_{12}
=
-C^1_{12}-D^1_{12}+D^1_{21}
=
-C^1_{12}-(E_2)_{11}+(E_1)_{21}
=
0.
\]
Consequently, the first identity in
\eqref{eq:abelian-zero-diagonal} gives
$(1+\chi)\|v_1^1\|^2=0$.  Since $\chi>0$, we obtain $$v_1^1=0.$$
Similarly, \eqref{eq:r2-general} gives
\[
\begin{aligned}
r_2
&=
|(E_2)_{21}|^2-|(E_2)_{12}|^2
-2|(E_2)_{22}|^2
-2\operatorname{Re}
 \bigl((E_2)_{12}\overline{(E_1)_{22}}\bigr)=
 -\frac{(c')^2}{2(\delta')^4}.
\end{aligned}
\]
By \eqref{eq:abelian-C12} and \eqref{torsion},
\[
\mu_2
=
-C^2_{12}-D^2_{12}+D^2_{21}
=
-(E_2)_{12}+(E_1)_{22}
=
-\frac{\I c'}{\sqrt2\,(\delta')^2}.
\]
Thus the second equation in \eqref{eq:abelian-zero-diagonal}
becomes
\[
-\frac{(c')^2}{2(\delta')^4}
=
(1+\chi)\|v_2^2\|^2
+\chi\frac{(c')^2}{2(\delta')^4},
\]
or equivalently,
\[
0
=
(1+\chi)
\left(
\|v_2^2\|^2
+\frac{(c')^2}{2(\delta')^4}
\right).
\]
Since $1+\chi>0$, we obtain $c'=0$ and $v_2^2=0$.  Consequently,
$$E_1=E_2=\mu_2=0.$$

Since $v_1^1=v_2^2=0$, write
$V_1=(0,v_1^2)$ and $V_2=(v_2^1,0)$.  As
$E_1=E_2=0$, formula \eqref{eq:P-block} gives
$P_{\alpha\beta}=-V_\beta^*V_\alpha$.  Since
$(P_{\alpha\beta})_{\gamma\delta}
=R^c_{\alpha\bar\beta\gamma\bar\delta}$, we obtain
\[
R^c_{1\bar1 2\bar2}=-\|v_1^2\|^2,
\qquad
R^c_{2\bar2 1\bar1}=-\|v_2^1\|^2,
\]
and
$R^c_{2\bar1 1\bar2}
=R^c_{1\bar2 2\bar1}=0$.
Consequently, \eqref{eq:symmetrization} yields
\[
4\widehat R^c_{1\bar1 2\bar2}
=
-\|v_1^2\|^2-\|v_2^1\|^2.
\]
On the other hand, taking $(i,j,k,\ell)=(1,1,2,2)$ in
\eqref{eq:vhat-definition} gives
\[
4\widehat v_{1\bar1 2\bar2}
=
\sum_a\left(
T^1_{1a}\overline{T^2_{2a}}
+|T^1_{2a}|^2
+|T^2_{1a}|^2
+T^2_{2a}\overline{T^1_{1a}}
\right).
\]
Since $v_1^1=v_2^2=\mu_1=\mu_2=0$, the first and fourth sums
vanish.  By \eqref{eq:abelian-torsion-3}, the remaining two sums
are $\|v_2^1\|^2$ and $\|v_1^2\|^2$, respectively.  Hence
\[
4\widehat v_{1\bar1 2\bar2}
=
\|v_1^2\|^2+\|v_2^1\|^2.
\]
The $(1,\bar1,2,\bar2)$ component of
\eqref{eq:master-equation}, with $\ka=0$, therefore yields
\[
0
=
-(1+\chi)
\bigl(\|v_1^2\|^2+\|v_2^1\|^2\bigr).
\]
Since $\chi>0$, we conclude that $v_1^2=v_2^1=0$.  Thus
\begin{equation}\label{eq:abelian-zero-EV}
E_1=E_2=V_1=V_2=0.
\end{equation}

Equations
\eqref{eq:abelian-torsion-1}--\eqref{eq:abelian-torsion-3},
together with $Z_1=Z_2=0$ and $V_1=V_2=0$, give
$T^q_{\alpha p}=T^\beta_{\alpha p}=0$.  Moreover, the first line
of \eqref{eq:abelian-C-D-relations} gives
$C^j_{pq}=D^j_{pq}=D^j_{qp}=0$ for every $j$ and $p,q\geq3$;
hence \eqref{torsion} yields $T^j_{pq}=0$.  Finally,
\eqref{eq:abelian-u} gives $T^p_{12}=0$, while
$\mu_\alpha=T^\alpha_{12}=0$ for $\alpha=1,2$.  By skew-symmetry, all components of the Chern torsion vanish. Thus $T=0$, and $g$ is K\"ahler.

Finally, we prove flatness.  By \eqref{eq:abelian-Y-commute} we have
\begin{equation*}
[Y_1,Y_2]=0.
\end{equation*}
By \eqref{eq:abelian-Z-definitions}, $Z_1=0$ gives
$Y_1^*=Y_1$, while $Z_2=0$ gives
$Y_2^*=Y_2+2\zeta Y_1$.  Combining these identities with
\eqref{eq:abelian-Y-commute}, we obtain
\[
[Y_1,Y_1^*]
=
[Y_1,Y_2^*]
=
[Y_2,Y_1^*]
=
[Y_2,Y_2^*]
=
0.
\]
Hence $[Y_\alpha,Y_\beta^*]=0$. By \eqref{eq:abelian-zero-EV}, one has
\[
D_\alpha
=
\begin{pmatrix}
0&0\\
0&Y_\alpha
\end{pmatrix},
\qquad \alpha=1,2.
\]
Since $E_1=E_2=0$, the formula
\eqref{eq:curvature-matrix} becomes
\[
\mathcal R_{\alpha\bar\beta}
=
[D_\alpha,D_\beta^*]
=
\begin{pmatrix}
0&0\\
0&[Y_\alpha,Y_\beta^*]
\end{pmatrix}=0.
\]
Hence $\mathcal R_{\alpha\bar\beta}=0$ for
$\alpha,\beta\in\{1,2\}$.  Equivalently,
$R^c_{\alpha\bar\beta k\bar\ell}=0$ for all
$1\leq k,\ell\leq n$. All remaining Chern curvature
components vanish by Lemma~\ref{lem:curvature-support}.  Hence
$R^c=0$.

Since $g$ is K\"ahler, its Chern, Bismut, and Levi--Civita
connections coincide.  Therefore every canonical metric
connection coincides with the flat Chern connection, so
$R^{D_s^r}=0$ and $g$ is K\"ahler flat, completing the proof.
\end{proof}

\subsection{The negative constant case}

\begin{lemma}\label{lem:abelian-negative-normal-form}
Let $(\mathfrak g,J,g)$ be a unimodular Hermitian Lie algebra of
complex dimension $n\geq3$ containing an abelian ideal
$\mathfrak a$ of real codimension two.  Assume that
$J\mathfrak a\neq\mathfrak a$ and that
$\mathfrak g/\mathfrak a$ is abelian.  If $\chi>0$ and
$H^{D_s^r}=\ka<0$, put $m=n-2$ and
$\mathcal W:=(\mathfrak a_J)^{1,0}$.  Then there exists
$\beta\in\R\setminus\{0\}$ such that
\begin{equation}\label{eq:abelian-negative-normal-form}
Z_1=0,
\qquad Z_2=2\I\beta I_m,
\qquad \ka=-4\chi\beta^2,
\qquad
\beta=-\frac{c+d'}{2\sqrt2\,m\delta'},
\qquad \zeta=\I\delta/\delta'.
\end{equation}
\end{lemma}

\begin{proof}
Write
\[
Y_\alpha=H_\alpha+\I S_\alpha,
\qquad
H_\alpha:=\frac12(Y_\alpha+Y_\alpha^*),
\qquad
S_\alpha:=\frac1{2\I}(Y_\alpha-Y_\alpha^*),
\]
for $\alpha=1,2$.  Thus $H_\alpha$ and $S_\alpha$ are Hermitian.
Since $\zeta=\I\delta/\delta'$, the definitions in
\eqref{eq:abelian-Z-definitions} give
\[
Z_1=2\I S_1,
\qquad
Z_2=
2\frac{\delta}{\delta'}S_1
+
2\I\left(S_2+\frac{\delta}{\delta'}H_1\right).
\]

For every unit vector
$w\in\mathcal W=(\mathfrak a_J)^{1,0}$, the Hermitian property of
$S_1$ and
$S_2+\frac{\delta}{\delta'}H_1$ implies that
$\langle S_1w,\overline w\rangle$ and
$\langle(S_2+\frac{\delta}{\delta'}H_1)w,\overline w\rangle$
are real. Hence
\[
\left|\langle Z_1w,\overline w\rangle\right|^2
=
4\langle S_1w,\overline w\rangle^2
\]
and
\[
\left|\langle Z_2w,\overline w\rangle\right|^2
=
4\frac{\delta^2}{(\delta')^2}
\langle S_1w,\overline w\rangle^2
+
4\left\langle
\left(S_2+\frac{\delta}{\delta'}H_1\right)w,\overline w
\right\rangle^2.
\]
Since
$1+\delta^2/(\delta')^2=1/(\delta')^2$,
equation \eqref{eq:abelian-W-reduction} becomes
\begin{equation}\label{eq:abelian-circle}
\left\langle\frac{S_1}{\delta'}w,\overline w\right\rangle^2
+
\left\langle
\left(S_2+\frac{\delta}{\delta'}H_1\right)w,\overline w
\right\rangle^2
=
-\frac{\ka}{4\chi}
\end{equation}
for every unit vector $w\in\mathcal W$.

Applying Lemma~\ref{lem:circle-rigidity} to the Hermitian matrices
$S_1/\delta'$ and
$S_2+\frac{\delta}{\delta'}H_1$, we obtain
$\alpha_0,\beta\in\R$ such that
\begin{equation}\label{eq:abelian-S-scalar}
\frac{S_1}{\delta'}=\alpha_0I_m,
\qquad
S_2+\frac{\delta}{\delta'}H_1=\beta I_m.
\end{equation}
By the first identity in \eqref{eq:unimod-abelian},
$\tr Z_1=0$.  On the other hand,
\eqref{eq:abelian-S-scalar} gives
$Z_1=2\I\delta'\alpha_0I_m$, and hence
$0=\tr Z_1=2\I m\delta'\alpha_0$.  Since $m\geq1$ and
$\delta'>0$, it follows that $\alpha_0=0$.  Thus
$S_1=0$ and $Z_1=0$.  The second identity in
\eqref{eq:abelian-S-scalar} then gives
\[
Z_2
=
2\frac{\delta}{\delta'}S_1
+
2\I\left(S_2+\frac{\delta}{\delta'}H_1\right)
=
2\I\beta I_m.
\]
Substituting $S_1=0$ and
$S_2+\frac{\delta}{\delta'}H_1=\beta I_m$ into
\eqref{eq:abelian-circle}, and using $|w|=1$, yields
$\beta^2=-\ka/(4\chi)$.  Therefore
$\ka=-4\chi\beta^2$.  Since $\ka<0$ and $\chi>0$, one has
$\beta\neq0$.

Finally, the second identity in \eqref{eq:unimod-abelian} gives
\[
0
=
\tr Z_2+\frac{\I(c+d')}{\sqrt2\,\delta'}
=
2\I m\beta+\frac{\I(c+d')}{\sqrt2\,\delta'}.
\]
Consequently,
\[
\beta
=
-\frac{c+d'}{2\sqrt2\,m\delta'}.
\]
This proves \eqref{eq:abelian-negative-normal-form}.
\end{proof}
\begin{lemma}\label{lem:abelian-negative-reduction}
Let $(\mathfrak g,J,g)$ be a unimodular Hermitian Lie algebra of
complex dimension $n\geq3$ containing an abelian ideal
$\mathfrak a$ of real codimension two.  Assume that
$J\mathfrak a\neq\mathfrak a$, that
$\mathfrak g/\mathfrak a$ is abelian, and that a canonical metric
connection $D_s^r$ with $\chi>0$ has constant holomorphic
sectional curvature $H^{D_s^r}=\ka<0$.  Put
$\mathcal W:=(\mathfrak a_J)^{1,0}$ and
$m:=\dim_{\C}\mathcal W=n-2$.  Use the normal form
$Z_1=0$ and $Z_2=2\I\beta I_m$ from
Lemma~\ref{lem:abelian-negative-normal-form}, and let $u$ be the
vector defined in \eqref{eq:abelian-u}.  Then
\[
u=v_2^1=v_2^2=0,
\qquad
v_1^2=2\zeta v_1^1,
\qquad \zeta=\I\delta/\delta'.
\]
In particular,
\begin{equation}\label{eq:abelian-V-negative}
V_1=(v_1^1,v_1^2),
\qquad
v_1^2=2\zeta v_1^1,
\qquad
V_2=0.
\end{equation}
\end{lemma}

\begin{proof}
By Lemma~\ref{lem:abelian-negative-normal-form},
$\beta\neq0$, $Z_1=0$, and $Z_2=2\I\beta I_m$.

Fix $\alpha\in\{1,2\}$ and $p,q,r\geq3$.  By
Lemma~\ref{lem:curvature-support} and
\eqref{eq:symmetrization}, one has
$\widehat R^c_{\alpha\bar p q\bar r}=0$.  Together with $\delta_{\alpha p}\delta_{qr}
+\delta_{\alpha r}\delta_{qp}=0$, the formula
\eqref{eq:master-equation} becomes
$-\chi\widehat v_{\alpha\bar p q\bar r}=0$.  Since $\chi>0$,
\[
\widehat v_{\alpha\bar p q\bar r}=0.
\]

Taking $(i,j,k,\ell)=(\alpha,p,q,r)$ in
\eqref{eq:vhat-definition}, we obtain
\[
\begin{aligned}
0=4\widehat v_{\alpha\bar p q\bar r}
=\sum_a\bigl(
&T^p_{\alpha a}\overline{T^q_{ra}}
+T^p_{qa}\overline{T^\alpha_{ra}}+T^r_{\alpha a}\overline{T^q_{pa}}
+T^r_{qa}\overline{T^\alpha_{pa}}
\bigr).
\end{aligned}
\]
Equations
\eqref{eq:abelian-torsion-1}--\eqref{eq:abelian-torsion-3},
the identities $Z_1=0$ and $Z_2=2\I\beta I_m$, and the
skew-symmetry of the torsion give
\[
T^p_{\alpha2}=\delta_{\alpha1}u_p,
\qquad
T^p_{q2}=-2\I\beta\delta_{pq},
\qquad
T^\alpha_{r2}=-(v_2^\alpha)_r.
\]
Furthermore, \eqref{eq:abelian-C-D-relations} and
\eqref{torsion} give $T^j_{pq}=0$ for every $j$ and
$p,q\geq3$.  Thus the summands with $a=1$ and $a\geq3$ vanish,
and only $a=2$ contributes.  Therefore
\begin{equation}\label{eq:abelian-three-W}
\begin{aligned}
0=4\widehat v_{\alpha\bar p q\bar r}
=2\I\beta\bigl[
&\delta_{\alpha1}
 (u_p\delta_{qr}+u_r\delta_{qp})+\delta_{pq}\overline{(v_2^\alpha)_r}
+\delta_{rq}\overline{(v_2^\alpha)_p}
\bigr].
\end{aligned}
\end{equation}

Fix $p\geq3$ and put $q=r=p$.  Taking $\alpha=2$ in
\eqref{eq:abelian-three-W} gives
$0=4\I\beta\overline{(v_2^2)_p}$, and hence
$(v_2^2)_p=0$.  Taking $\alpha=1$ gives
$0=4\I\beta\bigl(u_p+\overline{(v_2^1)_p}\bigr)$, so
$(v_2^1)_p=-\overline{u_p}$.  Since $p$ is arbitrary,
\begin{equation}\label{eq:abelian-v2-u}
v_2^2=0,
\qquad
v_2^1=-\overline u.
\end{equation}

For $p\geq3$,
Lemma~\ref{lem:curvature-support} gives
$\widehat R^c_{p\bar1p\bar1}=0$, because $\delta_{p1}=0$.
Thus \eqref{eq:master-equation} and $\chi>0$ imply
$\widehat v_{p\bar1p\bar1}=0$.

Taking $(i,j,k,\ell)=(p,1,p,1)$ in
\eqref{eq:vhat-definition}, the four summands coincide, and hence
\[
\widehat v_{p\bar1p\bar1}
=
\sum_aT^1_{pa}\overline{T^p_{1a}}.
\]
The term with $a=1$ vanishes because $T^p_{11}=0$, while the
terms with $a\geq3$ vanish because
$T^p_{1a}=(Z_1)_{ap}=0$.  Thus only $a=2$ contributes.  By
\eqref{eq:abelian-torsion-3}, the definition of $u$, and
skew-symmetry,
\[
\widehat v_{p\bar1p\bar1}
=
T^1_{p2}\overline{T^p_{12}}
=
-(v_2^1)_p\overline{u_p}.
\]
Using \eqref{eq:abelian-v2-u}, we obtain
$0=\overline{u_p}^{\,2}$, and hence $u_p=0$. Since $p\geq3$ is arbitrary, $u=0$.  Together with
\eqref{eq:abelian-v2-u}, this gives
\[
u=v_2^1=v_2^2=0.
\]

Finally, \eqref{eq:abelian-u} gives $\overline{v_1^2}=-2\zeta\overline{v_1^1}$.  Since
$\overline\zeta=-\zeta$, taking complex conjugates gives $v_1^2=2\zeta v_1^1$.  This complete the proof.
\end{proof}

For $\alpha\in\{1,2\}$, set
\begin{equation}\label{eq:abelian-negative-auxiliary}
\mu_\alpha:=T^\alpha_{12},
\qquad
B_0:=\frac{b\delta'}{\sqrt2},
\qquad
D_0:=\frac{d'+b\delta^2}{\sqrt2\,\delta'},
\qquad
\xi:=\frac{b\delta+\I a}{\sqrt2}.
\end{equation}

By \eqref{eq:abelian-C12}, \eqref{torsion}, and
\eqref{eq:E1-ab}--\eqref{eq:E2-ab}, one has
\begin{equation}\label{eq:abelian-mu-formulas}
\mu_1
=
\frac{\I(b-c)}{\sqrt2\,\delta'},
\qquad
\mu_2
=
-\frac{\delta(d'+b)+\I(a+c')}
{\sqrt2\,(\delta')^2}.
\end{equation}
Indeed,
$\mu_1=-C^1_{12}-(E_2)_{11}+(E_1)_{21}$ and
$\mu_2=-C^2_{12}-(E_2)_{12}+(E_1)_{22}$.

\begin{lemma}\label{lem:abelian-negative-Q}
Let $(\mathfrak g,J,g)$ be a unimodular Hermitian Lie algebra of
complex dimension $n\geq3$ containing an abelian ideal
$\mathfrak a$ of real codimension two.  Assume that
$J\mathfrak a\neq\mathfrak a$, that
$\mathfrak g/\mathfrak a$ is abelian, and that a canonical metric
connection $D_s^r$ with $\chi>0$ has constant holomorphic
sectional curvature $H^{D_s^r}=\ka<0$.  Put
$\mathcal W:=(\mathfrak a_J)^{1,0}$ and
$m:=\dim_{\C}\mathcal W=n-2$.  Use the relations
\[
Z_1=0,\qquad
Z_2=2\I\beta I_m,\qquad
V_1=(v_1^1,v_1^2),\qquad
v_1^2=2\zeta v_1^1,\qquad
V_2=0,
\qquad \zeta=\I\delta/\delta'
\]
from Lemmas~\ref{lem:abelian-negative-normal-form} and
\ref{lem:abelian-negative-reduction}.  Then the lower-right Chern
curvature blocks satisfy
\begin{equation}\label{eq:abelian-Chern-Q-negative}
\begin{aligned}
Q_{11}
&=
v_1^1(v_1^1)^*
+v_1^2(v_1^2)^*
-2B_0\beta I_m,\\
Q_{22}
&=-2D_0\beta I_m,\\
Q_{12}
&=-2\I\beta\xi I_m.
\end{aligned}
\end{equation}
Moreover, for $\alpha,\beta\in\{1,2\}$ and $p,q\geq3$, the
corresponding components of \eqref{eq:master-equation} yield the following matrix identities:
\begin{align}
Q_{11}
&=
\chi\left\{
v_1^1(v_1^1)^*
-\bigl(8\beta^2
+4\beta\operatorname{Im}\mu_1\bigr)I_m
\right\},
\label{eq:abelian-Q11-negative}\\
Q_{22}
&=
\chi\left\{
v_1^2(v_1^2)^*-4\beta^2I_m
\right\},
\label{eq:abelian-Q22-negative}\\
Q_{12}
&=
\chi\left\{
v_1^2(v_1^1)^*
+2\I\beta\mu_2I_m
\right\}.
\label{eq:abelian-Q12-negative}
\end{align}
\end{lemma}

\begin{proof}
Since $Z_1=0$, equation
\eqref{eq:abelian-Z-definitions} gives $Y_1^*=Y_1$.  By \eqref{eq:abelian-Y-commute} we have
$[Y_1,Y_2]=0$.
Since $Y_1^*=Y_1$, the identity $Z_2=2\I\beta I_m$ is equivalent
to
\[
Y_2^*=Y_2-2\I\beta I_m+2\zeta Y_1.
\]
Hence, $[Y_1,Y_2^*]=[Y_1,Y_2]-2\I\beta[Y_1,I_m]+2\zeta[Y_1,Y_1]
=0,\,[Y_2,Y_2^*]=-2\I\beta[Y_2,I_m]+2\zeta[Y_2,Y_1]=0.$ Thus
\[
[Y_1,Y_1^*]
=
[Y_1,Y_2^*]
=
[Y_2,Y_1^*]
=
[Y_2,Y_2^*]
=
0.
\]

By \eqref{eq:E1-ab}, \eqref{eq:E2-ab}, and
\eqref{eq:abelian-negative-auxiliary}, one has
\[
(E_1)_{11}=\xi,\qquad
(E_1)_{21}=\I B_0,\qquad
(E_2)_{22}=\I D_0,\qquad
(E_1)_{22}=-\xi,\qquad
\overline{(E_2)_{21}}=\xi.
\]

Taking $(\alpha,\beta)=(1,1)$ in \eqref{eq:Q}, and using
$Y_1^*=Y_1$ and $[Y_1,Y_1^*]=0$, we obtain
\begin{align*}
Q_{11}
&=
V_1V_1^*
-(E_1)_{11}Y_1-(E_1)_{21}Y_2^*
-\overline{(E_1)_{11}}Y_1
-\overline{(E_1)_{21}}Y_2\\
&=
V_1V_1^*
-(\xi+\overline\xi)Y_1
+\I B_0(Y_2-Y_2^*).
\end{align*}
Since $Z_2=2\I\beta I_m$ and $Y_1^*=Y_1$, equation
\eqref{eq:abelian-Z-definitions} gives
$Y_2-Y_2^*=2\I\beta I_m-2\zeta Y_1$.  Hence
\[
Q_{11}
=
V_1V_1^*-2B_0\beta I_m
+\bigl(-\xi-\overline\xi-2\I B_0\zeta\bigr)Y_1.
\]
By \eqref{eq:abelian-negative-auxiliary}, we have
\[
\xi+\overline\xi=\sqrt2\,b\delta,
\qquad
-2\I B_0\zeta=\sqrt2\,b\delta.
\]
Thus the coefficient of $Y_1$ vanishes.  Since
$V_1=(v_1^1,v_1^2)$, we conclude that
\[
Q_{11}
=
v_1^1(v_1^1)^*
+v_1^2(v_1^2)^*
-2B_0\beta I_m.
\]
Similarly, formula \eqref{eq:Q} gives
\[
\begin{aligned}
Q_{22}
={}&
-2\operatorname{Re}\bigl((E_2)_{12}\bigr)Y_1
-\I D_0Y_2^*+\I D_0Y_2\\
={}&
-2D_0\beta I_m
+\left(
-2\operatorname{Re}\bigl((E_2)_{12}\bigr)
-2\I D_0\zeta
\right)Y_1.
\end{aligned}
\]
By \eqref{eq:E2-ab},
$2\operatorname{Re}((E_2)_{12})
=2D_0\delta/\delta'$, whereas
$-2\I D_0\zeta=2D_0\delta/\delta'$.  Thus the $Y_1$-coefficient
vanishes, and $Q_{22}=-2D_0\beta I_m$.

For the mixed block, the relations
$V_2=0$ and $[Y_1,Y_2^*]=0$ yield
\[
\begin{aligned}
Q_{12}
={}&
-(E_1)_{12}Y_1
-\overline{(E_2)_{11}}Y_1
+\xi(Y_2^*-Y_2)\\
={}&
-2\I\beta\xi I_m
+\left(
-(E_1)_{12}
-\overline{(E_2)_{11}}
+2\zeta\xi
\right)Y_1.
\end{aligned}
\]
Directly from \eqref{eq:E1-ab} and \eqref{eq:E2-ab},
\[
(E_1)_{12}
+\overline{(E_2)_{11}}
=
2\zeta\xi.
\]
Hence the $Y_1$-term again vanishes, and
$Q_{12}=-2\I\beta\xi I_m$.  This proves
\eqref{eq:abelian-Chern-Q-negative}.

For $\gamma\in\{1,2\}$ and $p,q\geq3$, equations
\eqref{eq:abelian-torsion-1}--\eqref{eq:abelian-torsion-3}
and Lemma~\ref{lem:abelian-negative-reduction} give
\[
T^q_{1p}=0,\qquad
T^q_{2p}=2\I\beta\delta_{pq},\qquad
T^\gamma_{1p}=(v_1^\gamma)_p,\qquad
T^\gamma_{2p}=0,
\]
together with $T^\gamma_{12}=\mu_\gamma$ and
$T^p_{12}=0$.  Substitution in
\eqref{eq:vhat-definition} gives
\begin{equation}\label{eq:abelian-vhat-Q}
\begin{aligned}
4\widehat v_{\alpha\bar\beta p\bar q}
={}&
(v_1^\beta)_p\overline{(v_1^\alpha)_q}
+4\beta^2\delta_{\alpha2}\delta_{\beta2}\delta_{pq}+
2\I\beta
\bigl(
\delta_{\alpha1}\mu_\beta
-\delta_{\beta1}\overline{\mu_\alpha}
\bigr)\delta_{pq}.
\end{aligned}
\end{equation}
By the definition of $Q_{\alpha\beta}$,
$(Q_{\alpha\beta})_{pq}
=R^c_{\alpha\bar\beta p\bar q}$.  Moreover,
Lemma~\ref{lem:curvature-support} and
\eqref{eq:symmetrization} give
\[
R^c_{\alpha\bar\beta p\bar q}
=
4\widehat R^c_{\alpha\bar\beta p\bar q}.
\]
Thus the corresponding component of
\eqref{eq:master-equation} becomes
\[
(Q_{\alpha\beta})_{pq}
=
4\chi\widehat v_{\alpha\bar\beta p\bar q}
+2\ka\delta_{\alpha\beta}\delta_{pq}.
\]
Since $\ka=-4\chi\beta^2$, this is
\[
(Q_{\alpha\beta})_{pq}
=
4\chi\widehat v_{\alpha\bar\beta p\bar q}
-8\chi\beta^2\delta_{\alpha\beta}\delta_{pq}.
\]

Taking $(\alpha,\beta)=(1,1)$ in
\eqref{eq:abelian-vhat-Q}, and using
$\mu_1-\overline{\mu_1}
=2\I\operatorname{Im}\mu_1$, gives
\[
4\widehat v_{1\bar1p\bar q}
=
(v_1^1)_p\overline{(v_1^1)_q}
-4\beta\operatorname{Im}\mu_1\,\delta_{pq}.
\]
Combining the two equations above yields
\eqref{eq:abelian-Q11-negative}. The choices
$(\alpha,\beta)=(2,2)$ and $(1,2)$ similarly give
\eqref{eq:abelian-Q22-negative} and
\eqref{eq:abelian-Q12-negative}, respectively.
\end{proof}

\begin{lemma}\label{lem:abelian-m-one}
Under the hypotheses of
Lemma~\ref{lem:abelian-negative-Q}, assume that $n=3$.  Then
$v_1^1=v_1^2=0$.
\end{lemma}

\begin{proof}
Since $n=3$, one has $m=1$, so $v_1^1$, $v_1^2$, $Y_1$, and
$Y_2$ are complex numbers.  Since $H^{D_s^r}=\ka$ is constant,
\[
\widehat R^{D_s^r}_{1\bar2 1\bar3}
=
\frac{\ka}{2}
\bigl(\delta_{12}\delta_{13}+\delta_{13}\delta_{12}\bigr)
=
0.
\]
Moreover, Lemma~\ref{lem:abelian-negative-reduction} gives
$u=V_2=0$, while
Lemma~\ref{lem:abelian-negative-normal-form} gives $Z_1=0$.
Substitution in \eqref{eq:vhat-definition}, using
\eqref{eq:abelian-torsion-1}--\eqref{eq:abelian-u}, gives
$\widehat v_{1\bar2 1\bar3}=0$.  Hence the corresponding
component of \eqref{eq:master-equation} yields
$\widehat R^c_{1\bar2 1\bar3}=0$.

By \eqref{eq:curvature-matrix}, the upper-right block of
$\mathcal R_{1\bar2}$ is
\[
E_1V_2^*-V_2^*Y_1
-\sum_{\gamma=1}^2(E_1)_{\gamma2}V_\gamma^*.
\]
Indeed, the upper-right block of $[D_1,D_2^*]$ is
$E_1V_2^*-V_2^*Y_1$, while the upper-right block of
$-\sum_\gamma(E_1)_{\gamma2}D_\gamma^*$ is
$-\sum_\gamma(E_1)_{\gamma2}V_\gamma^*$; the remaining sum has
zero upper-right block.  Since $V_2=0$, this becomes
$-(E_1)_{12}V_1^*$.  When $m=1$,
$V_1^*=(\overline{v_1^1},\overline{v_1^2})^{\mathsf T}$, and
therefore
\[
R^c_{1\bar2 1\bar3}
=
-(E_1)_{12}\overline{v_1^1}.
\]
Thus $(E_1)_{12}\overline{v_1^1}=0$.

Suppose that $v_1^1\neq0$.  Then
$(E_1)_{12}=0$.  We consider separately the cases $b=0$ and
$b\neq0$.

\smallskip
\noindent
\emph{Case 1: $b=0$.}
The second relation in \eqref{eq:abelian-relations} gives $a=0$.
Since
$(E_1)_{12}=\I c/(\sqrt2\,\delta')$ by \eqref{eq:E1-ab}, we also have $c=0$.
Formula \eqref{eq:abelian-negative-normal-form} now gives
$\beta=-d'/(2\sqrt2\,\delta')$.  Since $\beta\neq0$, one has
$d'\neq0$.

With $a=b=c=0$, formula \eqref{eq:E2-ab} becomes
\[
E_2=
\begin{pmatrix}
0&
\dfrac{\I c'+\delta d'}{\sqrt2\,(\delta')^2}\\[2mm]
0&
\dfrac{\I d'}{\sqrt2\,\delta'}
\end{pmatrix}.
\]
Since $V_2=0$, equation \eqref{eq:abelian-block-Jacobi} becomes
$V_1E_2=Y_2V_1$. Here
$V_1=(v_1^1,2\zeta v_1^1)$.  The first component of this
row-vector identity gives $Y_2v_1^1=0$, and hence $Y_2=0$.
The second component then gives
$\I c'=\delta d'$.  Since $c',\delta,d'$ are real, both sides
must vanish; thus $c'=0$ and $\delta d'=0$.  Since $d'\neq0$, we get $\delta=0$, so $\zeta=i \frac{\delta}{\delta'}=0$.  Consequently,
$Z_2=Y_2-\overline{Y_2}+2\zeta Y_1=0$, contradicting
$Z_2=2\I\beta$ and $\beta\neq0$ by Lemma \ref{lem:abelian-negative-Q}.

\smallskip
\noindent
\emph{Case 2: $b\neq0$.}
By \eqref{eq:E1-ab}, the identity $(E_1)_{12}=0$ gives
$a\delta=0$ and $c=-b\delta^2$.

Suppose first that $\delta=0$.  Then $c=0$, and the first relation
in \eqref{eq:abelian-relations} gives $c'=0$.  Moreover,
$\zeta=0$ and $V_1=(v_1^1,0)$ by Lemma \ref{lem:abelian-negative-Q}.  By \eqref{eq:E2-ab}, the first
row of $E_2$ vanishes, so the first component of
$V_1E_2=Y_2V_1$ gives $Y_2=0$.  Hence
$Z_2=Y_2-\overline{Y_2}=0$, again contradicting
$Z_2=2\I\beta$ by Lemma \ref{lem:abelian-negative-Q}.

It remains to consider $\delta\neq0$.  Then $a=0$, and
\eqref{eq:abelian-relations} gives
$c'=0$ and $d'=c=-b\delta^2$.  In this case,
\[
(E_2)_{11}
=
-\frac{\sqrt2\,\I b\delta^2}{\delta'},
\qquad
(E_2)_{21}
=
\frac{b\delta}{\sqrt2},
\qquad
(E_2)_{12}=(E_2)_{22}=0.
\]
The first component of $V_1E_2=Y_2V_1$ is
\[
\bigl((E_2)_{11}
+2\zeta(E_2)_{21}-Y_2\bigr)v_1^1=0.
\]
Since
$(E_2)_{11}+2\zeta(E_2)_{21}=0$ and $v_1^1\neq0$, we obtain
$Y_2=0$.

Equations \eqref{eq:abelian-negative-normal-form},
\eqref{eq:abelian-negative-auxiliary}, and
\eqref{eq:abelian-mu-formulas} now give
\[
\beta=\frac{b\delta^2}{\sqrt2\,\delta'},
\qquad
\xi=\frac{b\delta}{\sqrt2},
\qquad
\mu_2=-\xi,
\qquad
D_0=0.
\]
Hence \eqref{eq:abelian-Chern-Q-negative} gives $Q_{22}=0$.
Comparing with \eqref{eq:abelian-Q22-negative}, and using
$v_1^2=2\zeta v_1^1$, yields
\[
|v_1^1|^2=\frac{b^2\delta^2}{2}.
\]
Thus
\[
v_1^2\overline{v_1^1}+2\I\beta\mu_2
=
2\zeta|v_1^1|^2+2\I\beta\mu_2
=
0.
\]
Thus \eqref{eq:abelian-Q12-negative} gives $Q_{12}=0$.
On the other hand,
\eqref{eq:abelian-Chern-Q-negative} gives
$Q_{12}=-2\I\beta\xi\neq0$, since
$b\delta\neq0$.  This is a contradiction.

Both cases are impossible.  Therefore $v_1^1=0$.  Since
$v_1^2=2\zeta v_1^1$, one also has $v_1^2=0$. This
completes the proof.
\end{proof}

\begin{lemma}\label{lem:abelian-scalar-system}
Under the hypotheses of
Lemma~\ref{lem:abelian-negative-Q}, put
$\mathcal W:=(\mathfrak a_J)^{1,0}$ and
$m:=\dim_{\C}\mathcal W=n-2$.  Then
\begin{equation}\label{eq:abelian-scalar-system}
B_0
=
2\chi\bigl(2\beta+\operatorname{Im}\mu_1\bigr),
\qquad
D_0=2\chi\beta,
\qquad
\xi=-\chi\mu_2.
\end{equation}
\end{lemma}

\begin{proof}
Assume first that $m\geq2$.  Since
$v_1^1\in\mathcal W\simeq\C^m$, there exists a unit vector
$w\in\mathcal W$ such that $(v_1^1)^*w=0$.  Since
$v_1^2=2\zeta v_1^1$, one also has $(v_1^2)^*w=0.$
Consequently,
\[
v_1^1(v_1^1)^*w
=
v_1^2(v_1^2)^*w
=
v_1^2(v_1^1)^*w
=
0.
\]
Comparing the two expressions for $Q_{11}w$ in
\eqref{eq:abelian-Chern-Q-negative} and
\eqref{eq:abelian-Q11-negative}, we obtain
\[
-2B_0\beta w
=
-\chi\bigl(
8\beta^2+4\beta\operatorname{Im}\mu_1
\bigr)w.
\]
Since $\beta\neq0$ and $w\neq0$, this gives
$B_0=2\chi(2\beta+\operatorname{Im}\mu_1)$.

Similarly, comparing the two expressions for $Q_{22}w$ gives
$-2D_0\beta w=-4\chi\beta^2w$, and hence
$D_0=2\chi\beta$.  Comparing the two expressions for $Q_{12}w$
gives
$-2\I\beta\xi w=2\I\chi\beta\mu_2w$, so
$\xi=-\chi\mu_2$.

Suppose now that $m=1$.  Lemma~\ref{lem:abelian-m-one} gives
$v_1^1=v_1^2=0$.  Hence all rank-one terms in
\eqref{eq:abelian-Chern-Q-negative} and
\eqref{eq:abelian-Q11-negative}--\eqref{eq:abelian-Q12-negative}
vanish.  The corresponding scalar identities are
\[
-2B_0\beta
=
-\chi\bigl(
8\beta^2+4\beta\operatorname{Im}\mu_1
\bigr),
\qquad
-2D_0\beta=-4\chi\beta^2,
\qquad -2\I\beta\xi=2\I\chi\beta\mu_2.
\]
Since $\beta\neq0$, these identities give
\eqref{eq:abelian-scalar-system}.
\end{proof}

\begin{proposition}\label{prop:abelian-negative-excluded}
Let $(\mathfrak g,J,g)$ be a unimodular Hermitian Lie algebra of
complex dimension $n\geq3$ containing an abelian ideal
$\mathfrak a$ of real codimension two.  Assume that
$J\mathfrak a\neq\mathfrak a$ and that
$\mathfrak g/\mathfrak a$ is abelian.  Then the holomorphic sectional curvature of a canonical metric connection $D_s^r$ with $\chi>0$ cannot be a negative constant.
\end{proposition}

\begin{proof}
Assume that
$H^{D_s^r}=\ka<0$.  By
Lemma~\ref{lem:abelian-scalar-system}, the identities
\eqref{eq:abelian-scalar-system} hold.

\smallskip
\noindent
\emph{Case 1: $b=0$.}
The second relation in \eqref{eq:abelian-relations} gives $a=0$,
while the third gives $c(d'-c)=0$.  By
\eqref{eq:abelian-negative-normal-form},
\eqref{eq:abelian-negative-auxiliary}, and
\eqref{eq:abelian-mu-formulas},
\[
\beta=-\frac{c+d'}{2\sqrt2\,m\delta'},
\qquad
B_0=0,
\qquad
\operatorname{Im}\mu_1
=-\frac{c}{\sqrt2\,\delta'}.
\]
The first equation in \eqref{eq:abelian-scalar-system} therefore
gives
\[
-\frac{c+d'}m-c=0.
\]
If $c\neq0$, then $c(d'-c)=0$ implies $d'=c$, and the above
identity becomes
$-(2/m+1)c=0$, a contradiction.  If $c=0$, the same identity
gives $d'=0$, and hence
$\beta=0$, contradicting
Lemma~\ref{lem:abelian-negative-normal-form}.  Thus $b=0$ is
impossible.

\smallskip
\noindent
\emph{Case 2: $b\neq0$.}
Put
\[
u_0:=\frac ab,
\qquad
\rho:=\frac{a^2+bc}{b^2},
\qquad
\eta:=\delta^2,
\qquad
L:=1-\eta=(\delta')^2>0.
\]
Then
\begin{equation}\label{eq:bne-parameters1}
a=bu_0,
\qquad
c=b(\rho-u_0^2),
\qquad
c'=-bu_0(\rho-u_0^2),
\qquad
d'=b(\rho+u_0^2).
\end{equation}
Together with \eqref{eq:abelian-negative-normal-form},
\eqref{eq:abelian-negative-auxiliary}, and
\eqref{eq:abelian-mu-formulas}, we have
\[
\beta=-\frac{b\rho}{\sqrt2\,m\delta'},
\qquad
D_0=
\frac{b(\rho+u_0^2+\eta)}
{\sqrt2\,\delta'},
\qquad
\operatorname{Im}\mu_1
=
\frac{b(1+u_0^2-\rho)}
{\sqrt2\,\delta'}.
\]
Since $\beta\neq0$, one has $\rho\neq0$.

Substitution in the first two equations of
\eqref{eq:abelian-scalar-system} gives
\begin{align}
&\rho+u_0^2+\eta
=-\frac{2\chi}{m}\rho,
\label{eq:abelian-neg-1}\\
&L
=
2\chi\left[
1+u_0^2-
\left(1+\frac2m\right)\rho
\right].
\label{eq:abelian-neg-2}
\end{align}
Indeed, \eqref{eq:abelian-neg-1} is the equation
$D_0=2\chi\beta$, while
\eqref{eq:abelian-neg-2} follows from
$B_0=2\chi(2\beta+\operatorname{Im}\mu_1)$ after multiplying by
$\sqrt2\,\delta'/b$.

Using \eqref{eq:abelian-mu-formulas} and
\eqref{eq:bne-parameters1}, we also obtain
\[
\mu_2
=
-\frac b{\sqrt2\,L}
\left\{
\delta(\rho+u_0^2+1)
+\I u_0(u_0^2+1-\rho)
\right\}.
\]
By \eqref{eq:abelian-negative-auxiliary} and \eqref{eq:bne-parameters1}, we have $\xi=b(\delta+\I u_0)/\sqrt2$. Taking real and imaginary parts in the third equation $\xi=-\chi\mu_2$ of \eqref{eq:abelian-scalar-system}, we obtain
\begin{equation}\label{eq:abelian-neg-3}
L\delta
=
\chi\delta(\rho+u_0^2+1),
\qquad
Lu_0
=
\chi u_0(u_0^2+1-\rho).
\end{equation}

Equation \eqref{eq:abelian-neg-3} can be written as
\[
\delta\left[
L-\chi(\rho+u_0^2+1)
\right]=0,
\qquad
u_0\left[
L-\chi(u_0^2+1-\rho)
\right]=0.
\]
Accordingly, whether the factors $\delta$ and $u_0$ vanish
determines which of these identities can be divided by them.
We therefore distinguish the four possible cases.

If $\delta u_0\neq0$, division in the two identities in
\eqref{eq:abelian-neg-3} gives
\[
\frac L\chi
=
\rho+u_0^2+1
=
u_0^2+1-\rho,
\]
and hence $\rho=0$, contradicting $\beta\neq0$.

If $\delta=0$ and $u_0\neq0$, then $L=1$, and \eqref{eq:abelian-neg-1} gives $u_0^2=-(1+2\chi/m)\rho$, so $\rho<0$.  On the other hand,
comparing \eqref{eq:abelian-neg-2} with the second identity in
\eqref{eq:abelian-neg-3} gives
$\rho=m/(4\chi)>0$, a contradiction.

If $\delta\neq0$ and $u_0=0$, comparison of
\eqref{eq:abelian-neg-2} with the first identity in
\eqref{eq:abelian-neg-3} gives
\[
\rho=\frac1{3+4/m}>0.
\]
But then the left-hand side of \eqref{eq:abelian-neg-1} is
$\rho+\eta>0$, whereas its right-hand side is negative, again a
contradiction.

Finally, if $\delta=u_0=0$, then
\eqref{eq:abelian-neg-1} gives
$\rho=-2\chi\rho/m$.  Since $\rho\neq0$, this implies
$1=-2\chi/m$, contradicting $\chi>0$.

All possibilities lead to contradictions.  Therefore no canonical
metric connection with $\chi>0$ can have negative constant
holomorphic sectional curvature in the abelian quotient case.
\end{proof}

\vspace{0.3cm}

\section{Proof of the main theorem}

\begin{proof}[Proof of Theorem~\ref{thm:main}]

Let $\mathfrak g$ be a unimodular Hermitian Lie algebra which contains an abelian ideal $\mathfrak a$ of real codimension $2$, and $J$ a complex structure on $\mathfrak g$ satisfying $J\mathfrak a\neq\mathfrak a.$

Suppose that $\chi=0$.  Since
$\chi=t^2+s^2/4$, one has $t=s=0$.  The identity
$t=(1-r+rs)/2$ then gives $r=1$, and hence
\[
(r,s)=(1,0),
\qquad
D_s^r=D_0^1=\nabla^c.
\]

 Suppose that the Chern holomorphic sectional curvature $H^c$ is constant $\ka$. If $\mathfrak g/\mathfrak a$ is non-abelian, then, by \cite{CaoZhengFV}, one of the generic, half-generic, or degenerate cases occurs. Propositions~\ref{prop:generic} and \ref{prop:half} rule out the first two cases, whereas Proposition~\ref{prop:degenerate} gives $\ka=0$ and $R^c=0$ in the third.

If $\mathfrak g/\mathfrak a$ is abelian, Proposition~\ref{prop:abelian-b0} gives $\ka=0$ and $R^c=0$ when $b=0$, while Proposition~\ref{prop:abelian-bnonzero} rules out the case $b\neq0$.

Assume now that $\chi>0$.  Suppose that
$\mathfrak g/\mathfrak a$ is non-abelian.
Proposition~\ref{prop:nonab-sign-reduction} gives $\ka\leq0$,
Proposition~\ref{prop:nonab-negative-excluded}
excludes $\ka<0$, while Proposition~\ref{prop:nonab-zero-canonical}
excludes $\ka=0$. Hence
$\mathfrak g/\mathfrak a$ must be abelian.

If $\mathfrak g/\mathfrak a$ is abelian, then Proposition
\ref{prop:abelian-sign-reduction} gives $\ka\leq0$, and
Proposition~\ref{prop:abelian-negative-excluded} excludes
$\ka<0$.  Thus $\ka=0$.  Proposition
\ref{prop:abelian-zero-canonical} shows that $g$ is K\"ahler flat.
All canonical metric connections coincide with the flat
Levi--Civita connection, and therefore $R^{D_s^r}=0$.
This completes the proof.
\end{proof}

For every $\chi>0$, constant holomorphic sectional curvature
forces the metric to be K\"ahler and flat.  Thus the curve
$\Gamma=\{\chi=1\}$ plays no exceptional role in the class
considered here.  At the Chern endpoint $\chi=0$, one obtains
Chern flatness instead.

\noindent\textbf{Acknowledgments.}
We would like to thank Kexiang Cao, Haojie Chen, and Xiaolan Nie for their interest and helpful discussions.

\vspace{0.3cm}

\noindent\textbf{Generative AI disclosure.}
During the development of this work, the authors used ChatGPT plus to assist with exploratory computations, possible proof directions, and checking algebraic reductions. The authors are responsible for the conceptualization and writing of the article, and take full responsibility for the correctness of the contents.

\vspace{0.3cm}

\noindent\textbf{Declaration on competing interests.}
All authors declare that there are no competing interests for this article.

\end{document}